\documentclass[12pt]{amsart}
\usepackage{a4wide}
\usepackage{times}
\usepackage{bbm}
\usepackage{mathtools, amssymb}
\usepackage{graphicx,xspace}
\usepackage{epsfig}
\usepackage{dsfont}
\usepackage[usenames,dvipsnames]{xcolor}
\usepackage{tikz}
\usepackage[T1]{fontenc}
\usepackage[utf8]{inputenc}
\usepackage{array,multirow} %
\usepackage{bm}
\usepackage{cancel}
\usepackage{tipa}
\usepackage{stmaryrd}
\usepackage{csquotes}

\usepackage{hyperref,pifont}
\usepackage{todonotes}
\usepackage{dsfont}

\newcommand{\ignorer}[1]{}

\usepackage[capitalize]{cleveref}

\theoremstyle{plain}
\def \inc{\text{inc}}
\def \next{\text{next}}
\def\setHH{\mathcal H^{g-1}_n}
\newtheorem{lemma}{Lemma}[section]
\newtheorem{theorem}[lemma]{Theorem}
\newtheorem{corollary}[lemma]{Corollary}
\newtheorem{proposition}[lemma]{Proposition}
\newtheorem{definition}[lemma]{Definition}

\theoremstyle{remark}
\newtheorem{remark}[lemma]{Remark}

\def\eps{\varepsilon}
\renewcommand\epsilon{\varepsilon}

\def\vvv{{v}}
\def\www{{w}}
\def\aaa{{a}}
\def\bbb{{b}}
\def\ccc{{c}}
\def\si{\sigma}
\def\ka{\kappa}
\def\core{\text{core}}
\def\anc{\text{anc}}
\def\root{\text{root}}
\def\tree{\text{tree}}   
\def\desc{\text{desc}}  

\def\R{\mathbb{R}}

\def\D{\mathbb{D}}
\def\P{\mathbb{P}}

\def\cEG{\mathcal{EG}}
\def\cP{\mathcal{P}}
\def\transpo{t}

\def\E{\mathbb{E}}
\def\O{\mathcal O}

\def\Z{\mathbb{Z}}
\def\bS{\mathbb{S}}
\def\bL{\mathbb{L}}
\def\kL{\mathfrak{L}}
\def\kT{\mathbb{T}}

\newcommand{\arc}[1]{{ %
  \setbox9=\hbox{#1}%
    \ooalign{\resizebox{\wd9}{\height}{\texttoptiebar{\phantom{A}}}\cr#1}}}

\DeclareMathOperator{\spec}{spec}
\DeclareMathOperator{\Tree}{Tree}
\DeclareMathOperator{\dt}{dt}
\DeclareMathOperator{\Poisson}{Poisson}

\DeclareMathOperator{\Arc}{Arc}

\DeclareMathOperator\Red{Red}
\def\N{N}
\def\SetSieve{\mathcal X}
\def\setF{\mathcal F^g_n}
\def\setFF{\mathcal F^{g-1}_n}
\def\setFZ{\mathcal F^{0}_n}
\def\setH{\mathcal H^g_n}

\def\Sieve{\mathbb X}

\def\Map{\mathsf{Map}}
\def\H{H}
\def\rH{{\bm H}^g_n}
\def\rF{{\bm F}^g_n}
\def\rFF{{\bm F}^{g-1}_n}
\def\rHH{{\bm H}^{g-1}_n}
\def\rFZ{{\bm F}^0_n}
\def\qrF{\tilde{\bm F}^g_n}

\def\cT{\mathcal{T}}
\def\cN{\mathcal{N}}
\def\kS{\mathfrak{S}}
\def\be{\mathbbm{e}}

\newcommand\restr[2]{{%
		\left.\kern-\nulldelimiterspace %
		#1 %
		\right|_{#2} %
	}}

\title[Random fixed genus factorizations]{Random generation and scaling limits\\
of fixed genus
factorizations into transpositions}

 \author[V. Féray]{Valentin Féray}
       \address[VF]{Université de Lorraine, CNRS, IECL, F-54000 Nancy, France}
       \email{valentin.feray@univ-lorraine.fr}

 \author[B. Louf]{Baptiste Louf}
 \address[BL]{Ångström Laboratory, Lägerhyddsvägen 1, Uppsala}
       \email{baptiste.louf@math.uu.se}

 \author[P. Thévenin]{Paul Thévenin}
 \address[PT]{Ångström Laboratory, Lägerhyddsvägen 1, Uppsala}
       \email{paul.thevenin@math.uu.se}
 


\begin{document}

\begin{abstract}
We study the asymptotic behaviour of random factorizations of the $n$-cycle into transpositions of fixed genus $g>0$. They have a geometric interpretation as branched covers of the sphere and their enumeration as Hurwitz numbers was extensively studied in algebraic combinatorics and enumerative geometry. 
On the probabilistic side, several models and properties of permutation factorizations
 were studied in previous works, in particular minimal factorizations of cycles
into transpositions (which corresponds to the case $g=0$ of this work).

Using the representation of factorizations via unicellular maps,
we first exhibit an algorithm which samples an asymptotically uniform factorization of genus $g$ in linear time. 
In a second time, we code a factorization as a process of chords appearing one by one in the unit disk, and we prove the convergence (as $n\to\infty$) of the process associated with a uniform genus $g$ factorization of the $n$-cycle.
The limit process can be explicitly constructed from a Brownian excursion.
Finally, we establish the convergence of a natural genus process, coding the appearance of the successive genera in the factorization.
\end{abstract}

\maketitle

\section{Introduction}

\subsection{Background and overview of the result}
In this paper, we are interested in properties of large random factorizations of the long cycle $(1\, 2\, \cdots\, n)$ into a given number of transpositions. Namely, if $\mathbb T_n$ is the set of transpositions in the symmetric group $S_n$,
we let, for $n \ge 1$ and $g \ge 0$,
\[ \setF = \big\{ (\transpo_1, \cdots, \transpo_{n-1+2g}) \in \mathbb T_n^{n-1+2g}: 
\transpo_1 \cdots \transpo_{n+1-2g}=(1\, 2\, \cdots\, n) \big\},\]
Our aim is to understand the asymptotic behaviour of an element of
 $\setF$ taken uniformly at random when $g>0$ is fixed and $n$ tends to infinity. 
\medskip

The cardinalities of the $\setF$ are examples of \emph{Hurwitz numbers}, which are important quantities in algebraic combinatorics.
 They were first introduced to enumerate ramified covering of the sphere by surfaces (see \cite{LZ04} and references therein).
 Later connections  with other objects and concepts have been discovered,
such as the moduli space of curves via the ELSV formula \cite{ELSV}, 
integrable hierarchies \cite{Oko00} or the topological recursion (see \cite{ACEH16} and references therein). In the particular case of factorizations of the long cycle, the enumeration in the genus $0$ case (also called \emph{minimal factorizations}) goes back to Dénes \cite{Den59}, who gave the exact formula
$|\setFZ|=n^{n-2}$.
Later on, the enumeration in greater genus was treated \cite{Jac88, SSV97}, and then broadly generalized in \cite{poulalhon2002factorizations}, using representation theory.
All these works inspired similar investigations in the more general setting of Coxeter groups \cite{Bes15,CS14}.

We insist on the following fact. 
Although the formula $|\setFZ|=n^{n-2}$ in genus $0$ has been explained
through several bijections \cite{Mos89,GY02,Biane2004}, 
in higher genus, there is no known combinatorial way to enumerate $\setF$.
In particular, it is not clear how to generate efficiently a uniform random element in $\setF$.
Our first contribution is to give a simple algorithm,
which generates an asymptotically uniform random element in $\setF$
(see \cref{ssec:generation} below).
\medskip

Recently, permutation factorizations have also been studied from a probabilistic point of view. 
We first mention the literature of random sorting networks
(see \cite{AHRV07,dauvergne2018archimedean} among other articles on the subject),
which are factorizations into adjacent transpositions -- i.e.~transpositions switching two consecutive integers.
Closer to the present paper, Kortchemski and the first author studied 
random uniform elements of $\setFZ$
(i.e. in the case $g=0$), both from the global \cite{MinFacto} and local \cite{MinFactoTrajectories} 
 points of view. 
It is rather easy to see that the local behaviour will be the same for any {\em fixed} genus $g$,
so that we focus here on the global limit, i.e. on scaling limits.

To make sense of the notion of scaling limits of factorizations,
the following approach was proposed in  \cite{MinFacto}.
 We see a transposition $(a\, b) \in \mathbb T_n$
as a chord $[\exp(-2\pi i \, a/n),\exp(-2\pi i\, b/n)]$ in the disk.
A factorization in $\setF$ is then encoded as a process of chords 
appearing one after the other;
the time in the process corresponds to the index of the transposition in the factorization.
In \cite{MinFacto}, the authors show that, for $g=0$, a phase transition
 occurs in the process of chords at time  $\Theta(\sqrt{n})$.
 More precisely, take a uniform element of $\setFZ$
 and its associated chord process.
 It is shown that, for any $c \geq 0$,  this chord process taken at time $\lfloor c \sqrt{n} \rfloor$ converges towards a random compact subset of the disk, which can be defined from a certain Lévy process. Later, the third author extended this work by showing that this convergence holds jointly for all $c \geq 0$ \cite{thevenin2019geometric}, and that the limiting process can be constructed directly from a Brownian excursion.

In the present work, we extend these results to the case of genus greater than $0$. 
Using our new generation algorithm,
we show the convergence of the chord process associated with a uniform random element in $\setF$. The phase transition still occurs at a time window $\Theta(\sqrt{n})$,
and we give a simple explicit construction of the limit process (\cref{ssec:intro_scaling}). 
\medskip

Our third main result (\cref{ssec:intro_genus}) describes the evolution of the genus 
in a uniform random factorization.
Again, we see a phase transition at a time window $\Theta(\sqrt{n})$.
Before this time window, the prefix of a uniform random factorization of genus $g$ - that is, the first $k$ transpositions for $k=o(\sqrt{n})$ -
is with high probability also the prefix of a minimal factorization.
With high probability this is not the case after the time window $\Theta(\sqrt{n})$.
The study of the genus process in the critical time window $\Theta(\sqrt{n})$
 turns out to be related with looking at reduced trees in conditioned Galton-Watson trees,
as done in the seminal work of Aldous \cite{aldous1993}.

\bigskip

{\em Convention}: we compose permutations from left to right, so that $\sigma \tau$ corresponds to the composition $\tau \circ \sigma$ ($\sigma$ is applied first and then $\tau$). We will use the notation $[a,b]$ to denote either the real interval or the set of integers between $a$ and $b$. This will always be made clear by the context.

Finally, given two sequences of random variables $(U_n)_{n \geq 1}, (V_n)_{n \geq 1}$ 
on the same probability space, we say that $U_n$ is dominated by $V_n$ in probability and 
write $U_n=\O_P(V_n)$ if, for every $\epsilon>0$, there exists $C>0$ such that $\limsup_{n \rightarrow \infty} \P(U_n \leq C V_n) \geq 1-\epsilon$.

\medskip

\subsection{Random generation of uniform genus $g$ factorizations}
\label{ssec:generation}
As said above, our first main contribution is a linear time algorithm that generates 
an random element $\qrF$ in $\setF$ which is {\em asymptotically uniform},
in the following sense: if $\rF$ is a uniform random element of $\setF$
and if $d_{TV}$ denotes the total variation distance, then one has
\[ \lim_{n\to \infty} d_{TV}(\qrF,\rF) =0.\]

Our algorithm works recursively on the genus $g$.
We start with a (random) factorization $F=(t_1,\dots,t_{n-1+2(g-1)})$ of genus $g-1$
and define a random factorization $\Lambda(F)$ as follows.
First, pick $\vvv < \www$ uniformly in $[1,n+1-2g]$, and $\aaa<\bbb<\ccc$ uniformly in $[1,n]$, independently of $F$. 

\begin{enumerate}
\item Then we define
\begin{align*}
\overline F_1 =(t_1,t_2,\dots,t_{\vvv-1},(\aaa \,\ccc),t_\vvv,\dots,t_{\www-2},(\aaa \, \bbb),t_{\www-1},\dots,t_{n-1+2(g-1)});\\
\overline F_2 =(t_1,t_2,\dots,t_{\vvv-1},(\aaa \,\bbb),t_\vvv,\dots,t_{\www-2},(\aaa \,\ccc),t_{\www-1},\dots,t_{n-1+2(g-1)}).
\end{align*}
If the product of all transpositions in $\overline F_1$ is a single cycle of length $n$, 
 (possibly distinct from $(1 \, \cdots \, n)$), then we set $\overline F=\overline F_1$ and denote this long cycle by $\zeta$.
Otherwise we consider $\overline F_2$: if the product of all transpositions in $\overline F_2$
 is a single cycle of length $n$, then we set $\overline F=\overline F_2$ and  denote this long cycle by $\zeta$.
If none of these products is a long cycle, we set $\overline F=\dag$, which is a formal symbol
to represent that the algorithm "fails".\footnote{It turns out that,
if $F$ is taken uniformly at random in $\setFF$, with high probability,
exactly one of the lists $\overline F_1$ and $\overline F_2$ is a factorization of a long cycle,
the other being a factorization of a product of three disjoint cycles; see \cref{lem_bad_phi}.}
\item If $\overline F \neq \dag$, let $\si$ be the unique permutation such that $\sigma(1)=1$
and $\si^{-1} \zeta \si =(1\, \cdots\, n)$. Then we set $\Lambda(F)=\sigma(\overline F)$,
where $\sigma$ acts by conjugation on each transposition in $\overline F$;
i.e.~we replace each
transposition $t=(i \,j)$ by $\si^{-1} t \si =(\si(i) \, \si(j))$.

One checks immediately that $\Lambda(F)$ is a factorization of genus $g$ of $(1\, \cdots\, n)$.
If  $\overline F=\dag$, we simply set $\Lambda(F)=\dag$.
\end{enumerate}
Note that $\Lambda$ is a random function. Formally, for any $g\ge 1$, it maps random variables
on $\setFF \uplus \{\dag\}$ to random variables on $\setF \uplus \{\dag\}$,
with the convention $\Lambda(\dag)=\dag$.

\begin{theorem}
\label{thm:RandomGenerationFacto}
Let $\rFF$ and $\rF$ be uniform random factorizations respectively in $\setFF$ and $\setF$, respectively.
Then \[ \lim_{n\to \infty} d_{TV}\big[\Lambda(\rFF),\rF\big] =0.\]
\end{theorem}

\begin{corollary}
\label{corol:RandomGeneration}
Let $\rFZ$ be a uniform random {\em minimal} factorization of $(1\, \cdots\, n)$, 
i.e. a uniform random element in $\setFZ$. Then
\[ \lim_{n\to \infty} d_{TV}\big[ \Lambda^g(\rFZ), \rF \big] =0.\]
\end{corollary}

This result gives a linear time algorithm to generate asymptotically uniform elements of $\setF$. Namely, one samples first a uniform element of $\mathcal{F}^0_n$ in linear time in $n$ using the encoding of minimal factorizations by Cayley trees \cite{Den59,Mos89} and 
Prüfer code for Cayley trees \cite{Prufer}. 
Then we apply $\Lambda^g$. Each of the steps of $\Lambda$ has a linear time-complexity\footnote{We consider that basic operations, such as taking an integer uniformly at random between $1$ and $n$,
adding an element or swapping elements in a list, are done in constant time.},  which makes our algorithm linear. Corollary~\ref{corol:RandomGeneration} ensures that the probability that the algorithm fails is only $o(1)$ and that the output is closed
in total variation to a uniform random element in $\setF$.
\medskip

While the algorithm is easily described uniquely in terms of factorizations,
the proof of Theorem~\ref{thm:RandomGenerationFacto}
heavily relies on an encoding of factorizations as unicellular maps.
Indeed, the elements of $\setF$
are known to be in bijection with a family of \emph{unicellular maps} 
with labeled edges and a consistency condition,
 which we call {\em Hurwitz condition} (see Section~\ref{sec:facto_maps} for a definition). 
 The main step in the proof of Theorem~\ref{thm:RandomGenerationFacto}
 is the construction of an "asymptotic bijection" on these sets of maps,
  which translates to the operation $\Lambda$ in terms of factorizations. 
Our bijection is inspired, but different, from Chapuy's bijection on unicellular maps \cite{Cha10}
(the latter is not compatible with the Hurwitz condition). 
Proving that it is indeed a bijection between sets of asymptotically total cardinalities
requires a careful analysis of random maps with the Hurwitz condition,
mixing tools of analytic combinatorics and results of random tree theory.
\medskip

Finally we mention that the permutation $\si$ appearing in step (ii) of the above algorithm
has a particular form, see \cref{prop:labelingalgorithm} below. 
Informally, it is close to the permutation which swaps the integer intervals $[a\!+\!1,b]$ and $[b\!+\!1,c]$
(keeping elements of the same interval in the same order).
This is a key point to address scaling limit problems, as done in the next section.
\medskip

\begin{remark}
To be complete on the comparison with existing methods in the literature,
let us mention that an asymptotically uniform sampling algorithm
for $\setFF$ can also be derived from \cite{chapuy2011new,chapuy2013simple}. 
This algorithm is also linear in time, but contrary to ours,
 its probability to succeed is bounded away from $1$ (and exponentially decreasing with the genus).
 This higher rejection probability makes the algorithm less convenient to study asymptotic properties of uniform random elements in $\setFF$.
 
For specialists, here is a brief description of this sampling algorithm: take a uniform Cayley tree, pick $g$ triples of vertices uniformly, and merge them to get a unicellular map (as prescribed in \cite{chapuy2011new}); if the resulting map satisfies the Hurwitz condition, we return the associated factorization, otherwise we repeat the operation.
We can also modify it into an exact uniform sampling algorithm by choosing $2h\!+\!1$-tuples of vertices (instead of triples) with appropriate probabilities.
\end{remark}

\subsection{Sieves and their scaling limits}
\label{ssec:intro_scaling}
~
\bigskip

As said above, we identify a transposition $\transpo=(a\, b)$ in $\mathbb T_n$
with a {\em chord} in the unit disk, given by the line segment $[\exp(-2\pi i \, a/n),\exp(-2\pi i\, b/n)]$.

In the literature, {\em noncrossing} sets of chords are called {\em laminations}.
Here, we shall consider sets of chords which are potentially crossing and call such objects {\em sieves}.

More precisely, a sieve of the unit disk $\D$ is a compact subset of $\D$ made of the union of the unit circle $\bS^1$ and a set of chords. The set of all sieves of $\D$ is denoted by $\SetSieve(\D)$.
If, furthermore, the chords of a sieve do not cross except possibly at their endpoints, then it is called a lamination. The set of all laminations of $\D$ is denoted by $\bL(\D)$.

With a factorization $F=(\transpo_1,\dots,\transpo_m)$ in $\setF$, 
we associate a {\em sieve} denoted $\Sieve(F)$:
by definition, $\Sieve(F)$ is simply the union of the $m$ chords 
corresponding to the transpositions in $F$.
We note that, unlike in the case $g=0$, these chords might be crossing,
and hence $\Sieve(F)$ is not a lamination in general. 

The problem we address here is the following.
Fix some genus $g>0$. 
For each $n$, we let $\bm F^g_n$ be a uniform random element in $\setF$.
We are looking for the limit of $\Sieve(\bm F^g_n)$ 
for the Hausdorff topology on compact subspaces of 
the closed unit disk $\D$.
A simulation for $g=1$ and $n=1000$ is given on \cref{fig:genre1}.
This simulation has been done using the algorithm described
in the previous section.

\begin{figure}[t]
\center
\includegraphics[scale=.3]{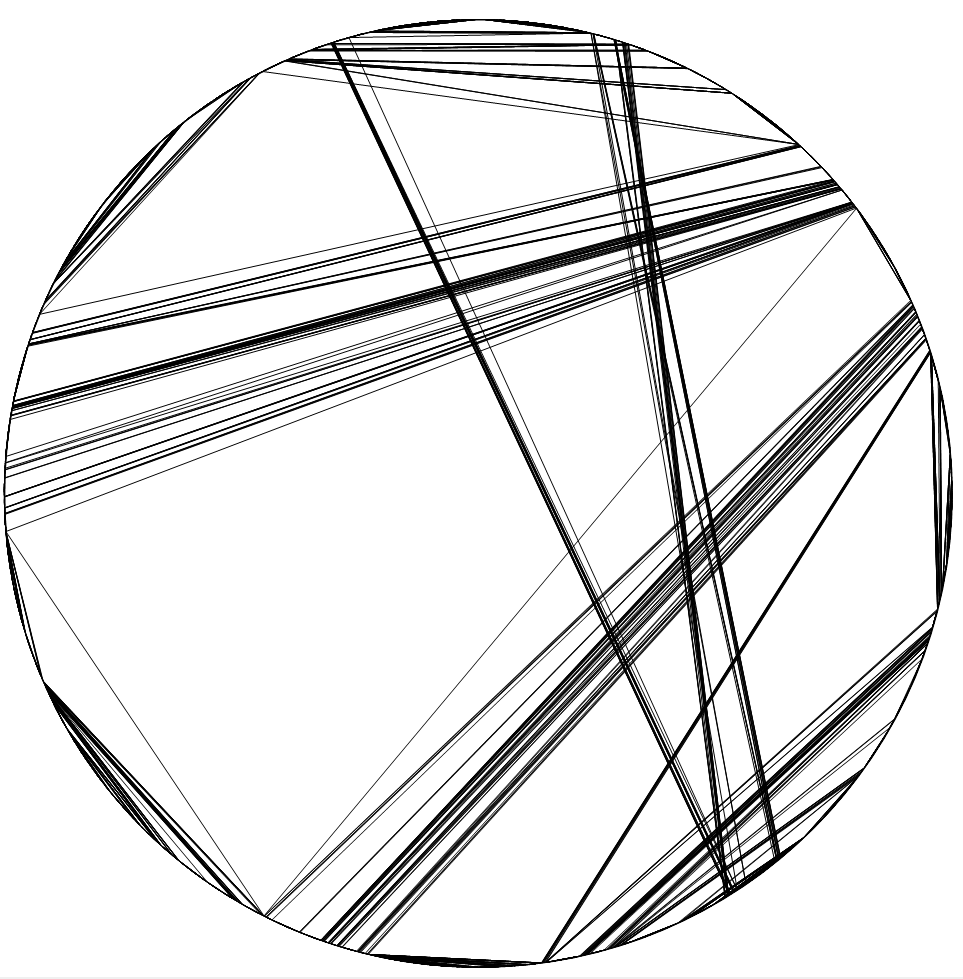}
\caption{A realization of $\Sieve(\bm F^1_{1000})$}
\label{fig:genre1}
\end{figure}

For $g=0$,
the problem was solved in \cite{MinFacto,thevenin2019geometric}:
the limit $\bm\Sieve^0_\infty$ of the sieve $\Sieve(\bm F^0_n)$ is 
the so-called Aldous' Brownian triangulation, also denoted $\mathbb L_\infty^{(2)}$
(see Section \ref{sec:background} for a definition).
The latter has been introduced by Aldous in \cite{Ald94b}
It has been proved to be the limit of various models of random laminations
\cite{Ald94b,Kor14,CK14,KM17}.
To define the limit  $\bm\Sieve^g_\infty$ for a general genus $g$,
we need to introduce first some notation.

If $A$ and $B$ are points on the circle,
we write $\arc{AB}$ for the arc going from $A$ to $B$ in counterclockwise order.
By convention it contains $A$ but not $B$.
Let $\Sieve$ be a sieve
and $\{A,B,C\}$ three points on the unit circle,
named such that  they appear in the order ($A$, $B$, $C$)
when turning counterclockwise around the circle starting at a reference point, say $(1,0)$.
Let $D$ be the point of the arc $\arc{AC}$ such that the arcs $\arc{AD}$ and $\arc{BC}$
have the same length.
We consider the piecewise rotation $\Psi$ of the circle sending $\arc{AB}$ to $\arc{DC}$,
$\arc{BC}$ to $\arc{AD}$ and fixing $\arc{CA}$.
Finally, we define $R_{\{A,B,C\}}(\Sieve)$ as the sieve obtained by replacing each chord $[X,Y]$ of $\Sieve$ by $[\Psi(X),\Psi(Y)]$. See an illustration of this operation on Fig. \ref{fig:rotation}. Now, for any $g \in \Z_+$, define an operation $R^g$ on the set $\SetSieve(\D)$ of sieves the following way: let $({A_i, B_i, C_i})_{i \geq 1}$ be i.i.d. uniform triples of points on $\bS^1$. Then, define $R^0$ as the identity on $\SetSieve(\D)$ and, for any $g \geq 1$, any $\Sieve \in \SetSieve(\D)$, $R^g(\Sieve) = R_{{A_g,B_g,C_g}}(R^{g-1}(\Sieve))$. 

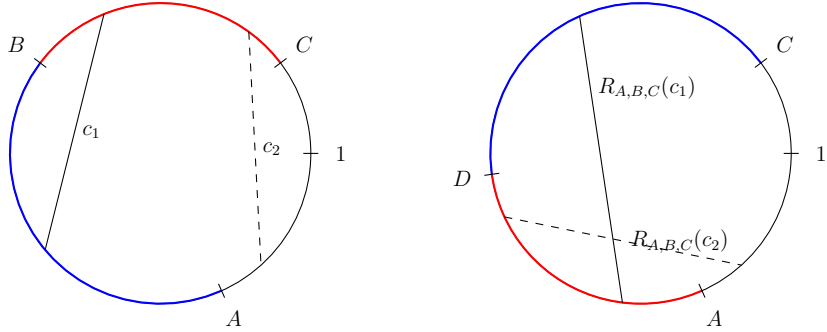
\begin{figure}
\begin{tabular}{c c c}
\begin{tikzpicture}[scale=2,every node/.style={scale=0.7}]
\draw (0,0) circle (1);
\draw[red,thick,domain=37:143] plot ({cos(\x)}, {sin(\x)});
\draw[blue,thick,domain=143:294] plot ({cos(\x)}, {sin(\x)});
\draw ({1.2*cos(37)}, {1.2*sin(37)}) node{$C$};
\draw ({1.2*cos(294)}, {1.2*sin(294)}) node{$A$};
\draw ({1.2*cos(143)}, {1.2*sin(143)}) node{$B$};
\draw (1.2,0) node{$1$};
\draw ({1.05*cos(37)}, {1.05*sin(37)}) -- ({.95*cos(37)}, {.95*sin(37)}) ({1.05*cos(143)}, {1.05*sin(143)})-- ({.95*cos(143)}, {.95*sin(143)}) ({1.05*cos(294)}, {1.05*sin(294)}) -- ({.95*cos(294)}, {.95*sin(294)}) (.95,0)--(1.05,0);
\draw[dashed] ({cos(54)}, {sin(54)}) -- ({cos(-48)}, {sin(-48)}) node[pos=0.5,right]{$c_2$};
\draw ({cos(112)}, {sin(112)}) -- ({cos(220)}, {sin(220)}) node[pos=0.5,right]{$c_1$};
\end{tikzpicture}
&
\begin{tikzpicture}
\draw[white] (0,0) -- (.5,0);
\end{tikzpicture}
&
\begin{tikzpicture}[scale=2,every node/.style={scale=0.7}]
\draw (0,0) circle (1);
\draw[blue,thick,domain=37:188] plot ({cos(\x)}, {sin(\x)});
\draw[red,thick,domain=188:294] plot ({cos(\x)}, {sin(\x)});
\draw ({1.2*cos(37)}, {1.2*sin(37)}) node{$C$};
\draw ({1.2*cos(294)}, {1.2*sin(294)}) node{$A$};
\draw ({1.2*cos(188)}, {1.2*sin(188)}) node{$D$};
\draw (1.2,0) node{$1$};
\draw ({1.05*cos(37)}, {1.05*sin(37)}) -- ({.95*cos(37)}, {.95*sin(37)}) ({1.05*cos(188)}, {1.05*sin(188)})-- ({.95*cos(188)}, {.95*sin(188)}) ({1.05*cos(294)}, {1.05*sin(294)}) -- ({.95*cos(294)}, {.95*sin(294)}) (.95,0)--(1.05,0);
\draw[dashed] ({cos(54+151)}, {sin(54+151)}) -- ({cos(-48)}, {sin(-48)}) node[pos=0.5,right]{$R_{A,B,C}(c_2)$};
\draw ({cos(112+151)}, {sin(112+151)}) -- ({cos(220-106)}, {sin(220-106)}) node[pos=0.75, right]{$R_{A,B,C}(c_1)$};
\end{tikzpicture}
\end{tabular}

\caption{The piecewise rotation $\Psi$. 
Note that the chords $R_{A,B,C}(c_1)$ and $R_{A,B,C}(c_2)$ cross each other, while $c_1$ and $c_2$ do not.}
\label{fig:rotation}
\end{figure}

Using this notation, we define 
$\Sieve^g_\infty = \overline{R^g(\Sieve^0_\infty)}$.
We then have the following result
-- we recall that we consider the Hausdorff topology on compact subsets of the unit disk $\D$.
\begin{theorem}
  \label{thm:main}
The sieve $\Sieve(\bm F^g_n)$ converges to $\Sieve^g_\infty$ in distribution.
\end{theorem}

One can also prove a functional version of the convergence of Theorem \ref{thm:main}. 
In this version, we associate with a factorization $F=(\transpo_1,\dots,\transpo_m)$ 
an increasing list of sieves $(\Sieve_k(F))_{0 \le k \le n}$,
where $\Sieve_k(F)$ is the union of the unit circle 
and of the chords corresponding to the transpositions $\transpo_1, \dots, \transpo_k$.
Our goal is to describe the limit of the process 
$(\Sieve_k(\bm F^g_n))_{0 \le k \le n}$ with a suitable renormalization of the time parameter $k$.

In genus $0$,  the following was proved in \cite{MinFacto, thevenin2019geometric}:
the time-rescaled process $\Sieve_{(c/\sqrt{2}) \sqrt{n}}$ converges to a process $\left(\bL_c^{(2)}\right)_{c \geq 0}$.
This limiting process interpolates between the circle (for $c=0$) and the Brownian triangulation  $\bL_\infty^{(2)}=\Sieve_\infty^0$ (for $c \rightarrow \infty$).
It can be constructed starting from a Brownian excursion 
and an associated Poisson point process 
(see Section \ref{sec:background} and \cite{thevenin2019geometric} for details).
We will show that, for any fixed $g \geq 0$, the limit of the process in genus $g$ is simply the process $\left(\bL_c^{(2)}\right)_{c \geq 0}$, to which we apply the operation $R^g$ (taking the same triples of points for all $c\geq 0$).
\medskip

To state this process convergence formally, we recall that,
given a subset $E \subseteq \R_+$ and a metric space $F$, 
$\mathcal D(E,F)$ denotes the set of càdlàg functions from $E$ to $F$ 
(that is, right-continuous with left limits on $E$). 
This set $\mathcal D(E,F)$ can be naturally endowed with the Skorokhod $J1$ topology
 (we refer to Annex $A2$ in \cite{Kal02} for  definitions).

\begin{theorem}
\label{thm:mainprocess}
Let $g \geq 0$. Then we have the following convergence in distribution
in $\mathcal D(\R_+, \SetSieve(\D))$, jointly with the convergence of Theorem \ref{thm:main}:
\[\left(\Sieve_{(c/\sqrt{2}) \sqrt{n}}(\bm F^g_n)\right)_{c \geq 0} \ \longrightarrow \left( R^g\left(\bL_c^{(2)}\right) \right)_{c \geq 0}.\]
\end{theorem}

Let us discuss briefly the proof of this theorem.
We use the algorithm defined in Section \ref{ssec:generation}
and identify $\bm F^g_n$ with $\Lambda^g(\bm F^0_n)$ 
on a set of probability tending to $1$.
Since $(\Sieve_{(c/\sqrt{2}) \sqrt{n}}(\bm F^0_n))_{c \geq 0}$ is known to tend to $(\bL_c^{(2)})_{c \geq 0}$ (this is the case $g=0$ of the theorem,
already proved in \cite{thevenin2019geometric}),
we need to understand how applying $\Lambda$ to a factorization modifies 
its associated sieve process.
The transpositions $(a\, b)$ and $(a\, c)$ added at step (i) of the construction of $\Lambda$ are added with high probability at times larger than $\Theta(\sqrt n)$ and thus are not visible
on the sieve process.
However, the conjugation by $\sigma$ at step (ii) is visible, and 
it turns out to act asymptotically as a rotation operation $R_{\{A,B,C\}}$
with respect to three uniform random points.

A technical difficulty comes from the fact that the rotation operations $R_{\{A,B,C\}}$ are not continuous, which prevents us from directly using the result in genus $0$. Indeed, small chords may be sent to large chords by the rotation operation, and conversely. Therefore, we need to prove that these noncontinuous phenomena asymptotically do not affect the sieve-valued process that we consider.
\bigskip

We end this section with a side result on the limiting objects $\bm \Sieve_\infty^g$.
Intuitively, successive rotations applied to the Brownian triangulation tend to add more chords inside the disk; therefore we expect $\bm \Sieve_\infty^g$
to be more and  more "dense" as $g$ grows.
 This is made rigorous in the following statement:

\begin{proposition}
\label{thm:cvginfty}
Almost surely, for the Hausdorff distance:
\begin{align*}
\bm \Sieve_\infty^g \underset{g \rightarrow \infty}{\rightarrow} \D,
\end{align*}
where $\D$ denotes the whole unit disk.
\end{proposition}

This proposition follows from the construction of $\bm \Sieve_\infty^g$
and basic properties of the Brownian excursion; 
we prove it in \cref{ssec:ProofGToInfty}.

\subsection{The genus process}
\label{ssec:intro_genus}
A natural question on factorizations of a given genus $g$ is the following: 
when does the genus appear? More precisely,
as the transpositions of a factorization are read and the corresponding chords drawn in the disk,
when do we know that we are not considering a factorization of genus $h$ for $h < g$?

To study this question, we introduce the notion of genus of a list of transpositions.
For a factorization $F=(t_1, \ldots, t_{n-1+2g})$ of $(1\, \cdots\, n)$, we denote
by $f_k$ its prefix of length $k$, i.e. $f_k=(t_1, \ldots, t_k)$.
We say that $F$ {\em extends} $f_k$.
\begin{definition}
\label{defi:genus}
Fix $n,k \geq 1$, and let $f := (\transpo_1, \ldots, \transpo_k) \in \kT_n^k$, where $\kT_n$ is the set of transpositions on $[ 1,n ]$. The genus $G(f)$ of $f$ is defined as the minimum genus $g$ of a factorization $F$ of $(1\, \dots\, n)$ extending $f$.
\end{definition}

\begin{theorem}
\label{thm:asymptoticgenus}
Let $g \in \Z_+$ be fixed. Then the genus process 
$\big( G \big((\bm F_n^g)_{c\sqrt n} \big)\big)_{c\ge 0}$
converges to a random process $\big(G^g_c\big)_{c\ge 0}$
in the Skorokhod space $\mathcal D(\R,\{0,\dots,g\})$.
Moreover, almost surely, the limiting process $G^g_c$ converges to $0$ when $c$ tends to $0$ and to $g$ when $c$ tends to $+\infty$, and has jumps of size exactly $1$.
\end{theorem}

In other terms, the genus of a uniform random partial factorization of genus $g$ goes from $0$ to $g$ in the time window $\Theta(\sqrt{n})$. A formula for the limiting process for $g=1$ is given in \cref{prop:Genus1}.

The main idea in the proof of Theorem \ref{thm:asymptoticgenus} is to consider the representation of a list of transpositions as a sieve, and to focus on chords that cross each other. Indeed, only such chords may increase the genus of a partial factorization;
see Section~\ref{ssec:Red_facto}. 

The algorithm defined in Section \ref{ssec:generation} allows us to obtain $\bm F_n^g$ from a uniform random factorization $\rFZ$ of genus $0$ and the $3g$ rotation points $(A_i,B_i,C_i)_{i\le g}$. 
Furthermore, as already said, the minimal factorization $\rFZ$ can be coded by a tree. 
Summing up, the factorization $\bm F_n^g$ is obtained from a set $E$ of $3g$ uniform points on a random tree $T$. 
It turns the crossing chords of $\bm F_n^g$ (which are the ones explaining the genus process) correspond in some sense to edges of the \enquote{reduced tree of $T$ with respect to the vertices in $E$}.
The latter notion has been introduced by Aldous \cite{aldous1993} 
and is central in his theory of tree limits.
Here, we use one of his result to prove
that the genus of a prefix of a factorization of genus $g$
typically increases from $0$ to $g$ in the time window $\Theta(\sqrt{n})$.

\subsection{Outline of the paper}
In section \ref{sec:facto_maps}, we present a previously known encoding of factorizations
by edge-labelled unicellular maps.
Based on this, we prove \cref{thm:RandomGenerationFacto} in \cref{sec:asympbij}.

\cref{sec:background,sec:rotations} are a preparation for the proof
of our scaling limit results; we give respectively some background on trees and laminations
and some first results on sieve rotations.
We then prove \cref{thm:main,thm:mainprocess} in \cref{sec:proofs_Scaling}.

Finally, \cref{sec:genus_process} is devoted to the convergence of the genus process
(i.e.~to the proof of \cref{thm:asymptoticgenus}).\medskip

\section{Factorizations and unicellular maps}
\label{sec:facto_maps}
In this section, we explain how to encode factorizations of the long cycle as unicellular maps,
see \cite{Irv04}.
This encoding is a variant of the encoding of factorizations as constellations \cite{LZ04}
or of that of minimal factorizations as trees \cite{GY02}.
To make the article self-contained, we provide full proofs of our claims.
\medskip

We recall that a map is a {\em cellular embedding} of a graph into a 2-dimensional compact surface
without border; here {\em embedding} means that vertices are distinct, and edges do not cross,
except at vertices, while {\em cellular} means that the complement of the graph in the surface is homeomorphic 
to a disjoint union of open disks.
Maps are considered up to orientation-preserving homeomorphism,
and the genus of a map is the genus of the underlying surface.
It is well-known that maps can be alternatively described by combinatorial data,
namely as a connected graph with the additional data,
for each vertex, of a cyclic order on the edges adjacent to that vertex. 
We often represent maps in the plane by adding artificial crossings
between edges so that the order of edges around vertices is respected.

The faces of a map are the connected components of the surface when removing the graph.
The genus can be recovered combinatorially by Euler formulas $m+2=n+f+2g$,
where $n$, $m$ and $f$ are the number of vertices, edges and faces of the map, respectively.
A map is {\em unicellular} if it has only one face.
A {\em rooted} map is a map with a distinguished corner, called root corner.
We will consider maps with an edge labeling (each label from $1$ to $m$, where $m$ is the number of edges,
should be used exactly once). 
The labeling is said to be {\em consistent} if, when turning counterclockwise around each vertex starting with the edge of minimal label,
edges appear in increasing order of their labels.
With such a labeling, each vertex has a unique {\em special corner}, 
the one between edges
with the minimal and maximal labels adjacent to it. We denote by $spec(V)$ the special corner of a vertex $V$.

A rooted edge-labeled map is called a {\em Hurwitz map}
if its edge labeling is consistent and the root corner is special.
In the left-hand side of \cref{fig:Hurwitz}, we show a rooted edge-labeled map;
the arrow indicates the root corner.
To help the reader, we put blue crosses in the special corners (one around each vertex).
The root corner is indeed special and one can check that the edge-labeling is consistent:
hence, this is an example of a Hurwitz map.
This map is unicellular and has genus $2$.

\begin{figure}
  \begin{center}
    \includegraphics[height=3cm]{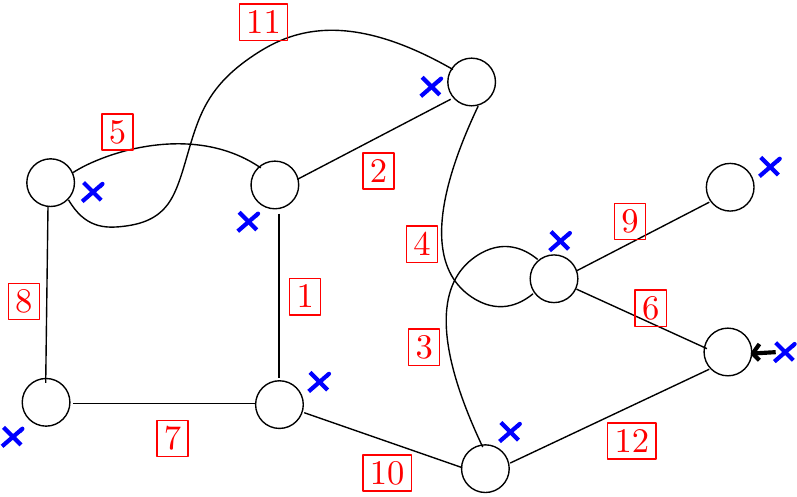} \qquad
    \includegraphics[height=3cm]{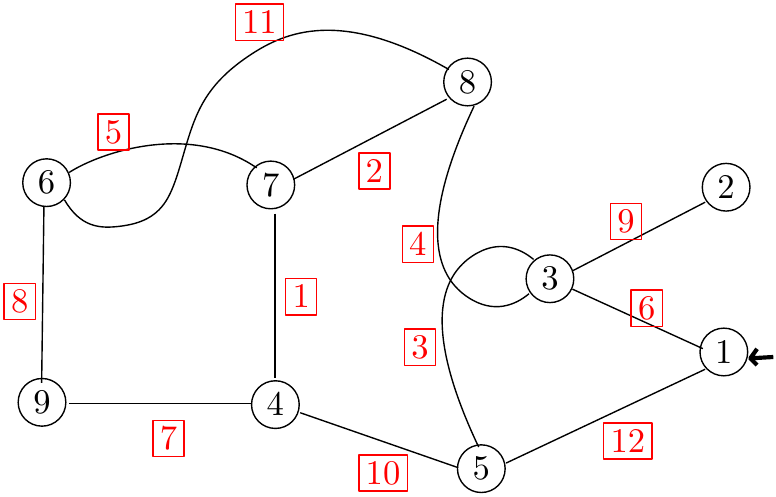}
  \end{center}
  \caption{(Left) an example of a Hurwitz map with special corners marked in blue.
  (Right) a vertex- and edge-labeled rooted map associated to the factorization \eqref{eq:ExFacto}.
  Forgetting the vertex labels gives back the map on the left.}
  \label{fig:Hurwitz}
\end{figure}
\medskip

Let $F=(\transpo_1,\dots,\transpo_m)$ be a factorization in $\setF$.
We associate to $F$ a Hurwitz map as follows:
\begin{itemize}
  \item its vertex set is $\{1,\ldots,n\}$;
  \item for $i \le m$, writing $\transpo_i=(j_i \,k_i)$,
    there is an edge labelled $i$ between $j_i$ and $k_i$;
  \item edges adjacent to each vertex are ordered in increasing order of their labels.
\end{itemize}
This is a map with labeled edges (with labels from $1$ to $m$)
and labeled vertices (with labels from $1$ to $n$). 
We root the map at vertex $1$ in its special corner.
As an example, we show on the right-hand side of \cref{fig:Hurwitz} the edge-
and vertex-labeled map associated to the factorization
\begin{equation}
  (4 \,7)\,(7 \,8)\, (3 \,5)\, (3 \,8)\, (6 \,7)\, (1 \,3)\, (4 \,9)\, (6 \,9)\, (2 \,3)\, (4 \,5)\, (6 \,8)\, (1 \,5).
  \label{eq:ExFacto}
\end{equation}
Finally we erase the labels of the vertices.
The resulting (edge-labeled rooted) map is denoted $\Map(F)$;
this is trivially a Hurwitz map.
For the above example of factorization, we obtain the map on the left-hand side of \cref{fig:Hurwitz}.

\begin{lemma}
  \label{lem:Unicellular+Relabeling}
  The map $\Map(F)$ is always unicellular. 
  Moreover, we can recover the vertex labels by turning around the face starting at the root corner
  and labeling the vertices with $1,2,\dots,n$ when we pass through their special corner
  (in particular the root vertex gets label $1$).
\end{lemma}
The left picture of \cref{fig:RelabelingHurwitz} illustrates this relabeling procedure.
\begin{figure}
  \begin{center}
    \includegraphics[height=3cm]{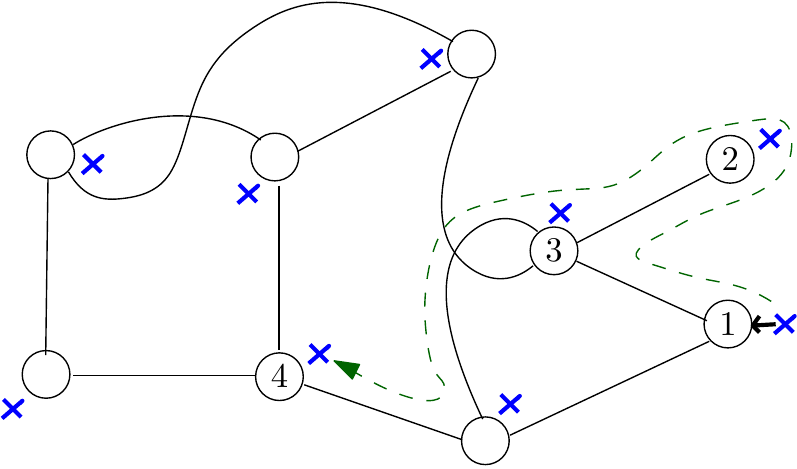}
    \qquad
    \includegraphics[height=3cm]{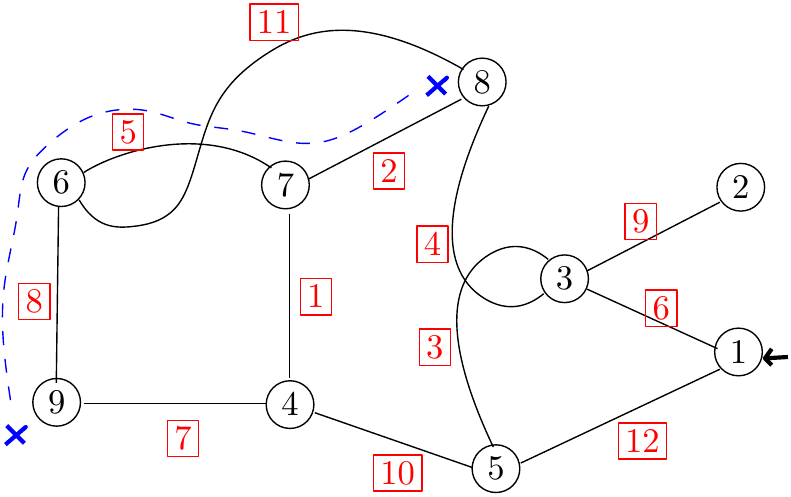}
  \end{center}
  \caption{Illustration of the relabeling procedure. 
      (Left) We turn around the unique face of the map and label the vertices
      in order $1,2,\dots,n$
      when we cross their special corner: for readability, we only show 
      the beginning of the tour of the face until label 4 is assigned.
      Note in particular that the bottommost vertex has not been labeled,
      since we only crossed its topleft corner, which is not special.\qquad
      (Right) This is the same edge- and vertex-labeled map as in \cref{fig:Hurwitz} (right),
      corresponding to the factorization in \eqref{eq:ExFacto}. We explain
      how to find the image of $8$ through the product: $\transpo_1\cdots\transpo_m$.
      The integer $8$ is mapped to $7$ by $\transpo_2$, which is itself mapped to $6$ by $\transpo_5$,
      itself mapped to $9$ by $\transpo_8$. Other transpositions do not affect the {trajectory}
      of $8$, so that the image of $8$ is $9$. Note that finding the transpositions affecting 
      the trajectory of $8$ and its successive images is simply done in the map by turning counterclockwise,
      starting at the special corner of $8$ and stopping when we cross another special corner
      (dashed blue line on the picture).}
  \label{fig:RelabelingHurwitz}
\end{figure}
\begin{proof}
We prove a more general fact.
With a sequence of transpositions $(\transpo_1,\dots,\transpo_m)$
(but not generally a factorization of $(1\,\cdots\,n)$),
we associate an edge- and vertex-labeled map $M$ as above.
Then the product $\transpo_1\cdots\transpo_m$ can be read on the map $M$ as follows:
for each face of $M$, turning around that face
and reading the labels of vertices when passing through their special corners
gives a cycle of the product.
\medskip

To understand this claim, we simply observe how the image of $i$
through the product $\transpo_1\cdots\transpo_m$ can be read of the map
(recall that we apply transposition from left to right).
The following discussion is illustrated on the right-hand side of \cref{fig:RelabelingHurwitz}.
We should first look for the smallest index $k$ such that $\transpo_k$ contains $i$,
i.e. the edge $e_1$ adjacent to vertex $i$ with smallest label.
We have $\transpo_k=(i \,j)$, where $j$ is the label of the other extremity of $e_1$.
To find it in the map, one starts at the special corner of $i$ and begins turning around its face.
We then have to find, if any, the smallest index $k'>k$ such that $\transpo_{k'}$ contains $j$.
\begin{enumerate}
  \item if there is none, it means that we have arrived at the special corner of $j$;
    in this case, $\transpo_1\cdots\transpo_m(i)=j$.
  \item if there is such a $k'$, we are not at the special corner of $j$.
    We have $\transpo_{k'} =(j \,j')$ for some $j'$. To find $j'$ on the map,
    we simply continue turning counterclockwise in the same face.
    Then we should look for the smallest index $k''>k'$ such that $\transpo_{k''}$ contains $j'$
    (if it exists) and the same case distinction applies.
\end{enumerate}
Iterating this, we turn counterclockwise in the same face until we reached a special corner.
The image of $i$ through $\transpo_1\cdots\transpo_m$ is the label of the corresponding corner, call it $h$.
Similarly, to find the image of $h$, we continue turning counterclockwise until we reach the next special corner,
and so on\ldots\ 
Each face of the map gives rise to a cycle of $\transpo_1\cdots\transpo_m$,
determined by reading labels of vertices when we cross their special corners.

We now explain why this implies our lemma.
If $(\transpo_1,\dots,\transpo_m)$ is a factorization of $(1\,\cdots\,n)$,
then the product has one cycle and hence the map is unicellular.
Moreover reading the labels of the vertices when passing through their special corners
give $(1\,\cdots\,n)$ (when starting at $1$). This proves the lemma.
\end{proof}

Since $\Map(F)$ is unicellular, has $n$ vertices and $m=n-1+2g$ edges,
it has genus $g$.
We denote by $\setH$ the set of genus $g$ unicellular Hurwitz maps with $n$ vertices.

\begin{proposition}
  $F \mapsto \Map(F)$ defines a bijection between $\setF$
  and $\setH$.
\end{proposition}
\begin{proof}
  \cref{lem:Unicellular+Relabeling} gives us the inverse mapping. Indeed, starting from a rooted
  Hurwitz map of genus $g$, we relabel the vertices as explained in \cref{lem:Unicellular+Relabeling}.
  We can then read the corresponding sequence of transpositions:
  if edge $i$ has extremities labeled $j$ and $k$, then $\transpo_i=(j \,k)$.
  By construction this sequence of transpositions is a factorization of genus $g$,
  giving the inverse mapping and proving that $F \mapsto \Map(F)$ is a bijection.
\end{proof}

\begin{remark}
If $\bm F^g_n$ is a uniform random element in $\setF$,
the above proposition ensures that $\rH= \Map(\bm F^g_n)$ is a uniform random element in $\setH$.
However, if we forget the edge labeling in $\rH$,
we get a random genus $g$ unicellular rooted map with $n$ vertices,
which is in general {\em not uniform}.
\end{remark}

\section{Asymptotic bijection}
\label{sec:asympbij}
In this section, we present an asymptotic bijection, i.e. a combinatorial operation that takes an asymptotically uniform Hurwitz map of genus $g-1$ and returns an asymptotically uniform Hurwitz map of genus $g$.

\subsection{Scheme decomposition of (almost all) Hurwitz maps}
\label{sec:hurwmaps}
In this section,
we show how we can reconstruct Hurwitz maps 
by gluing non-plane trees on a cubic map, called the {\em scheme}.
Not all Hurwitz maps can be constructed this way (because we restrict ourselves
to cubic schemes), but it turns out that the proportion of constructed Hurwitz maps 
is asymptotically $1$.
This is largely inspired from analogue considerations
on standard unicellular rooted maps (i.e without requiring that the edge-labeling is consistent),
see \cite{CMS09,Cha10}.

This section uses the theory of labelled combinatorial classes
and exponential generating functions (EGF)
as presented in the book \cite{FS09} of Flajolet and Sedgewick.
In all generating series below, 
we count objects by their number of {\em edges}.
In particular, let us denote ${h}_{g,n}=| {\setH}|$ and let 
\[{H}_g(z)=\sum_{n \ge 1} z^{n-1+2g}
\frac{{h}_{g,n}}{(n-1+2g)!}\]
be the exponential generating function of genus $g$ Hurwitz maps.

For $g=0$, the elements of $\mathcal{H}^0_n$
are trees with a consistent edge-labeling and a root corner required to be the special corner
of some vertex; we will call them \emph{Hurwitz trees}.
We note that
\begin{itemize}
\item since the root corner has to be special,
it is equivalent to distinguish a root vertex instead of a root corner;
\item  the plane embedding of a Hurwitz tree is completely determined
by its edge-labeling and that any plane embedding of a tree makes it unicellular
(this seems like a tautology since trees are obviously unicellular,
but we insist on that since this is specific to the case $g=0$).
\end{itemize}
Therefore, one can forget the plane embedding and see
Hurwitz trees as non-plane trees rooted at a distinguished {\em vertex}.
It is easily seen, see, e.g. \cite[p 127]{FS09}\footnote{We warn the reader that, in \cite{FS09},
as in most references, the size of a tree is its number of leaves, inducing a shift in the generating series.}, that
the associated generating series $T(z)$ satisfies $T(z)=\exp(zT(z))$ and is explicitly given by 
\[T(z) :=  H_0(z) = \sum_{n \ge 1} z^{n-1} \frac{n^{n-2}}{(n-1)!} .\]

We also need to consider \emph{doubly rooted Hurwitz trees},
i.e. non-plane rooted trees (or Hurwitz trees) with an additional distinguished vertex
(potentially identical to the root). 
Let $D$ be the exponential generating function of doubly rooted trees.
Since there are $n$ choices for the additional distinguished vertex in an object of size $n-1$
(i.e. with $n-1$ edges), we get $D(z)=zT'(z)+T(z)$.

We now introduce a subset $\tilde{\setH}$ of $\setH$
through a combinatorial construction.
In this construction, we will glue edge-labelled maps together, say $M$ and $M'$
with label sets $[m]$ and $[m']$ recursively.
In such operations, we deal with the labelings as 
in the theory of labelled combinatorial classes \cite[Chapter II]{FS09};
namely we consider all the ways to relabel $M$ and $M'$ in an increasing way
such that their label sets become disjoint and the union of their label sets is $[1, m+m']$.
Furthermore, when merging vertices, 
the circular order of edges incident to the new vertex needs to be defined;
in such cases we choose the unique order making 
the edge-labeling of the new map consistent.

\begin{enumerate}
\item We start from a rooted cubic\footnote{We recall that a map is cubic
  if all its vertices have degree 3.} map $M_0$ of genus $g$. 
From Euler's formula, we know that it has $4g-2$ vertices, and $6g-3$ edges. 
We split each edge into two, adding a vertex of degree two in the middle,
and choose a consistent edge-labeling of the resulting map.
Finally we forget the root corner, obtaining a (non-rooted) map $M_1$.
\item For each vertex $V$ of degree $3$ of $M_1$, 
choose a Hurwitz tree $T_V$.
We then glue $T_V$ to $M_1$, merging the root of $T_V$ with $V$.
\item For each vertex $V$ of degree $2$ of $M_1$,
we choose a doubly rooted Hurwitz tree $D_V$.
Call $e_1$ and $e_2$ the edges incident to $V$ in $M_1$
(choose arbitrarily which one is $e_1$). 
We glue $D_V$ to $M_1$ as follows: we erase the vertex $V$
and attach the now free extremities of $e_1$ and $e_2$
to the root and additional distinguished vertex of $D_V$ respectively;

(Steps ii) and iii) are illustrated on \cref{fig_garnir_carte}.)
 \item Finally, we root the map in the special corner of some vertex.
\end{enumerate}
The resulting map is a Hurwitz map of genus $g$
and we denote the set of maps obtained through this construction by 
$\tilde{\setH}$. We have $\tilde{\setH}\subset \setH$, indeed, by construction the labeling of the edges is always consistent, and we obtain unicellular maps, since we start from a unicellular map and add some trees to it. 

\begin{figure}
\center
\includegraphics[scale=0.5]{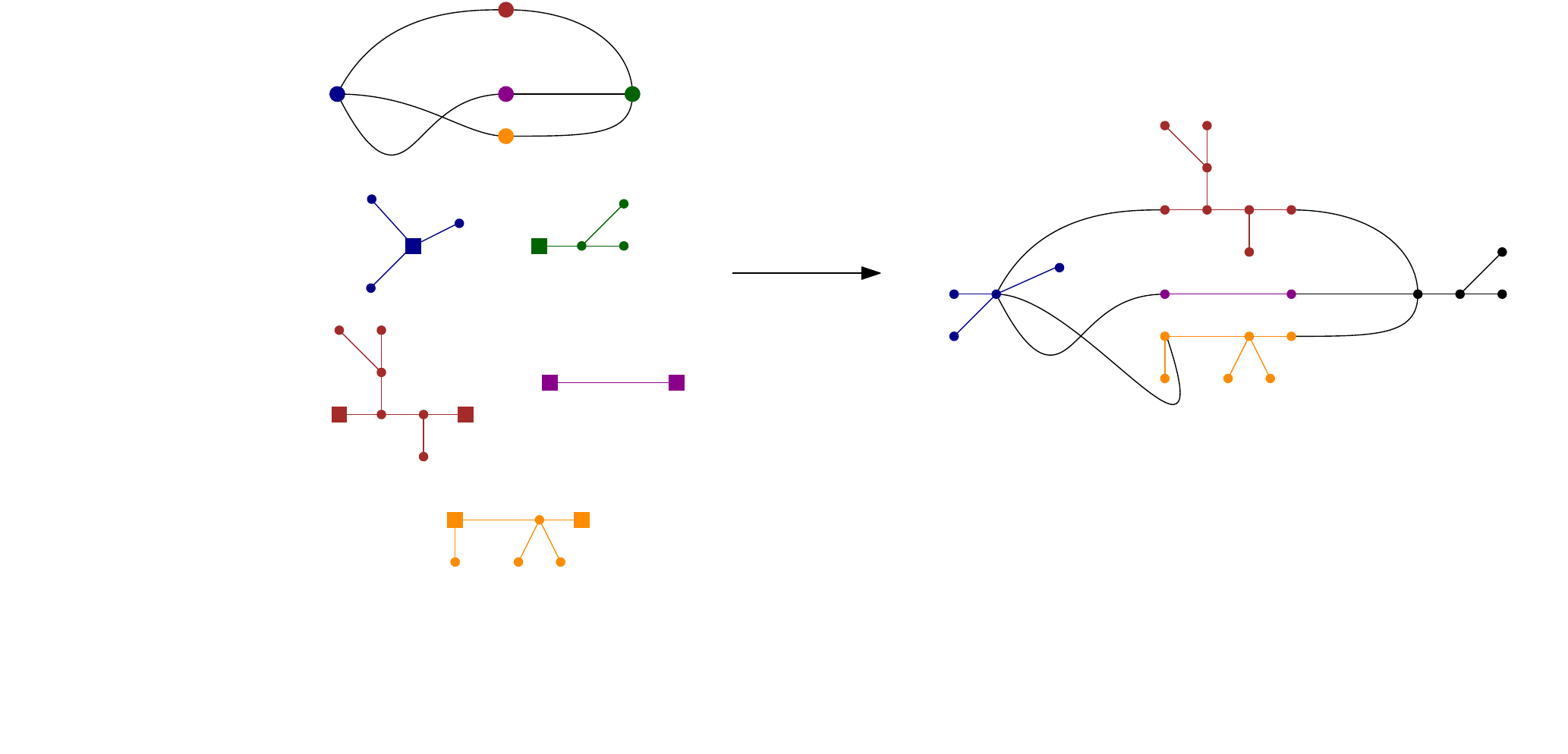}
\caption{Adding trees and doubly rooted trees to the map $M_1$. Here, $g=1$, the roots of the trees are represented by square vertices. We did not include the labeling of the edges.
This labeling however dictates the cyclic order of the edges around the vertex created by merging.}
\label{fig_garnir_carte}
\end{figure}

Let us compute the EFG
$\tilde{H}_g$ of $\bigcup_{n \ge 1} \tilde{\setH}$ following the steps of the construction.
\begin{description}
\item[step i)]
There are $\frac{2(6g-3)!}{12^g g!(3g-2)!}$ choices for $M_0$, see \cite[Corollary 2]{Cha10};
 each of those, after splitting edges into two, admits $(12g-6)!$ labelings.
The probability that a uniform random labeling is consistent is $2^{-(4g-2)}$: indeed, for each vertex of degree $3$, there is a probability $1/2$ that the order of the labels
around that vertex match the cyclic order in the map, and these events are independent
since they involve labels of disjoint sets of edges.
We conclude that there are $\frac{(12g-6)!(6g-3)!}{2^{4g-3} 12^g g!(3g-2)!}$ 
rooted versions of $M_1$.
Since each unrooted $M_1$ corresponds to $3(4g-2)$ rooted versions (the root is in a corner of a vertex of degree 3),
we finally obtain that the number of choices for $M_1$ is
\begin{equation}\label{eq:AlphaG}
 \alpha_g := \frac{(12g-6)! (6g-4)!}{2^{4g-2} 12^g g!(3g-2)!}.
\end{equation}
\item[steps ii) and iii)]
The non-rooted maps obtained after step iii) correspond to the independent choices
of a scheme $M_1$ (with $12g-6$ edges), $4g-2$ Hurwitz trees and $6g-3$ doubly rooted Hurwitz trees.
Their EGF is therefore given by
\[  \left(\alpha_g \frac{z^{12g-6}}{(12g-6)!} \right) T(z)^{4g-2}D(z)^{6g-3}.\]
\item[step iv)]
each map obtained in step iii) with $n$ vertices and $n+1-2g$ edges (represented by $z^{n+1-2g}$
in the generating series) can be rooted in $n$ possible ways.
In the EGF this multiplication by $n$ is performed by applying the differential operator 
$z\frac{\partial}{\partial z}-2g+1$.
\end{description}
As a conclusion, we have
\begin{equation}\label{eq:HgEGF}
\tilde{H}_g(z)= \Big(z\frac{\partial}{\partial z}-2g+1\Big) \cdot \Big(\alpha_g \frac{z^{12g-6}}{(12g-6)!} T(z)^{4g-2}D(z)^{6g-3}\Big).
\end{equation}

Let us set, for all $g,n$, $\tilde{h}_{g,n} := |\tilde{\setH}_{g,n}|$.

\begin{proposition}
\label{prop:AlmostAllTilde}
As $n$ tends to infinity, we have
$\tilde{h}_{g,n} \sim {h}_{g,n}$, i.e. almost all Hurwitz maps can be constructed as above.
\end{proposition}
From now on, maps of  $\tilde{\setH}$ will be called \emph{dominant Hurwitz maps}.
A consequence of Proposition \ref{prop:AlmostAllTilde} is that $\bm H^g_n$ is close in total variation to a uniform
random element $\tilde{\bm H}^g_n$ in $\tilde{\setH}$. Therefore, in many proofs, we can replace
$\bm H^g_n$ by $\tilde{\bm H}^g_n$, which we will do without further notice.
\begin{proof}
On the one hand, the asymptotics of $h_{g,n}$ is given 
in \cite[Section 6.1]{poulalhon2002factorizations}\footnote{Our quantity
$h_{g,n}$ is denoted $c^{(n)}_{T^{n-1+2g}}$ in \cite{poulalhon2002factorizations}.
We warn the reader that there is a second asymptotic formula in \cite{poulalhon2002factorizations}
in terms of the number $m$ of edges which is incorrect.}:
\begin{equation}
\label{eq:Asympt_h}
h_{g,n} \sim \frac{n^{n-2+5g}}{24^g g!}.
\end{equation}

On the other hand, we use singularity analysis starting from the expression \eqref{eq:HgEGF}
to evaluate $\tilde h_{g,n}$.
It is known (see e.g. \cite[Example VI.8 p 403]{FS09}\footnote{Our series $T$ corresponds to $y/z$ in the notation of \cite[Example VI.8]{FS09}; as noted in the \href{https://ac.cs.princeton.edu/errata/}{errata} page of the book, a minus sign is missing in front of the main term of the expansion of $y-1$.}) that the smallest singularity of $T$ is at $\rho=1/e$ and 
that the expansion of $T$ around this singularity is
\begin{equation}\label{eq_devt_T}
T(z)=e-e\sqrt{2}(1-ez)^{1/2}+O(1-ez).
\end{equation}
Using singular differentiation \cite[Thm. VI.8 p. 419]{FS09}
 (all functions involved throughout the paper are trivially $\Delta$-analytic),
we infer from this that $D(z)=zT'(z)+T(z)$ also has its smallest singularity at $\rho=1/e$
with singular expansion
\begin{equation}\label{eq_devt_D}
D(z)=\frac{e}{\sqrt 2} (1-ez)^{-1/2} +O(1).
\end{equation}
Plugging $T(z)=e+o(1)$ and \cref{eq_devt_D} into \cref{eq:HgEGF}, we get
 a singular expansion for $\tilde{H}_g$:
\begin{align*}
 \tilde{H}_g(z) &=\Big(z\frac{\partial}{\partial z}-2g+1\Big) \cdot 
 \Big(  \frac{\alpha_g}{e^{12g-6}(12g-6)!} e^{4g-2} \big(\frac{e}{\sqrt 2}\big)^{6g-3} (1-ez)^{-(6g-3)/2} (1+o(1)) \Big)\\
 & =  \frac{\alpha_g}{(12g-6)! \, 2^{3g-2} \sqrt{2}} e^{-2g+1} \, \frac{(6g-3)}{2} (1-ez)^{-(6g-1)/2} (1+o(1)).
\end{align*}
We now apply the transfer theorem \cite[Chapter VI]{FS09} and we get
\begin{equation}
\label{eq:equiv_coef_Htilde}
 [z^{n+2g-1}] \tilde{H}_g(z) = \frac{\alpha_g (6g-3)}{(12g-6)! \, 2^{3g-1} \sqrt{2}} e^{-2g+1} \, 
e^{n+2g-1} \frac{n^{(6g-3)/2}}{\Gamma\big((6g-1)/2\big)} (1+o(1)),
\end{equation}
where $[z^k]F$ is the coefficient of $z^k$ in $F$
and $\Gamma$ is the standard $\Gamma$ function from complex analysis.
Using the classical formula $\Gamma(k+1/2)=\frac{(2k)!}{2^{2k}k!}\sqrt{\pi}$,
Stirling's approximation and \cref{eq:AlphaG} for $\alpha_g$, we have
\begin{align}
\nonumber
\tilde h_{g,n}&= (n+2g-1)! [z^{n+2g-1}] \tilde{H}_g(z) \\
\nonumber
&= \frac{n^n}{e^n} \sqrt{2\pi n}\, n^{2g-1} \frac{\alpha_g (6g-3) 2^{6g-2} (3g-1)!}{(12g-6)!\,  2^{3g-1} \sqrt{2\pi} (6g-2)!} e^n n^{(6g-3)/2} (1+o(1))\\
&= \frac{n^{n+5g-2} }{24^g g!} \, (1+o(1)).
\label{eq:Asympt_htilde}
\end{align}
Comparing \cref{eq:Asympt_h,eq:Asympt_htilde} concludes the proof.
\end{proof}

\begin{remark}
It is possible to define a scheme decomposition of any Hurwitz map,
using schemes with vertices of higher degree.
In general, the number of edges of the scheme dictates
the exponent of $n$ in the number of maps with that underlying scheme.
Therefore, since non-cubic schemes have fewer than $6g-3$ edges,
maps with non-cubic schemes do not contribute to the asymptotics of $h_{g,n}$.
This  gives an alternative proof that $\tilde{h}_{g,n} \sim {h}_{g,n}$,
not relying on previously known results (neither on the asymptotics of $h_{g,n}$ \cite{poulalhon2002factorizations},
nor on the explicit formula for the number of rooted cubic unicellular maps of genus $g$ \cite{Cha10}).
We decided not to present this since this is longer than the path followed here,
and rather similar to what is done for standard unicellular maps in \cite{Cha10}.
\end{remark}

\subsection{Preliminaries}
We start by introducing some terminology.

The \emph{$2$-core} of a Hurwitz map $M$ (denoted by $\core(M)$) is the map obtained by iteratively removing leaves (and the edges they are attached to) 
until there are no more leaves left. 
The vertices of $M$ that end up having degree $3$ or more in the $2$-core are called \emph{branching vertices}.
One can see that the maps in $\tilde{\setH}$
are characterized by the fact that all branching vertices have degree $3$
and that there are no adjacent branching vertices.

If one removes all edges of $\core(M)$ from $M$, then one obtains a collection of trees. 
For any vertex $V$ in $M$, we call $\tree(V)$ the Hurwitz tree 
(rooted at a vertex of the $2$-core) to which $V$ belongs. 

We will use standard terminology for trees for vertices of Hurwitz maps;
doing that we think at the vertex $V$ as a vertex of $\tree(V)$. In particular,
  we write $\root(V)$ for the root of $\tree(V)$ and $\desc(V)$ for the set of descendants of $V$ in $\tree(V)$ (including $V$ itself).
  The parent of $V$ in $M$ will refer to the parent of $V$ in $\tree(V)$
  (well-defined only if $V \ne \root(V)$, i.e. if $V$ is not on the $2$-core of $M$).
  Also, if $V$ and $V'$ are such that $\tree(V)=\tree(V')$,
  we let $\anc(V,V')$ to be their closest common ancestor 
  (this is not defined for vertices of the same map that are in different trees attached to the $2$-core).

For any triplet $(A,B,C)$ of points in $M$ such that $\spec(A) \prec \spec(B) \prec \spec(C)$, we call $\mathcal{S}_{A,B,C}$ the set \[\{A, \root(A),\root(B),\root(C)\}\] plus $\anc(A,B)$, $\anc(A,C)$ and $\anc(B,C)$ if they exists. Note that we always have $|\mathcal{S}_{A,B,C}|\leq 4$. A triplet $(A,B,C)$ is said to be \emph{generic} if
\begin{enumerate}
\item $B,C \notin \mathcal{S}_{A,B,C}$,
\item $|\mathcal{S}_{A,B,C}|= 4$ and no elements of $\mathcal{S}_{A,B,C}$ are adjacent to each other,
\item The elements of $\mathcal{S}_{A,B,C}$ are not branching vertices or adjacent to a branching vertex,
\item the root is not a corner of a descendant of $C$ nor of 
a descendant of the parent of $A$.
\end{enumerate}

{\bf The exploration of a Hurwitz map:}
Given a corner $\gamma$ incident to a vertex $V$, we call $\next(\gamma)$ 
 the next corner around $V$ in counterclockwise order,
  and $\inc(\gamma)$ the edge that is incident to $\gamma$ and $\next(\gamma)$.
Given a Hurwitz map $M$, start at the root corner, and go along the edges of $M$, keeping the edges on the left. This exploration defines an order $\prec$ on the corners of $M$ (it is the order in which they appear in the exploration). Given a vertex $V$, we call $\text{first}(V)$ the corner around $V$ that is visited first by the exploration.

For each vertex $V$ (except the root vertex), we say that the \emph{arrival edge} of $V$ is the edge $e$ that is visited right before $V$ in the exploration.
 In other words, $e=\inc(\next^{-1}(\text{first}(V)))$. Note that $e$ is not necessarily the edge with minimal label around $V$.
The following definition was given in \cite{chapuy2011new}.
\begin{definition}
A \emph{trisection} is a corner $\gamma$ incident to a vertex $V$, such that  $\next(\gamma)\prec \gamma$, but $\next(\gamma)\neq \text{first}(V)$.
\end{definition}
Note that the definition of trisections is independent from the edge labeling.
Besides, trisections can only be incident to vertices of degree 3 or more,
and vertices of degree $3$ can be incident to at most one trisection.

\begin{lemma}\label{lem_trisec}
There are $2g$ trisections in a  Hurwitz map of genus $g$.
Moreover, trisections are always incident to branching vertices. 
Finally, in a dominant Hurwitz map, a vertex is incident to at most $1$ trisection.
\end{lemma}
By abuse of terminology, we will say that a vertex is a trisection if it is incident to a trisection.
\begin{proof}
The first point holds for any unicellular map, see \cite[Lemma 3]{chapuy2011new}.

To prove the other statements,
 we recall that the 2-core of a map $M$ is obtained
by applying the following recursive operation: pick a leaf $l$, if there is one, 
and delete it together with its incident edge $e$.
 Each step is a local operation, which does not modify the set of trisections,
 except maybe for the corners that are incident to $e$. Let us see how to affect these corners.

\begin{figure}
\center
\includegraphics[scale=0.7]{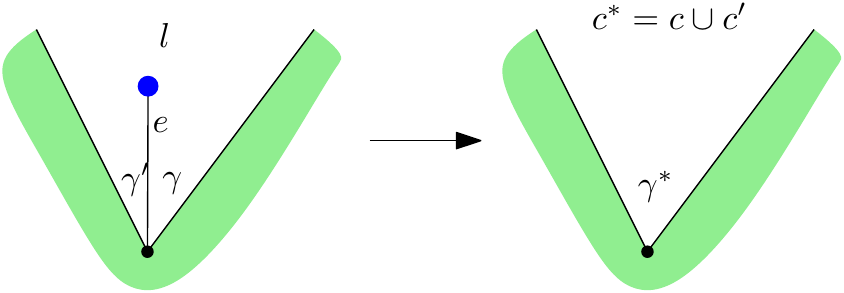}
\caption{A step of the core decomposition.}\label{fig_core_decomp}
\end{figure}
The corner incident to $l$ was obviously not a trisection and it is deleted. Let $\gamma$ and $\gamma'$ be the other corners that were incident to $e$, such that $\gamma'=\next(\gamma)$ (see Figure~\ref{fig_core_decomp}).
We have $\gamma\prec \gamma'$, unless the root is $\gamma'$ or the corner incident to $l$.
In any case, $\gamma$ was not a trisection in the initial map. 
The corners $\gamma$ and $\gamma'$ are merged into a corner $\gamma^*$ in the map.
It is easy to see that $\gamma'$ was a trisection in the initial map 
if and only if $\gamma^*$ is a trisection in the map obtained after deletion.
 Therefore, the multiset of vertices adjacent to trisections
 (taking as many copies of each vertex as the number of trisections adjacent to it) 
 is preserved when going from a map $M$ to its 2-core $\core(M)$.
 
Recall that by definition, the vertices of degree at least 3 in $\core(M)$ are the branching vertices of $M$.
Since trisections in $\core(M)$ are incident to vertices of degree 3 or more, 
in $M$, they can only be incident to branching vertices. This is the second statement of the lemma.

For the last statement, we assume that $M$ is a dominant Hurwitz map,.
In particular, its branching vertices are cubic in $\core(M)$;
this implies that they can be incident to at most one trisection in $\core(M)$, 
and the same should be true in $M$.
\end{proof}

Thanks to the above lemma, we can define a notion of \emph{good trisection}.

\begin{definition}
Let $A$ be a trisection in a dominant Hurwitz map $M$. It has exactly three neighbors that belong to the $2$-core, let us call them $B$, $C$ and $D$. The vertex $A$ is said to be a good trisection if and only if the root of $M$ is not a corner of a descendant of $A$, $B$, $C$ or $D$.
\end{definition}

\subsection{The bijection}

In this section we describe a bijection between subsets of marked 
unicellular Hurwitz maps of genera $g-1$ and $g$, respectively.
We will see in the next section that these sets are asymptotically dominant,
so that this bijection is in some sense an "asymptotic bijection"
between genus $g-1$ and genus $g$ unicellular Hurwitz maps.

We introduce two sets of maps with marked vertices.  
\begin{itemize}
\item
Let $\mathcal{B}_{g-1,n}$ be the set of tuples $(M,A,B,C,L)$
where $M$ is a map in $\tilde{\mathcal{H}}^{g-1}_n$,
$A$, $B$ and $C$ are three vertices of $M$ {such that $\spec(A) \prec \spec(B) \prec \spec(C)$}
 (which we refer to as distinguished vertices), 
and $L$ is a 2-element subset of the integer interval $[1,n-1+2g]$.
\item
Let $\mathcal{A}_{g-1,n}$ be the set of tuples $(M,A,B,C,L)$ in  $\mathcal{B}_{g-1,n}$
with the additional property that $(A,B,C)$ is generic. 
\end{itemize}

We also introduce two sets of maps with a marked trisection.

\begin{itemize}
\item
Let $\mathcal{D}_{g,n}$ be the set of couples $(M,A)$
where $M$ is a map in $\tilde{\mathcal{H}}^{g}_n$, and $A$ is a trisection of $M$.

\item
Let $\mathcal{C}_{g,n}$ be the set of couples $(M,A)$ in  $\mathcal{D}_{g,n}$
with the additional condition that $A$ is a good trisection.
\end{itemize}

Finally, given a vertex $V$ and a number $\vvv$, we say that the \emph{$\vvv$-corner} of $V$ is the corner around $V$ in which we could add an edge with label $\vvv$ while still satisfying the condition for Hurwitz maps (we assume that $\vvv$ is not the label of an edge of the map).
We now introduce a bijection between the sets $\mathcal{A}_{g-1,n}$ and $\mathcal{C}_{g,n}$.
\medskip

{\bf The operation $\Phi$:} 
Given $(M,A,B,C,L)\in\mathcal{A}_{g-1,n}$, the operation $\Phi$ is defined as follows.
First relabel the edges of $M$ in such a way that the labels belong to $[1,n-1+2g]\setminus L$ preserving the order between labels 
(i.e. we add $1$ to all labels which are at least $\min(L)$,
and then an additional $1$ to all that are at least $\max(L)$).
Now, we write $L = \{\vvv,\www \}$, 
assuming without loss of generality that the $\vvv$-corner of $A$ is visited before the $\www$-corner of $A$ in the exploration (this does not imply anything about who is greater between $\vvv$ and $\www$). 
Draw an edge with label $\vvv$ between the $\vvv$-corners of $A$ and $B$,
and an edge with label $\www$ between the $\www$-corners of $A$ and $C$ to obtain a map $M^*$.
We set $\Phi(M,A,B,C,L)=(M^*,A)$, i.e. $M^*$ with a marked vertex $A$.
We refer to Fig.~\ref{fig_bij} for a schematic representation of the mapping $\Phi$,
useful to follow the next proof.

\begin{figure}
\includegraphics[scale=0.8]{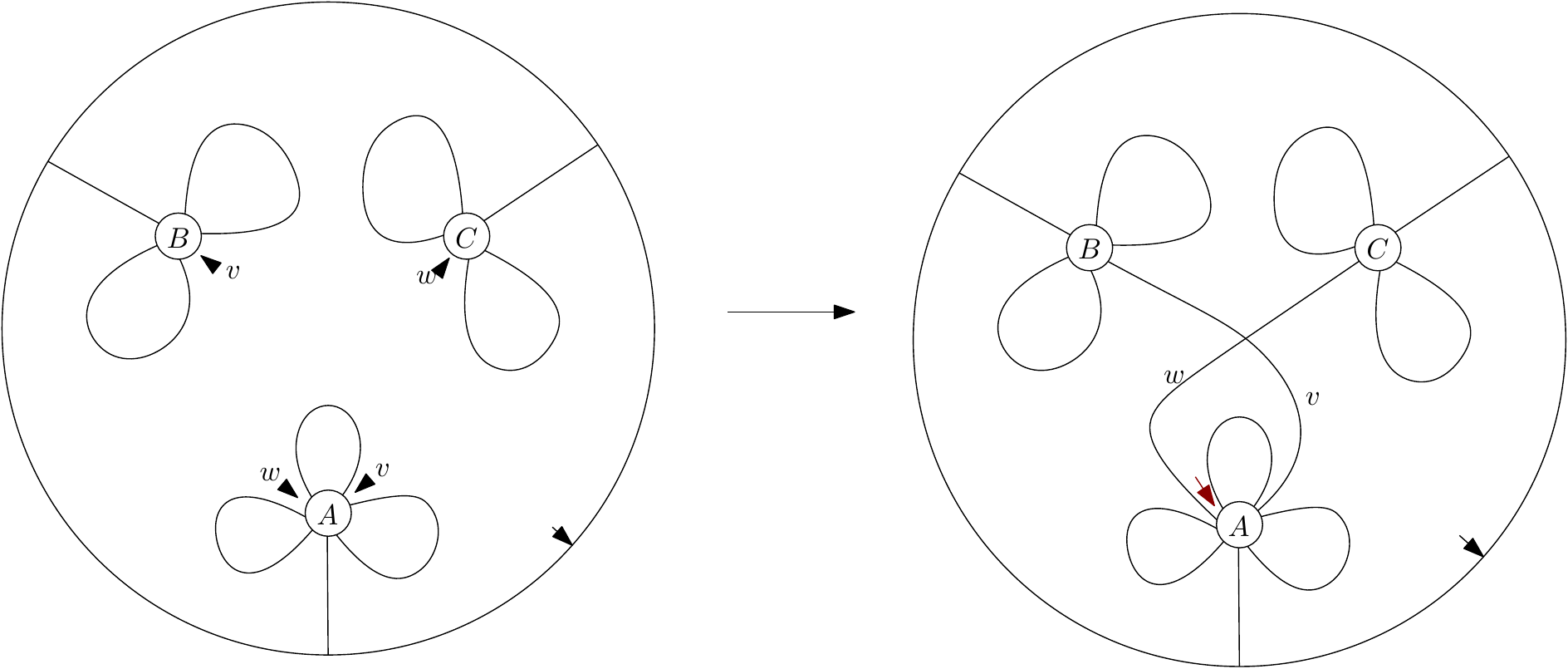}
\caption{Schematic representation of a map $M$ and its image $\Phi(M)$.
The "bubbles" represent the trees attached to $A$, $B$ and $C$, 
while the circle represents the rest of the map;
this representation encodes the fact that $(A,B,C)$ is generic in $M$.
In particular, in $M$, all corners of $A$ are visited first, then all corners of $B$,
then all corners of $C$. The black arrow on the outer circle represents the root,
the red arrow on the right hand side is a new trisection (see the proof of \cref{lem:set_bij}).
}\label{fig_bij}
\end{figure}

\begin{lemma}\label{lem:set_bij}
The image $\Im\Phi$ is included in $\mathcal{C}_{g,n}$.
\end{lemma}

\begin{proof}
It is easy to see that the exploration process of $M^*$ visits all its corners (see Fig.~\ref{fig_bij}).
Therefore, $M^*$ only has one face, and by the Euler formula, it has genus $g$. 
Moreover, it is a Hurwitz map, because its labelling is consistent by construction.
Note that branching vertices of $M$ are obviously still branching vertices of $M^*$.
Also, the elements of $\mathcal{S}_{A,B,C}$ become branching vertices.
 The fact that $(A,B,C)$ is generic implies that $M^*$ is a dominant Hurwitz map: indeed, the new branching vertices all have degree $3$ in the $2$-core, and they are not adjacent to another branching vertex.

Now we observe that $A$ is a trisection in the map. More precisely, in $M^*$, let $e_\www$ be the edge labeled $\www$, and let $\gamma$ be the corner incident to $A$ such that $\inc(\gamma)=e_\www$. Then $\gamma$ is a trisection, because $\next(\gamma)\prec \gamma$ and $\next(\gamma)\neq\text{first}(V)$ (see Figure~\ref{fig_bij}).

We finally check that $A$ is a good trisection.
The neighbors of $A$ in the $2$-core of $M^*$ are $B$, $C$ and the parent $D$
of $A$ in $M$. 
The root is not a descendant of $A$ or $D$ in $M^*$ since 
it is not a descendant of $D$ in $M$ (condition iv) in the definition of a generic triple).
Again, by condition iv), it is not a descendant of $C$.
Finally note that the condition $\spec(A) \prec \spec(B) \prec \spec(C)$
imposed on elements of $\mathcal A_{g-1,n}$ forbids the root to be a descendant of $B$.
This proves that $A$ is a good trisection.
\end{proof}

{\bf The inverse operation $\Psi$:}
We start with a pair $(M^*,A) \in\mathcal C_{g,n}$. Let $e_1$, $e_2$ and $e_3$ be the edges incident to $A$ that belong to the $2$-core. Since the root does not belong to $\desc(A)$, one of these edges is the arrival edge of $A$. Wlog, say it is $e_1$ and that $e_1, e_2, e_3$ are in this cyclic order.
Let $B$ and $C$ be the respective other endpoints of $e_2$ and $e_3$, and $\vvv$ and $\www$ their respective labels. 
We call $M$ the map obtained from $M^*$ by deleting $e_2$ and $e_3$, and set $L=\{\vvv,\www\}$.
We relabel the edges of $M$ in the unique order-compatible way, so that the labels are in $[1,n-3+2g]$. 

Finally, we set $\Psi(M^*,A)=(M,A,B,C,L)$.

\begin{lemma}\label{lem_set_bij_inverse}
The image $\Im\Psi$ is included in $\mathcal{A}_{g-1,n}$.
\end{lemma}

\begin{proof}
Once again, it is easy to see from the exploration that $M$ has only one face (see Figure~\ref{fig_inverse_bij}). By Euler's formula, since we deleted two edges, the genus of $M$ has to be be $g-1$. 
Since we have only been deleting two edges, the edge labeling stays consistent, 
and the set of branching vertices of $M$ is included in that of $M^*$ ;
in particular there are no pairs of adjacent branching vertices and they are all of degree $3$
in the $2$-core. Hence $M\in \tilde{\mathcal{H}}^{g-1}_n$.

Now, we check that $(A,B,C)$ is a generic triplet.
 First, note that $A$ is a branching vertex of degree $3$ in the $2$-core of $M^*$,
 while $B$ and $C$ have degree $2$ ($B$ and $C$ are neighbors of the branching vertex $A$ in $M^*$ and thus not branching vertices themselves).
 If we delete the edges $e_2$ and $e_3$ from the $2$-core of $M^*$, 
 the vertices $A$, $B$ and $C$ have degree $1$ in the resulting map.
  Therefore $A$, $B$ and $C$ do not belong to the $2$-core of $M$ and cannot be ancestors of each other:
  in particular, condition i) of being generic is satisfied.
 Also, the set $\mathcal{S}_{A,B,C}$ in $M$ is exactly the set of vertices that were branching in $M^*$ but not anymore in $M$. 
 Therefore $|\mathcal{S}_{A,B,C}|=4$ and elements of
 $\mathcal{S}_{A,B,C}$ cannot be adjacent to each other,
 nor equal or adjacent to a branching vertex of $M$
 (recall that since it is a dominant Hurwitz map,
 $M^*$ have $4g-2$ branching vertices,
 and that no two of them can be adjacent). 
 Finally,
 since $e_1$ is the arrival edge of $A$ and since $A$ is a good trisection,
 the root of $M^*$ must be somewhere in the blue zone of Figure~\ref{fig_inverse_bij}, 
 and therefore same goes for the root of its image $M$.
 This shows condition iv) of the definition of generic.
\begin{figure}
\center
\includegraphics[scale=0.7]{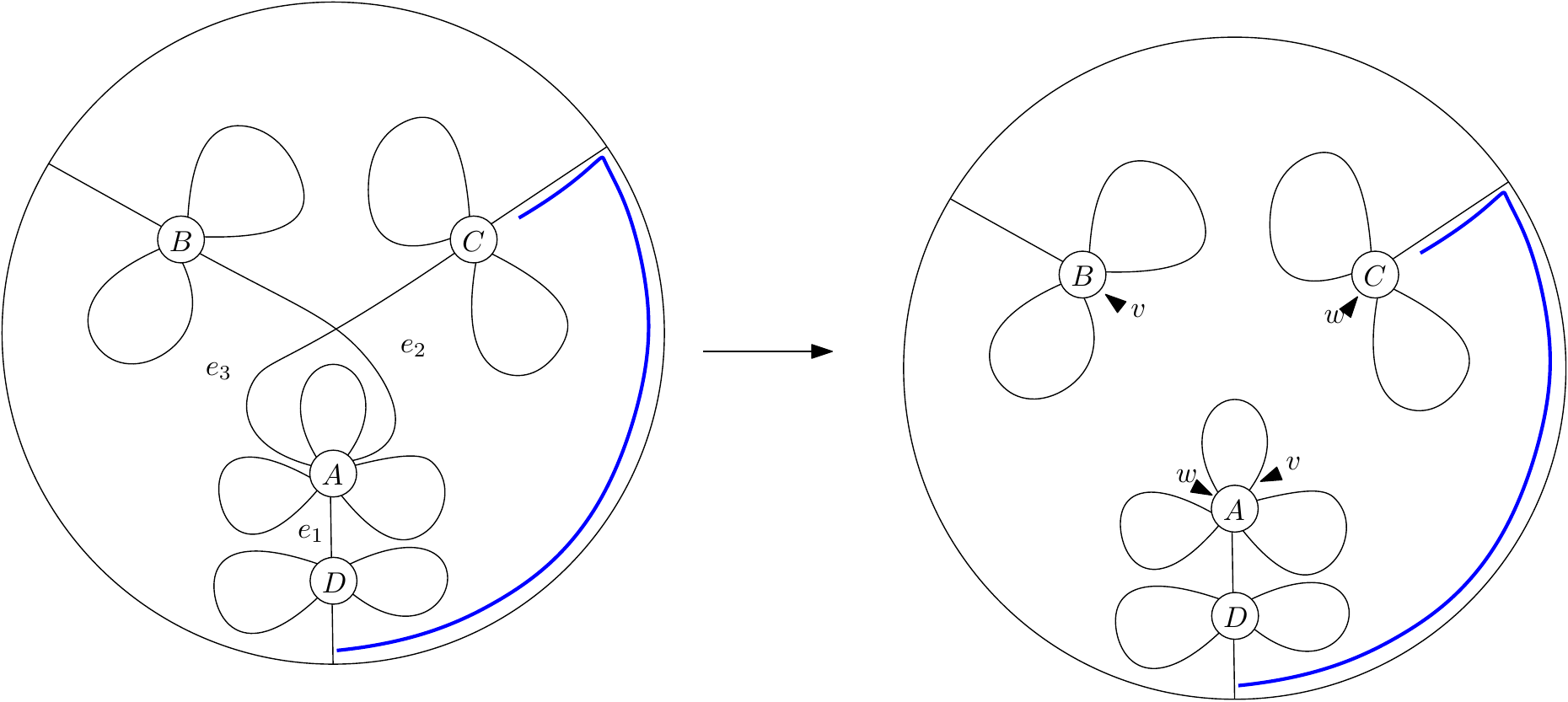}
\caption{The inverse bijection. The root has to be somewhere in the blue zone.}\label{fig_inverse_bij}
\end{figure}
\end{proof}

The following is now immediate by construction.
\begin{proposition}\label{prop_asympto_bij}
The operations $\Phi$ and $\Psi$ are inverse to each other;
 thus they are bijective mappings from $\mathcal{A}_{g-1,n}$ to $\mathcal{C}_{g,n}$
 and conversely.
\end{proposition}

\subsection{Almost all elements are in bijection}
The main goal of this section is to prove the following proposition.
The proofs mix analytic combinatorics techniques and results from the theory of random trees.
\begin{proposition}
\label{prop:almost_all_bij}
  For fixed $g\ge 1$, as $n$ tends to $+\infty$, we have the following.
\begin{itemize}
\item If $x$ is uniformly chosen in $\mathcal{B}_{g-1,n}$, then $x\in \mathcal{A}_{g-1,n}$ whp.
\item If $x$ is uniformly chosen in $\mathcal{D}_{g,n}$, then $x\in \mathcal{C}_{g,n}$ whp.
\end{itemize}
\end{proposition}
We start with the first item.
We note that a uniform element $(M,A,B,C,L)$ in $\mathcal{B}_{g-1,n}$ is taken as follows.
First take a uniform $M$ in $\tilde{\mathcal{H}}^{g-1}_n$. Second take a uniform random 3-element
subset of the vertex set of $M$ and name its elements $A$, $B$ and $C$
in such a way that $\spec(A) \prec \spec(B) \prec \spec(C)$.
Finally, take a uniform 2-element subset $L$ of $[1,n-1+2g]$,
independently from the rest.

Furthermore by construction, a uniform random map $M$ in $\tilde{\mathcal{H}}^{g-1}_n$ can be obtained
by 
\begin{itemize}
\item taking uniformly at random the scheme $M_1$ in a finite set (see the step i) and ii)
of the construction of $\setH$ in \cref{sec:hurwmaps});
\item
and then taking a tuple of simply and doubly rooted Hurwitz trees 
\[\Big( (T_V)_{V \in M_1;\deg(V)=3},\, (D_V)_{V \in M_1;\deg(V)=2} \Big)\]
uniformly at random such that the sum of their sizes is equal to $n$.
\end{itemize}
In the second step, the size of the individual trees $T_V$ and $D_V$ are random
but it is important to notice that, conditionally on its size $n_V$, each $T_V$ (resp. $D_V$)
is a uniform simply rooted (resp. doubly rooted) Hurwitz tree of size $n_V$.

We first state and prove some lemmas about the respective sizes
 of various parts of the random map $M$: trees attached to branching vertices
 and the descendant sets of (the parents of) the random vertices
 $A$, $B$ and $C$.

\begin{lemma}\label{lem_small_tree}
Let $H$ be uniform in $\setH$, and $V$ be a branching vertex of $H$. Then 
$|tree(V)|$ is $\O_P(\sqrt{n})$.
\end{lemma}

\begin{proof}
We can modify $\tilde{H}_g$ by adding a variable $u$ that records the size 
$|\tree(V)|$ of an arbitrary tree attached to a branching vertex:
\[\tilde{H}_g(z,u)=\Big(z\frac{\partial}{\partial z}-2g+1\Big) \cdot \Big(\alpha_g \frac{z^{12g-6}}{(12g-6)!} T(uz)T(z)^{4g-3}D(z)^{6g-3}\Big).\]
Then it is classical (see, e.g., \cite[p. 159]{FS09}) 
that the expected size of $|\tree(V)|$ can be recovered 
from this bivariate generating series by the formula
\[\E\big[ |\tree(V)| \big] = \frac{[z^n]\tilde{H}'_g(z,1)}{[z^n]\tilde{H}_g},\]
where $\tilde{H}'_g$ denotes the derivative of $\tilde{H}_g$ with respect to $u$.
In our case, we have
\[\tilde{H}'_g(z,1)=\Big(z\frac{\partial}{\partial z}-2g+1\Big) \cdot \Big(\alpha_g \frac{z^{12g-5}}{(12g-6)!} T'(z) \, T(z)^{4g-3}\, D(z)^{6g-3}\Big).\]

But, from \eqref{eq_devt_T} and singular differentiation, we get
 \[T'(z)=\frac{e^2}{\sqrt 2} (1-ez)^{-1/2} +O(1).\]
Recalling the singular expansions \eqref{eq_devt_T} and \eqref{eq_devt_D} for $T$ and $D$,
we have the following expansion for $\tilde{H}'_g(z,1)$,
\[\tilde{H}'_g(z,1) = C (1-ez)^{-3g} (1+o(1)),\]
where, here and below, $C$ is a constant which we do not need to explicit
and whose value can change from line to line.
By the transfer theorem, this yields
\[[z^n]\tilde{H}'_g(z,1) = C e^n n^{3g-1} (1+o(1)) \]
and finally, using \eqref{eq:equiv_coef_Htilde},
\[\E\big[ |\tree(V)| \big]=\frac{[z^n]\tilde{H}'_g(z,1)}{[z^n]\tilde{H}_g}=C\sqrt{n} (1+o(1)).\]
We conclude by applying Markov inequality.
\end{proof}

\begin{lemma}\label{lem_small_desc_trees}
Using the above notation, the parents of $A$, $B$ and $C$
all have $\O_P(1)$ descendents.
\end{lemma}

Before proving the lemma, we make two comments:
\begin{itemize}
\item if $A$, $B$ or $C$ is on the $2$-core of $M$,
its parent is ill-defined. We will see later in the proof of \cref{prop:almost_all_bij}
that this happens with probability tending to $0$;
\item
the lemma implies in particular
that  $A$, $B$ and $C$ themselves have $\O_P(1)$ descendents in $M$.
\end{itemize}
\begin{proof}
It is enough to prove that the number of descendents
of the parent of a uniform random vertex $A$ in a uniform random unicellular
Hurwitz $M$ of genus $g-1$ has $\O_P(1)$ descendents.
If the size of $\tree(A)$ is bounded (along a subsequence),
then the lemma is trivial (along that subsequence).
We therefore assume that the size of $\tree(A)$ tends to infinity.

Conditioning on the $2$-core of $M$ and on which vertex of this $2$-core
is $\root(A)$, then $\tree(A)$ conditioned to its size is a uniform random Hurwitz tree of this size and $A$ is a uniform random vertex in $\tree(A)$.
A uniform random Hurwitz tree is the same as a uniform random Cayley tree
(endowed with its unique plane embedding making the labeling consistent);
it is known that such trees are distributed as conditioned Galton-Watson
trees with offspring distribution $\Poisson(1)$ (see e.g.  \cite[Example $10.2$]{Jan12}).
For such trees in the critical case (as here), 
it is known (see \cite[Theorem 5.1]{stufler2019local}),
that the fringe subtree rooted at the parent of
a uniform random vertex\footnote{The fringe subtree
is the tree consisting at a node and all its descendents.} 
tends to a limiting tree, which is finite a.s.
In particular its size is stochastically bounded, proving the lemma.
\end{proof}

After this preparation we can prove \cref{prop:almost_all_bij}.

\begin{proof}[Proof of Proposition \ref{prop:almost_all_bij}]
We consider the first item:
we need to prove that if $(M,A,B,C,L)$ is a uniform random in $\mathcal B_{g-1,n}$,
 then $(A,B,C)$ is generic with high probability
We will prove in fact that a stronger statement holds with high probability:
 $(A,B,C)$ is generic, $B$ and $C$ are not neighbors of any element in $\mathcal S_{A,B,C}$ nor of a branching vertex,
 and the root is not a descendant of the parent of $B$ or $C$.
The point is that this is a symmetric condition on $(A,B,C)$
 (in the definition of generic,
$A$ is not a neighbor of other elements of $\mathcal S_{A,B,C}$,
and the root should not be a descendant of the parent of $A$).
Thus we can prove it
for a uniform random triple $(A,B,C)$ of vertices of $M$,
without the constraint $\spec(A) \prec \spec(B) \prec \spec(C)$
imposed above.

First we note that, in the random map $M$ the root can be chosen at the end,
as a uniform random special corner in the map.
In particular, \cref{lem_small_desc_trees} implies that, with high probability,
the root is not a descendant of the parent of $A$, $B$ or $C$.

Let us now discuss conditions i), ii) and iii) of being generic.
We first note that as an immediate consequence of \cref{lem_small_tree},
the uniform random points $A$, $B$ and $C$
are not elements of a single-rooted tree $T_V$ in the construction of $M$,
but of some doubly-rooted tree $D_V$.
In particular, elements of $\mathcal S_{A,B,C}$ are not branching vertices.
We denote by $\dt(A)$, $\dt(B)$ and $\dt(C)$ the doubly rooted trees
containing $A$, $B$ and $C$.

We may have $\dt(A)=\dt(B)$, $\dt(A)=\dt(C)$ or $\dt(B)=\dt(C)$;
such events happen 
with probability bounded away from $0$ at the limit $n \to +\infty$.
We make a case distinction.

We first condition on the fact that $\dt(A)$, $\dt(B)$ and $\dt(C)$
are distinct.
We look at the doubly-rooted tree $\dt(A)$, considering one of the distinguished vertices, call it $R$ as the root and the other one, call it $P$,
a marked vertex; the vertex $\root(A)$ (i.e. the vertex of $\core(M)$ closest to $A$) is then the closest common ancestor of $A$ and $P$ 
in $\dt(A)$.

There are only $6g-3$ doubly rooted trees attached to vertices of $M_1$. Hence, 
if we choose a sequence $\omega_n=o(n)$, we have
\begin{align*}
\sum_{V \in M_1, deg(V)=2} |D_V| \mathds{1}_{|D_V|\leq \omega_n} \leq (6g-3) \omega_n = o(n).
\end{align*}
In particular, with high probability, $dt(A), dt(B), dt(C)$ have sizes $\geq \omega_n$ since $A, B$ and $C$ are uniform.
We now use the fact that the tree $\dt(A)$, conditioned to its size $n'_A$, is a uniform random Hurwitz tree of size $n'_A$ with two marked vertices $A$ and $P$.
As explained above, it has the same distribution as a
conditioned Galton-Watson
tree with offspring distribution $\Poisson(1)$ conditioned to its size.
In such trees (when the offspring distribution is critical and has finite variance), it is known that
the distances between marked vertices and their closest common ancestors,
between their closest common ancestors and between the closest common ancestors and the root of the tree are all of order $\sqrt{n'_A}$
(in fact, the limiting law of these distances after rescaling by $\sqrt{n'_A}$
is known see, e.g., \cite{aldous1993}).
We conclude that the distance between $A$, $P$, $R$ and $\root(A)$ are all of order $\sqrt{n'_A} \ge \sqrt{\omega_n}$.
In particular, w.h.p., $A$ and $\root(A)$ are distinct, not neighbors from each other and not neighbors of a branching vertex in $M$
(by construction, among the vertices of $\dt(A)$,
only $R$ and $P$, the two roots of $D_V$
are neighbors of a branching vertex in $M$).
The same holds, replacing $A$ by $B$ and $C$, proving
that $(A,B,C)$ is generic w.h.p.

Consider now the case where $\dt(A)=\dt(B)=\dt(C)$ is the same doubly rooted tree $D$
of size $n'_D$. Again, for any sequence $\omega_n=o(n)$,
one has $n'_D \ge \omega_n$ with high probability.
As above, we see one of the distinguish vertices of $D$, say $R$,
of $D$ as the root, and the other, say $P$, as a marked vertex.
Then the set $\mathcal S_{A,B,C} \setminus \{A\}$ is exactly the set of closest common
ancestors of some pairs of vertices among $\{A,B,C,P\}$ in $T$.
Using \cite[Theorem 2.11]{LG05} and the fact that $A,B,C,P$ are i.i.d. uniform random vertices in the tree, we know that with high probability,
$\mathcal S_{A,B,C}$ has cardinality $3$, 
and that these points are at distance $\Theta(\sqrt{n'_D})$
from each other and from $A,B,C,P$ and $R$.
We conclude that with high probability $(A,B,C)$ is generic.

The case where two of $\dt(A)$, $\dt(B)$ and $\dt(C)$ coincide but
where the third one is distinct, is treated similarly.
This proves the first item of Proposition \ref{prop:almost_all_bij}.
\bigskip

We now consider the second item of Proposition \ref{prop:almost_all_bij}.
 Take $M$ a uniform map in $\tilde{\mathcal{H}}_n^g$. 
 We prove that, with high probability, the corner root is not incident to a descendant of a branching
 vertex or a descendant of one of its neighbors.
Since trisections are incident to branching vertices by Lemma \ref{lem_trisec}, 
this will conclude the proof.

We know by Lemma \ref{lem_small_tree} that $|tree(A)|=\O_P(\sqrt{n})$, and thus with high probability the root corner is not incident to a vertex of $tree(A)$. Now let $B, C, D$ be the vertices of $\core(M)$ adjacent to $A$ and $dt(B), dt(C), dt(D)$ be the doubly rooted Hurwitz trees containing $B, C, D$ respectively. We need to prove that with high probability the root corner is not a descendant of $B$ in $M$. We can assume without loss of generality that $B$ is the root of $dt(B)$, which has size at most $n$. Remark that if $|dt(B)| \leq \sqrt{n}$, it is clear that the root is not in $dt(B)$ with high probability.

Let us therefore restrict ourselves to the case $|dt(B)| \geq \sqrt{n}$. In this case, using again the fact that $dt(B)$ conditioned to its size is distributed as a size-conditioned $Poisson(1)$-Galton-Watson tree, it is known that its root is not a branching point, in the sense that one of its children has a fringe subtree of size $|dt(B)|(1-o(1))$. Hence, the additional marked vertex in $dt(B)$ lies with high probability in this fringe subtree, and the rest of the tree (corresponding to the descendants of $B$ in $M$)
has size $o(|dt(B)|)=o(n)$. Therefore, with high probability the root is not a descendant of $B$. The same applies for $C$ and $D$, and the result follows.
\end{proof}

As a consequence of \cref{prop_asympto_bij,prop:almost_all_bij},
 we will see how to use $\Phi$
to recursively generate an asymptotically uniform element of $\setH$.

Formally, let $\Phi^*$ be the following random operation.
Given a map $M\in \setHH$, sample three uniform vertices $A,B,C$ in $M$
($A,B,C$ are named so that $\spec(A) \prec \spec(B) \prec \spec(C)$), 
and  take $\{\vvv,\www\}$ uniform in $[1,n-1+2g]$. 
If $M$ is not in $\tilde{\mathcal{H}}^{g-1}_n$ or $(M,A,B,C,\vvv,\www)\not\in \mathcal{A}_{g-1,n}$, set $\Phi^*(M)=\dag$.
Otherwise, let $(M',A)=\Phi(M,A,B,C,\vvv,\www)$ and set $\Phi^*(M)=M' \in \setH$.

\begin{theorem}\label{thm_asympto_bij}
We recall that, for $g\ge 0$ and any $n\ge 1$, we denote by $\bm H^g_n$
a uniform random element in $\setH$. Then the above defined random mapping $\Phi^*$
has the property that
\[d_{TV}(\Phi^*(\bm H^{g-1}_n),\bm H^g_n)\to 0\] 
as $n\to\infty$.
\end{theorem}

\begin{proof}
We first condition on $(\bm H^{g-1}_n,A,B,C,\vvv,\www)$ being in $\mathcal{A}_{g-1,n}$.
Then $(\bm H^{g-1}_n,A,B,C,\vvv,\www)$ 
is uniformly distributed in $\mathcal{A}_{g-1,n}$ . 
By Theorem~\ref{prop_asympto_bij}, if we let \[(M,A)= \Phi(\bm H^{g-1}_n,A,B,C,\vvv,\www),\] then $(M,A)$ is uniform in $\mathcal{C}_{g,n}$.

By the second point of Proposition~\ref{prop:almost_all_bij}, the uniform measures in $\mathcal{C}_{g,n}$ and $\mathcal{D}_{g,n}$ have total variation distance $o(1)$. 
Moreover if $(M^*,A)$ is uniform in $\mathcal{D}_{g,n}$,
then $M^*$ is uniform in $\setHH$ as every map of $\setHH$ has exactly $2g$ trisections (Lemma~\ref{lem_trisec}).
We conclude that, conditioning on $(\bm H^{g-1}_n,A,B,C,\vvv,\www)$ being in $\mathcal{A}_{g-1,n}$,
\[d_{TV}(\Phi^*(\bm H^{g-1}_n),\bm H^g_n)\to 0.\] 
Since the conditioning event has probability tending to $1$
(see Proposition~\ref{prop:AlmostAllTilde} and the first point of Proposition~\ref{prop:almost_all_bij}), the same holds without conditioning.
\end{proof}

\subsection{From Hurwitz maps to factorizations}
Here we transfer our asymptotically uniform generation algorithm for unicellular
Hurwitz maps (\cref{thm_asympto_bij}) to factorisations of the long cycle.
This will complete the proof of Theorem~\ref{thm:RandomGenerationFacto}.
We need one more definition and one more lemma before proceeding to the proof.

Let us define $\Phi'$ to be just like $\Phi$ but by pairing the "wrong" corners. 
More precisely, given $(M,A,B,C,\vvv,\www)\in\mathcal{A}_{g-1, n}$, the operation $\Phi'$ 
is defined as follows.

Relabel the edges of $M$ in such a way that they belong to $[1,n-1+2g]\setminus \{\vvv,\www\}$ while keeping the order between them (this means adding 1 or 2 to some of the labels).
As for $\Phi$, say wlog that the $\vvv$-corner of $A$ is visited before the $\www$-corner of $A$ in the exploration (again, this does not imply anything about who is greater between $\vvv$ and $\www$).
 Draw an edge (with label $\vvv$) between the $\vvv$-corners of $A$ and $C$, 
 and an edge  between the $\www$-corners of $A$ and $B$.
 We obtain a map $M^*$ and set $\Phi'(M,A,B,C,\vvv,\www)=(M^*,A)$.

\begin{lemma}\label{lem_bad_phi}
For any $(M,A,B,C,\vvv,\www)\in\mathcal{C}_{g,n}$, writing $\Phi'(M,A,B,C,\vvv,\www)=(M^*,A)$,
 the map $M^*$ is {\bf not} unicellular.
\end{lemma}

\begin{proof}
Let $p$ be the part of $M$ that is visited before the $\vvv$-corner of $A$ in the exploration, and $p'$ be the part of $M$ that is visited after the $\vvv$-corner of $C$ in the exploration. Then, in $M^*$, the exploration starts at the root, visits $p$, then goes along the edge with label $\vvv$, then visits $p'$, then is back at the root. It did not visit $B$, therefore $M^*$ has more than one face, and is not unicellular. See Figure~\ref{fig_fausse}.
\begin{figure}
\center
\includegraphics[scale=0.7]{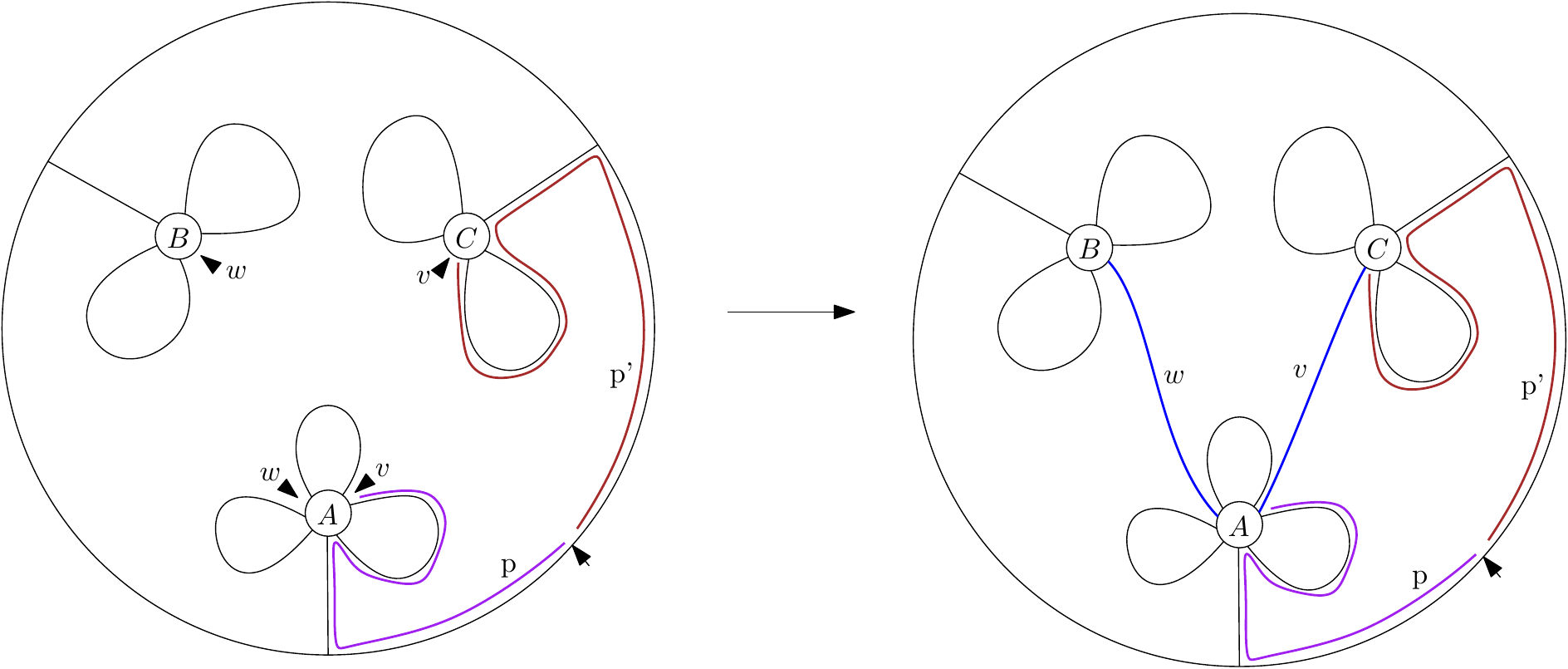}
\caption{The operation $\Phi'$. The part $p$ is in purple, and $p'$ is in brown.}\label{fig_fausse}
\end{figure}
\end{proof}

We can now prove Theorem~\ref{thm:RandomGenerationFacto}:

\begin{proof}[Proof of Theorem~\ref{thm:RandomGenerationFacto}]
Let us couple $\rFF$ and $\rHH$, so that $\rHH$ is the map associated to $\rFF$. 
As in \cref{ssec:generation}, we take $\aaa<\bbb<\ccc$ uniformly at random, and call $A$ (resp. $B$, $C$) 
the vertex of $\rHH$ that is labeled $\aaa$ (resp. $\bbb$, $\ccc$) 
by the labeling algorithm of \cref{lem:Unicellular+Relabeling}. Also, take $\{\vvv,\www\}$ uniform in $[1,n-1+2g]$.
With probability tending to $1$, $\rHH$ belongs to $\tilde{\mathcal{H}}^{g-1}_n$ and 
$(\rHH,A,B,C,\vvv,\www)$ is in $\mathcal A_{g-1,n}$.

Now, at step i) of the algorithm, we are given two factorizations $\overline F_1$ and $\overline F_2$. 
They correspond\footnote{See the proof of \cref{lem:Unicellular+Relabeling}, where a Hurwitz edge- and vertex- labeled map is associated to any sequence of transpositions; faces of the maps then correspond to cycles 
in the product of all transpositions.}
 to the maps $\Phi(\rHH,A,B,C,\vvv,\www)$ and $ \Phi'(\rHH,A,B,C,\vvv,\www)$, but we do not know which one is which
 (recall that we may have $\vvv<\www$ or $\vvv>\www$, depending on which corner of $A$ is visited first). 
 However, by Lemma~\ref{lem_bad_phi}, we know that,
 if $(\rHH,A,B,C,\vvv,\www)$ is in $\mathcal A_{g-1,n}$ (which happens with probability tending to $1$), the factorization corresponding to 
$\Phi'(\rHH,A,B,C,\vvv,\www)$ is not a factorization of a long cycle.
 Therefore, w.h.p., the factorization $\overline{\rFF}$ selected at step i)
 of the algorithm corresponds to the map $\Phi(\rHH,A,B,C,\vvv,\www)$.
 
Note however that the map $\Phi(\rHH,A,B,C,\vvv,\www)$ has two natural labelings of its vertices: the first one is inherited from $\rHH$ by keeping the same labels of the vertices, which we call {\em inherited labeling}
and the second one is defined by the labeling algorithm of Lemma~\ref{lem:Unicellular+Relabeling} ;
we call the latter the {\em natural labeling}. 
Calling $V_i$ (resp. $W_i$) the vertex with label $i$ in the inherited labeling (resp. the natural labeling),
we let $\tilde \sigma$ be the permutation such that $V_i=W_{\tilde \sigma(i)}$ for all $i$.

The factorization $\overline\rFF$ of a cycle $\zeta$ that we get at the end of step i) 
corresponds to the map with the inherited labeling.
The map $\Phi(\rHH,A,B,C,\vvv,\www)$ with the natural labeling
corresponds to the factorization ${\tilde \sigma}(\overline{\rFF})$, obtained from $\overline{\rFF}$ 
by replacing each transposition $\transpo=(i \,j)$
by $\tilde\sigma^{-1} \transpo \tilde\sigma=(\tilde\sigma(i) \,\tilde\sigma(j))$.
Note that ${\tilde\sigma}(\overline{\rFF})$ is a factorization of $\tilde\sigma^{-1} \zeta \tilde\sigma$.
On the other hand, we know that maps with their natural labelings
always correspond to factorizations of the cycle $(1\, \cdots\, n)$;
we conclude that $\tilde\sigma^{-1} \zeta \tilde\sigma=(1\, \cdots\, n)$.

Note also that $\tilde\sigma(1)=1$, since both labelings assign $1$ to the root.
These two conditions together determine uniquely $\tilde\sigma$,
in particular $\tilde\sigma$ coincides with the permutation $\sigma$
introduced in the step ii) of the algorithm of \cref{ssec:generation}.
Recall that $\Lambda(\rFF)={\sigma}(\overline{\rFF})$ by construction.
Since we have proved $\sigma=\tilde\sigma$,
we have  $\Lambda(\rFF)={\tilde\sigma}(\overline{\rFF})$;
but the latter corresponds to the map $\Phi(\rHH,A,B,C,\vvv,\www)$ with its natural labeling.
We conclude that $\Lambda(\rFF)$ is asymptotically close
to $\rF$ by invoking \cref{thm_asympto_bij}.
\end{proof}

We can in fact say more about the permutation $\sigma$ 
appearing in step ii) of the generation algorithm.
The following proposition will be useful to prove our scaling limit result.

\begin{proposition}
\label{prop:labelingalgorithm}
Let $\bm F^{g-1}_n$ be a random uniform factorization of genus $g-1$ of the cycle $(1\, \cdots\, n)$.
Then the permutation $\sigma$ appearing in the step ii) of the computation of $\Lambda(\rFF)$
is of the form 
\begin{equation}
\label{eq:sigma}
\sigma(i)=\begin{cases}
i &\text{ if $i\leq \aaa'$ or $i> \ccc'$;}\\
i+\ccc'-\aaa''&\text{ if $\aaa'<i\le \aaa''$;}\\
i-\aaa''+\aaa'+\ccc'-\bbb'&\text{  if $\aaa''< i\le \bbb'$;}\\
i-\bbb'+\aaa'&\text{  if $\bbb'< i\leq \ccc'$,}
\end{cases},\end{equation}
for some $\aaa'=\aaa+\O_P(1)$, $\aaa''=\aaa+\O_P(1)$, $\bbb'=\bbb+\O_P(1)$ and $\ccc'=\ccc+\O_P(1)$.
\end{proposition}

\begin{proof}
Let $\bm H^{g-1}_n$ be the random Hurwitz map associated to $\bm F^{g-1}_n$.
As seen in the proof of \ref{thm:RandomGenerationFacto},
w.h.p. the factorization $\Lambda(\rFF)$ is well-defined and corresponds to the map $\Phi^*(\bm H^{g-1}_n)$.
We recall also that $\Phi^*(\bm H^{g-1}_n)$ has two labelings of its vertices:
the  labeling $(V_i)_{1 \le i \le n}$ inherited from $\bm H^{g-1}_n$ and the natural one $(W_i)_{1 \le i \le n}$.

Consider first the exploration of the unique face of $\bm H^{g-1}_n$,
leading to the labeling $(V_i)_{1 \le i \le n}$. We define the following quantities:
$\aaa'$ (resp. $\aaa''$, $\bbb'$ and $\ccc'$) is the last label assigned before crossing 
the $\vvv$-corner of $A$, (resp. the $\www$-corner of $A$, the $\vvv$-corner of $B$
and the $\www$-corner of $C$).

Now it is directly seen (e.g. on Figure~\ref{fig_bij})
 that the exploration of $\Phi(\bm H^{g-1}_n,A,B,C,\vvv,\www)$,
which leads to the labeling $(W_i)_{1 \le i \le n}$, behaves this way:
\begin{itemize}
\item it first visits the special corners of $V_i$ for $i\in [1,\aaa']$ (in that order),
\item then goes along the edge labeled $\vvv$ from $A$ to $B$,
\item then visits the special corners of $V_i$ for $i\in [\bbb'+1,\ccc']$ (in that order),
\item then goes along the edge labeled $\www$ from $C$ to $A$,
\item then visits the special corners of $V_i$ for $i\in [\aaa''+1,\bbb']$ (in that order),
\item then goes along the edge labeled $\vvv$ from $B$ to $A$,
\item then visits the special corners of $V_i$ for $i\in [\aaa'+1,\aaa'']$ (in that order),
\item then goes along the edge labeled $\www$ from $A$ to $C$
\item and finally visits the special corners of $V_i$ for $i\in [\ccc'+1,n]$ (in that order).
\end{itemize}
Recalling that $\sigma$ is the unique permutation such that $V_i=W_{\sigma(i)}$ for all $i \le n$,
we get that $\sigma$ is indeed given by \eqref{eq:sigma}.
Now, by Lemma~\ref{lem_small_desc_trees}, the trees attached to $A$, $B$ and $C$ have size $\O_P(1)$, implying $\aaa'=\aaa+\O_P(1)$, $\aaa''=\aaa+\O_P(1)$, $\bbb'=\bbb+\O_P(1)$ and $\ccc'=\ccc+\O_P(1)$.
\end{proof}

\section{Background on trees and laminations}
\label{sec:background}

In this section, we recall some known facts about the connection between trees and laminations. In particular, we provide a rigorous framework for the convergence of Theorems \ref{thm:main} and \ref{thm:mainprocess}.

\subsection{Duality between trees and laminations}
\label{ssec:background_lamination}

We first  define \textit{plane trees}, following Neveu's formalism \cite{Nev86}. First, let $\N^* = \left\{ 1, 2, \ldots \right\}$ be the set of all positive integers, and $\mathcal{U} = \cup_{n \geq 0} (\N^*)^n$ be the set of finite sequences of positive integers, with $(\N^*)^0 = \{ \emptyset \}$ by convention.

By a slight abuse of notation, for $k \in \Z_+$, we write an element $u$ of $(\N^*)^k$ by $u=u_1 \cdots u_k$, with $u_1, \ldots, u_k \in \N^*$. For $k \in \Z_+$, $u=u_1\cdots u_k \in (\N^*)^k$ and $i \in \Z_+$, we denote by $ui$ the element $u_1 \cdots u_ki \in (\N^*)^{k+1}$. A plane tree $T$ is formally a subset of $\mathcal{U}$ satisfying the following three conditions:

(i) $\emptyset \in T$ (the tree has a root); 

(ii) if $u=u_1\cdots u_n \in T$, then, for all $k \leq n$, $u_1\cdots u_k \in T$ (these elements are called ancestors of $u$, and the set of all ancestors of $u$ is called its ancestral line; $u_1 \cdots u_{n-1}$ is called the \textit{parent} of $u$); 

(iii) for any $u \in T$, there exists a nonnegative integer $k_u(T)$ such that, for every $i \in \N^*$, $ui \in T$ if and only if $1 \leq i \leq k_u(T)$ ($k_u(T)$ is called the number of children of $u$, or the outdegree of $u$).

See an example of a plane tree on Fig. \ref{fig:arbconlam}, left. The elements of $T$ are called \textit{vertices}, and we denote by $|T|$ the total number of vertices in $T$. The height of $T$, which is the maximum $k \geq 0$ such that $(\N^*)^k \cap T$ is nonempty, is denoted by $H(T)$. In the sequel, by tree we always mean plane tree unless specifically mentioned.

The \textit{lexicographical order} $\prec$ on $\mathcal{U}$ is defined as follows:  $\emptyset \prec u$ for all $u \in \mathcal{U} \backslash \{\emptyset\}$, and for $u,w \neq \emptyset$, if $u=u_1u'$ and $w=w_1w'$ with $u_1, w_1 \in \N^*$, then we write $u \prec w$ if and only if $u_1 < w_1$, or $u_1=w_1$ and $u' \prec w'$. The lexicographical order on the vertices of a tree $T$ is the restriction of the lexicographical order on $\mathcal{U}$; for every $0 \leq k \leq |T|-1$ we write $v_k(T)$ for the $(k+1)$-th vertex of $T$ in the lexicographical order.

We do not distinguish between a finite tree $T$, and the corresponding planar graph where each vertex is connected to its parent by an edge of length $1$, in such a way that the vertices at the same distance from the root are sorted from left to right in lexicographical order. Finally, for any two vertices $x,y \in T$, we denote by $\llbracket x, y \rrbracket$ the unique path between $x$ and $y$ in the tree $T$.

\emph{The contour function of a tree}

Let $T_n$ be a plane tree with $n$ vertices. We first define its contour function $(C_t(T_n))_{0 \leq t \leq 2n}$ as follows: imagine that a particle explores the tree from left to right at unit speed along the edges, starting from the root and . Then, for $0 \leq t \leq 2n-2$, $C_t(T_n)$ denotes its distance from the root. By convention, we set $C_t(T_n)=0$ for $2n-2 \leq t \leq 2n$. Notice that this contour function is nonnegative, continuous, and goes back to $0$ at time $2n-2$.

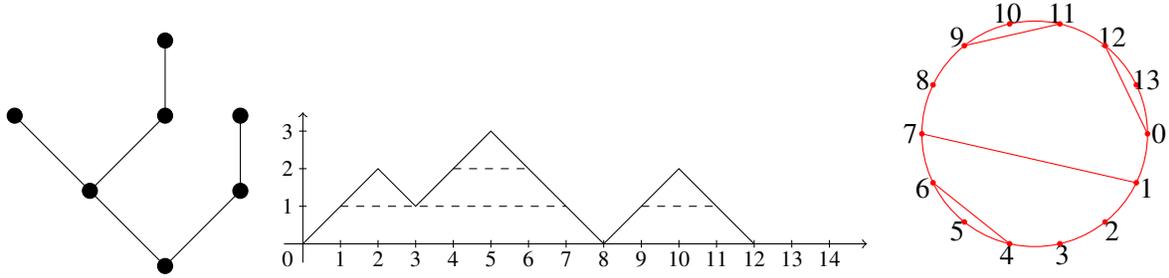
\begin{figure}[!ht]
\begin{tabular}{c c c}
\begin{tikzpicture}
\draw (1,2) -- (1,1) -- (0,0) -- (-1,1)--(0,2)--(0,3) (-1,1) -- (-2,2);

\draw[fill=black] (0,0) circle (.1);

\draw[fill=black] (1,1) circle (.1);

\draw[fill=black] (-1,1) circle (.1);

\draw[fill=black] (0,2) circle (.1);

\draw[fill=black] (0,3) circle (.1);

\draw[fill=black] (-2,2) circle (.1);

\draw[fill=black] (1,2) circle (.1);
\end{tikzpicture}
&
\begin{tikzpicture}[scale=.5, every node/.style={scale=0.7}]
\draw (0,0) -- (1,1) -- (2,2) -- (3,1) -- (4,2) -- (5,3) -- (6,2) -- (7,1) -- (8,0) -- (9,1) -- (10,2) -- (11,1) -- (12,0);
\draw[->] (0,-.5) -- (0,3.5);
\draw[->] (-.5,0) -- (15,0);
\draw (1,.1) -- (1,-.1);
\draw (1,-.4) node{1};
\draw (2,.1) -- (2,-.1);
\draw (2,-.4) node{2};
\draw (3,.1) -- (3,-.1);
\draw (3,-.4) node{3};
\draw (4,.1) -- (4,-.1);
\draw (4,-.4) node{4};
\draw (5,.1) -- (5,-.1);
\draw (5,-.4) node{5};
\draw (6,.1) -- (6,-.1);
\draw (6,-.4) node{6};
\draw (7,.1) -- (7,-.1);
\draw (7,-.4) node{7};
\draw (8,.1) -- (8,-.1);
\draw (8,-.4) node{8};
\draw (9,.1) -- (9,-.1);
\draw (9,-.4) node{9};
\draw (10,.1) -- (10,-.1);
\draw (10,-.4) node{10};
\draw (11,.1) -- (11,-.1);
\draw (11,-.4) node{11};
\draw (12,.1) -- (12,-.1);
\draw (12,-.4) node{12};
\draw (13,.1) -- (13,-.1);
\draw (13,-.4) node{13};
\draw (14,.1) -- (14,-.1);
\draw (14,-.4) node{14};
\draw (.1,1) -- (-.1,1);
\draw (-.4,1) node{1};
\draw (.1,2) -- (-.1,2);
\draw (-.4,2) node{2};
\draw (.1,3) -- (-.1,3);
\draw (-.4,3) node{3};
\draw (-.4,-.4) node{0};
\draw [dashed](1,1) -- (7,1);
\draw [dashed](4,2) -- (6,2);
\draw [dashed](9,1) -- (11,1);
\end{tikzpicture}
&
\begin{tikzpicture}[scale=1.5, every node/.style={scale=.9}]
\foreach \i in {0,...,13}
{
\draw[auto=right] ({1.1*cos(-(\i)*360/14)},{1.1*sin(-(\i)*360/14)}) node{\i};
\draw[red,fill=red] ({cos(-(\i-1)*360/14)},{sin(-(\i-1)*360/14)}) circle (.02);
}
\draw[red] (0,0) circle (1);
\draw[red] ({cos(-360*0/14)},{sin(-360*0/14)}) -- ({cos(-360*12/14)},{sin(-360*12/14)});
\draw[red] ({cos(-360*7/14)},{sin(-360*7/14)}) -- ({cos(-360*1/14)},{sin(-360*1/14)});
\draw[red] ({cos(-360*4/14)},{sin(-360*4/14)}) -- ({cos(-360*6/14)},{sin(-360*6/14)});
\draw[red] ({cos(-360*9/14)},{sin(-360*9/14)}) -- ({cos(-360*11/14)},{sin(-360*11/14)});
\end{tikzpicture}
\end{tabular}
\caption{A tree $T$, its contour function $C(T)$ and the associated lamination $\bL(T)$.}
\label{fig:arbconlam}
\end{figure}

\bigskip

\emph{The lamination associated to a tree}

Starting from the contour function of $T_n$, we can define a lamination. Let $\emptyset$ be the root of $T_n$. For any vertex $V \in T_n \backslash \{ \emptyset \}$ that is not the root, define $g_V$ (resp. $d_V$) the first (resp. last) time at which $V$ is visited by the contour function of $T_n$, and set $c(V) := [e^{-2\pi i g_V/(2n)},e^{-2\pi i d_V/(2n)}]$, a chord of the unit disk. Then, the lamination $\bL(T_n)$ is defined as 
\begin{align*}
\bL(T_n) := \bS^1 \cup \bigcup_{V \in T_n \backslash \{\emptyset\}} c(V),
\end{align*} 
where we recall that $\bS^1$ denotes the unit circle. By continuity of the contour function, $\bL(T_n)$ is a closed subset of the disk, and hence indeed a lamination. One can see an example of a tree along with its contour function and its lamination on Fig. \ref{fig:arbconlam}.

\bigskip

\emph{The lamination-valued process associated to a labelled tree}

Let us now consider a labeling of the edges of $T_n$, from $1$ to $n-1$. We construct from it a discrete lamination-valued process $(\bL_k(T_n))_{0 \leq k \leq n-1}$, as follows. For any $0 \leq k \leq n-1$, denote by $V_k$ the vertex of $T_n$ such that the edge from $V_k$ to its parent is labelled $k$ (by an abuse of notation, we say that $V_k$ is itself labelled $k$). We define, for $t \geq 0$:
\begin{align*}
\bL_t(T_n) := \bS^1 \cup \bigcup_{i=1}^{(n-1) \wedge \lfloor t\rfloor} c(V_i).
\end{align*} 

In particular, for any $t \geq n-1$, $\bL_t(T_n)=\bL(T_n)$.

\bigskip

\emph{Construction of a tree from a continuous function}

We provide here the "reverse" construction of a tree from a continuous function.
Let us take $f: [0,1] \rightarrow \R_+$ be a continuous function satisfying $f(0)=f(1)=0$. Then, one can construct a tree $\Tree(f)$ as follows: define a pseudo-distance $d_f$ on $[0,1]$ as:
\begin{align*}
\forall x,y \in [0,1], d_f(x,y) = f(x)+f(y)-2 \inf_{u \in [x\wedge y, x \vee y]} f(u).
\end{align*}
In particular, $d_f(0,1)=0$.
This defines an equivalence relation $\sim_f$ on $[0,1]$:
\begin{align*}
\forall x,y \in [0,1], x \sim_f y \Leftrightarrow f(x)=f(y)=\inf_{u \in [x\wedge y,x \vee y]} f(u).
\end{align*}
One can check that $\sim_f$ is indeed an equivalence relation. In particular $0 \sim_f 1$. Finally, define $\Tree(f)$ as
\begin{align*}
\Tree(f)=[0,1] / \sim_f.
\end{align*}

Therefore, $d_f$ induces a distance on $\Tree(f)$, which we still denote by $d_f$ by a small abuse of notation. Furthermore, as $f$ is continuous, $\Tree(f)$ endowed with this distance is compact. For any $x,y \in T$, we denote by $\llbracket x, y \rrbracket$ the unique path from $x$ to $y$ in $T$.

Let us immediately define some important notions about trees. We say that an equivalence class $u \in \Tree(f)$ is a leaf of the tree is an equivalence class $x$ such that $\Tree(f) \backslash \{u \}$ is connected. The volume measure $h$, or mass measure on $\Tree(f)$, is defined as the projection on $\Tree(f)$ of the Lebesgue measure on $[0,1]$. Finally, the length measure $\ell$ on $\Tree(f)$, supported by the set of non-leaf points, is the unique $\sigma$-finite measure on this set such that, for $u, v \in \Tree(f)$, $\ell(\llbracket u, v \rrbracket) =d_f(u,v)$. See \cite{ALW17} for further details about this length measure. This $\sigma$-finite measure expresses the intuitive notion of length of a branch in the tree. For any tree $T$, any $x \in T$, define $\theta_x(T)$ as the set of points $y$ in $T$ such that $x \in \llbracket \emptyset,y \rrbracket$, where $\emptyset$ denotes the root of $T$. We then define $|\theta_x(T)|$ as the mass measure of this subtree.

Finally, observe that, for a finite planar tree $T$ of size $|T|$, the metric space $\Tree(C_{2|T|\cdot}(T))$ is simply obtained from $T$ by replacing all edges by a line segment of length $1$.
\bigskip

\emph{Construction of a lamination-valued process from a continuous function}

As in the previous paragraph, let $f$ be a continuous nonnegative function such that $f(0)=f(1)=0$. Let $\cEG(f)$ be its epigraph, that is, 
\begin{align*}
\cEG(f) := \{ (s,t) \in \R^2, 0 \leq s \leq 1, 0 < t < f(s) \}.
\end{align*}

We now define a Poisson point process $\cN$ on $\cEG(f) \times \R_+$, of intensity 
\begin{align*}
\frac{ds \, dt}{d(s,t,f)-g(s,t,f)} \mathds{1}_{(s,t) \in \cEG(f)} \times dr,
\end{align*}
where, for $(s,t) \in \cEG(f)$, $g(s,t)= \sup \{u<s, f(u)=t\}$ and $d(s,t)=\inf \{u>s, f(u)=t\}$.

This allows us to define a lamination-valued process $(\bL_c(f))_{c \geq 0}$ as follows. For any $(s,t) \in \cEG(f)$, define the chord $c(s,t)$ as
\begin{align*}
c(s,t) := \left[ e^{-2\pi i g(s,t,f)}, e^{-2\pi i d(s,t,f)} \right].
\end{align*}

Then, for $c \geq 0$, define $\bL_c(f)$ as 
\begin{align*}
\bL_c(f) := \overline{\bS^1 \cup \bigcup_{\substack{((s,t),x) \, \in \, \cN \\ x \leq c}} c(s,t)}.
\end{align*}

Define also
\begin{align*}
\bL_\infty(f) := \overline{\bS^1 \cup \bigcup_{(s,t) \in \cEG(f)} c(s,t)}.
\end{align*}

There is a natural projection map from $\cEG(f)$ to $\Tree(f)$,
associating to any $(s,t)$ the equivalence class of $g(s,t)$.
Through this mapping,  the Poisson point process $\cN(f)$ 
is sent to a Poisson point process $\cP(\Tree(f))$ on $\Tree(f) \times \R_+$, of intensity $d\ell \times dr$.
For any $c \geq 0$, we set 
$$\cP_c(\Tree(f)) := \{ u \in \Tree(f), \exists (u,x) \in \cP(\Tree(f)), x \leq c \}$$
 the set of points appearing on $\Tree(f)$ before time $c$.

\subsection{Random trees}

\text{  }

\emph{Aldous' Continuum Random Tree}

We introduce here the random tree $\cT_\infty := \Tree(\mathbbm{e})$, constructed from the standard Brownian excursion $\mathbbm{e}$. This tree, which arises in the literature as the scaling limit of various models of random trees, is a random compact metric space, which is also called the continuum random tree (CRT). See Fig. \ref{fig:arbconlamcontinuous}, left for a simulation of $\cT_\infty$. A famous property of this tree is that its mass measure $h$ is supported by the set of its leaves.

\bigskip

\emph{Galton-Watson trees}

Let $\mu$ be a probability distribution on $\Z_+$. We assume in what follows that $\mu$ is critical (that is, with expectation $1$) and has finite variance. A $\mu$-Galton-Watson tree (or $\mu$-GW tree) is a variable $T$ taking its values in the set of finite plane trees such that, for any tree $\tau$, $\P(T=\tau) = \prod_{V \in \tau} \mu_{k_V(\tau)}$, where we recall that $k_V(\tau)$ denotes the number of children of the vertex $V$. For convenience, we will always assume that $\mu_i>0$ for all $i \in \Z_+$, although our results generalize easily if we remove this assumption.

The asymptotic behaviour of large Galton-Watson trees has been extensively studied since the pioneering work of Aldous \cite{aldous1993}. In particular, Aldous shows  the convergence of discrete Galton-Watson trees after renormalization, as their size grows, towards the CRT:

\begin{theorem}[Aldous \cite{aldous1993}, Le Gall \cite{LG05}]
\label{thm:aldleg}
Let $\mu$ be a critical distribution on $\Z_+$ with finite variance $\sigma^2$. For $n \in \N$, denote by $T_n$ a $\mu$-Galton-Watson tree conditioned to have $n$ vertices. Then, the following holds in distribution for the Gromov-Hausdorff distance:
\begin{align*}
\frac{\sigma T_n}{\sqrt{2 n}} \underset{n \rightarrow \infty}{\overset{(d)}{\rightarrow}} \cT_\infty.
\end{align*}

Moreover, jointly with this convergence, the following holds for the $J1$ Skorokhod topology on $\mathcal D([0,1], \R_+)$: letting $\tilde C_T(t)=\frac{\sigma C_{2nt}(T_n)}{\sqrt{2 n}}$,
\begin{align*}
\tilde C_T \underset{n \rightarrow \infty}{\overset{(d)}{\rightarrow}} \mathbbm{e},
\end{align*}
where $\mathbbm{e}$ has the law of the standard Brownian excursion.
\end{theorem}

\subsection{Convergence of lamination processes of Galton-Watson trees}
\label{ssec:chords_chords}
From now on, we assume that $\mu$ has finite variance;
in fact the only case of interest for this article is when $\mu$
is a $\Poisson(1)$ distribution.
Then, jointly with the convergence of Theorem \ref{thm:aldleg}, 
the following was proven in \cite[Theorem $3.3$ and Proposition $4.3$]{thevenin2019geometric}:

\begin{theorem}
\label{thm:cv_lamination_G0}
In distribution, for the $J1$ Skorokhod topology on $\mathcal D(\R_+,\SetSieve(\D))$:
\begin{align*}
\left(\bL_{(\sigma/\sqrt{2})c\sqrt{n}}(T_n)\right)_{c \geq 0} \underset{n \rightarrow \infty}{\overset{(d)}{\rightarrow}} \left(\bL_c(\mathbbm{e})\right)_{c \geq 0}.
\end{align*}
\end{theorem}

\begin{figure}
\includegraphics[scale=0.4]{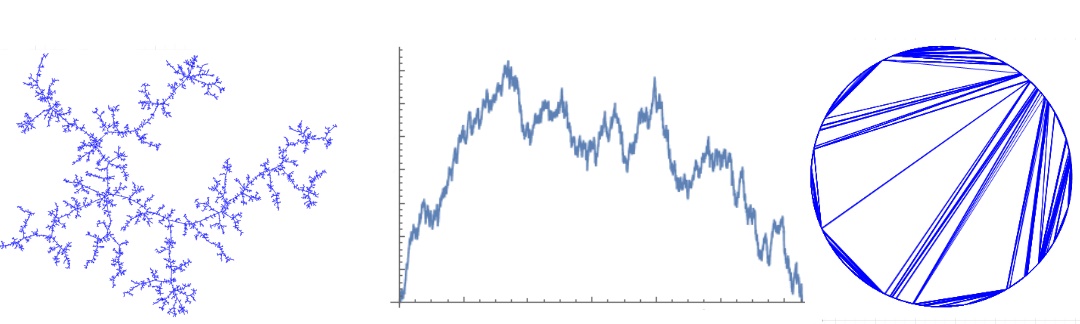}
\caption{An approximation of $\left( \cT_\infty, \be, \bL_\infty^{(2)} \right)$.}
\label{fig:arbconlamcontinuous}
\end{figure}

In what follows, we will denote by $\left(\bL_c^{(2)}\right)_{c \geq 0}$ the limiting process $\left(\bL_c(\mathbbm{e})\right)_{c \geq 0}$. In particular, $\bL_\infty^{(2)}:=\lim_{c \rightarrow \infty} \bL_c^{(2)}$ is Aldous' Brownian triangulation, informally presented in the introduction. 
See Fig. \ref{fig:arbconlamcontinuous}, right for an approximation of $\bL_\infty^{(2)}$.

\emph{The chord-to-chord correspondence:} 
We will need a more concrete version of the convergence in \cref{thm:cv_lamination_G0},
stating that for any $M>0$, any chord $c$ in the limit $\bL_M^{(2)}$, there is a sequence of chords
in $\bL_{(\sigma/\sqrt{2})M\sqrt{n}}(T_n)$ appearing roughly at the same time as $c$
and converging to $c$.
We refer to this as the chord-to-chord correspondence and we present it now formally.
This section follows \cite{thevenin2019geometric} (see in particular Proposition 2.5 there).

\bigskip

For an integer $s \geq 1$ and $k \in [ 0,s-1 ]$, we denote by $x_k$ the arc $(e^{-2\pi i\, k/s}, e^{-2\pi i\, (k+1)/s})$, and set $I_s := \{ x_k, 0 \leq k \leq s-1 \}$.
We also split the time interval $[0,M]$ into pieces (forgetting finitely many points):
 for any $D \in \Z_+, D \geq 1$,
let $R^M_D := \{ [jM/D, (j+1)M/D], j \in [ 0,D-1] \}$.

Fix $M>0, \eta>0, s,D \geq 1$ and an integer $K \geq 0$. 
In what follows (including in union and summation index),
 $Y=(y_1, \ldots, y_K)$ and $Z=(z_1, \ldots, z_K)$ are lists in $I_s^K$,
 and $R=(R_1,\dots,R_K)$ is a list of time windows in $(R^M_D)^K$. 
 Triples $(Y,Z,R)$ obtained from each other by the simultaneous action of a permutation
 in the symmetric group $S_K$ on $Y$, $Z$ and $R$ are considered identical.
 We now define the event $E^M_{Y,Z,R}(\kL)$, 
 for a nondecreasing lamination-valued process $\kL := (L_r)_{r \geq 0}$.

$E^M_{Y,Z,R}(\kL): L_M \text{ has exactly } K \text{ chords of length } >\eta, \text{ which can be indexed so that the  } i-\text{th one connects the arcs } y_i \text{ and } z_i \text{ and has appeared in the process during the time window } R_i.$

The following proposition is a consequence of the definition 
of the Skorokhod topology on $\mathcal D(\R_+,\SetSieve(\D))$. Roughly speaking, if two lamination-valued processes have their large chords close to each other, then the processes are close to each other for the Skorokhod topology.

\begin{proposition}
\label{prop:constant}
Fix $\eta,s,D>0$. Then there exists a (deterministic) constant $C(\eta,s,D)$ depending only on $\eta,s,D$ such that:
\begin{itemize}
\item[(i)] for any two lamination-valued processes $\kL_1:=(L^1_c)_{c \geq 0}, \kL_2:=(L^2_c)_{c \geq 0}$, any $M>0$, any $K \geq 0$, any $(Y,Z,R) \in I_s^K \times I_s^K \times R_D^M$, conditionally on $E^M_{Y,Z,R}(\kL^1)$, $E^M_{Y,Z,R}(\kL^2)$:
\begin{align*}
d_{Sk} \left( (L^1_c)_{0\leq c \leq M}, (L^2_c)_{0\leq c \leq M} \right) \leq C(\eta,s,D),
\end{align*}
where we recall that $d_{Sk}$ denotes the Skorokhod distance on $\mathcal{D}([0,M], \SetSieve(\D))$.

\item[(ii)]$C(\eta,s,D) \rightarrow 0$ as $\eta^{-1},s,D$ all go to $+ \infty$.
\end{itemize}
\end{proposition}

In what follows, let $T_n$ be a $\mu$-GW tree conditioned to have $n$-vertices. 
We now let $\kL_n := (\bL_{(\sigma/\sqrt{2}) c \sqrt{n}}(T_n))_{c \geq 0}$ and $\kL := (\bL_c^{(2)})_{c \geq 0}$.

\begin{proposition}[Chord-to-chord correspondence]
\label{prop:chordtochord}
The following holds:
\begin{itemize}
\item[(i)] For any $K>0$, $Y,Z,R \in I_s^K \times I_s^K \times  R_D^M$: $\P\left( E^M_{Y,Z,R}(\kL_n) \right) \underset{n \rightarrow \infty}{\rightarrow} \P\left( E^M_{Y,Z,R}(\kL) \right)$.
\item[(ii)] $\P \left( \bigcup_{K \geq 0} \bigcup_{Y,Z,R} E^M_{Y,Z,R}(\kL) \right)=1$.
\item[(iii)] The events $\left(E^M_{Y,Z,R}(\kL), (Y,Z,R) \in I_s^K \times I_s^K \times R_D^M \right)$ are almost surely disjoint.
\end{itemize}
\end{proposition}

\begin{proof}
Item (i) is proved in \cite[Proposition $2.5$]{thevenin2019geometric}, while (ii) and (iii) are clear since $\bL_M^{(2)}$ has a.s. a finite number of chords of length $>\eta$.
\end{proof}

The two above propositions allow us to consider a particular coupling 
between the discrete processes $\kL_n$ and the limit $\kL$,
 depending on thresholds $\epsilon,\delta>0$.
 This coupling will be useful in the proofs in the next sections,
 so we describe it now.
 
 Fix $\epsilon,\delta>0$ and choose $\eta,s,D$ such that $C(\eta,s,D)<\epsilon$
 (they exist by \cref{prop:constant} (ii)).
 By \cref{prop:chordtochord} (ii),
 there exists $K_0$ such that
 \begin{equation}
   \P \left( \bigcup_{K=0}^{K_0} \bigcup_{Y,Z,R} E^M_{Y,Z,R}(\kL) \right) \ge 1-\frac{\delta}2.
   \label{eq:DefK0}
\end{equation}
 Now, there is a finite number of events $E^M_{Y,Z,R}(\kL)$, with $K\le K_0$.
 Hence, by \cref{prop:chordtochord} (i),
 one can choose $N_0$ so that, for any $n \geq N_0$,
 \[\sum_{K \le K_0} \sum_{Y,Z,R} \Big| \P\left( E^M_{Y,Z,R}(\kL_n) \right) -
 \P\left( E^M_{Y,Z,R}(\kL) \right) \Big| \le \delta/2.\]
 Therefore, one can couple $\kL_n$ and $\kL$ such that,
 with probability $1-\delta/2$, for any $n \geq N_0$, $E^M_{Y,Z,R}(\kL)$ holds if and only if $E^M_{Y,Z,R}(\kL_n)$ holds
 (for all $K \le K_0$, $Y$, $Z$ and $R$).
 By \cref{prop:constant} (i), recalling \cref{eq:DefK0}, this implies that
 with probability $1-\delta$, one has
\[d_{Sk} \left( \kL^M_n , \kL^M \right) \leq C(\eta,s,D) <\eps.\]
 Furthermore, let us do the following observation:
 if both $E^M_{Y,Z,R}(\kL)$ and $E^M_{Y,Z,R}(\kL_n)$ hold 
 (for some $n$, $K$, $Y$, $Z$ and $R$),
 then for any $i \le K$, the $i$-th chords of length $>\eta$
 in $\bL_{(\sigma/\sqrt{2})M\sqrt{n}}(T_n)$ and $\bL_M^{(2)}$ are at Hausdorff distance $<2\pi/s$ 
 from each other and appear at times that differ by at most $M/D$.

\begin{remark}
Notice that it may happen that chords in the discrete process of length $>\eta$ converge to a chord in the limit that has length exactly $\eta$. However, at $\eta$ fixed, this happens with probability $0$ as, almost surely, no chord in the limiting lamination $\bL_M^{(2)}$ has size exactly $\eta$.
Similarly, the probability that $\bL_M^{(2)}$ has a chord
appearing at time exactly $jM/D$ for some $j$, or a chord of length $>\eta$
with extremity exactly $e^{-2\pi i\, k/s}$ is zero; therefore taking open or closed
time intervals or circle arcs in our discretization is irrelevant.
\end{remark}

\section{Sieves and rotations}
\label{sec:rotations}

In this section, we construct the limiting process $(\Sieve_c^g, c \geq 0)$ in genus $g$ out of the limiting process in genus $0$ by applying rotations to it, in the way described in Section \ref{ssec:generation}.

\subsection{Framework}
\label{ssec:framework}
Let us recall the construction of the limiting sieve $\Sieve_\infty^g$ of \cref{thm:main}.
We start from the Brownian lamination $\bL_\infty^{(2)} := \Sieve_\infty^0$, and $g$ independent uniform
random sets $(\{A_i,B_i,C_i\})_{1 \le i \le g}$ on the unit circle.
We take the convention that $(1,A_i,B_i,C_i)$ appear in this order clockwise on the unit circle.
For simplicity, recalling the definition of the rotation operations in Section \ref{ssec:generation}, we write $R_i=R_{\{A_i,B_i,C_i\}}$ and $R_j^h=R_h \circ \cdots \circ R_j$
for $1 \le j <h\le g$ and $R^g=R_1^g$.
We then have by definition $\bm \Sieve^g_\infty=R^g(\bL_\infty^{(2)})$.

For each $i\le g$, the points obtained by applying to $A_i,B_i,C_i$ 
later rotations will play a particular role:
we therefore set 
\[
A'_i=R^g_{i}(A_i) ;\qquad
B'_i=R^g_{i}(B_i) ;\qquad
C'_i=R^g_{i}(C_i).
\]

Other points of interest are the points sent to $A_i$, $B_i$, $C_i$
by the first $i-1$ rotations, namely we define
\[
A''_i=(R^{i-1}_{1})^{-1}(A_i) ;\qquad
B''_i=(R^{i-1}_{1})^{-1}(B_i) ;\qquad
C''_i=(R^{i-1}_{1})^{-1}(C_i).
\]
Finally, throughout the paper we denote by $E$ the
set $\bigcup_{i=1}^g \{ A''_i, B''_i, C''_i \}$.

For the proof of our main result (\cref{thm:mainprocess}),
it will be important to keep track of the Hausdorff distance between the images of chords in a sieve after some rotations. We start with a remark.

\begin{remark}
\label{rem:continuity}
The operators $R_{\{A,B,C\}}$
 (and thus more generally the $R_j^h$)
are not continuous.
Indeed let $c=[P,Q]$ be a chord of length $\eps>0$ around $A$ 
(i.e. $A$ belongs to the smallest of the two arcs $\arc{PQ}$ or $\arc{QP}$;
 we will use this terminology of {\em chord around a point}
throughout the paper).
If $\eps$ is sufficiently small, one of its extremities, say $P$, belongs
to $\arc{CA}$ and the other, say $Q$, to $\arc{AB}$.
By construction, $R_{\{A,B,C\}}(P)=P$ is at distance $<\eps$ from $A$,
while  $R_{\{A,B,C\}}(Q)$ belongs to $R_{\{A,B,C\}}(\arc{AB})=\arc{DC}$,
where $D=R_{\{A,B,C\}}(A) \ne A$.
Moreover, $R_{\{A,B,C\}}(Q)$ is at distance $<\eps$ from $D$.
The image of $c$ is therefore at Hausdorff distance $<\eps$
from the chord $[A,D]$.
We note that adding $c$ can only move a sieve 
containing the circle by $\eps$ in the Hausdorff space, 
while adding $[A,D]$ is a macroscopic change.
Similarly, a small chord around $B$ will be transformed to a chord
close to $[C,A]$, while small chords around $C$ are mapped to
chords closed to $[D,C]$.
This shows that $R_{\{A,B,C\}}$ is not continuous.

Furthermore, the Brownian triangulation $\bL^{(2)}_\infty$ almost surely contains short chords
around almost all points (see proof of \cref{lem:Triangles_In_XgInfty} below),
 so that $R_{\{A,B,C\}}$ is not even continuous
on a set of measure $1$ with respect to the distribution of the Brownian triangulation.
\end{remark}\medskip

In good cases, we can nevertheless control Hausdorff distances
between chords after rotation.
The following lemma provides such a control. In this lemma, $d$ denotes the usual Euclidean distance in the plane.
\begin{lemma}
\label{lem:chorddistances}
\begin{itemize}
\item[(i)] For any chords $c=[P,Q], c'=[P',Q']$, 
\begin{align*}
d_H(c,c') \leq \max \{ d(P,P'), d(Q,Q') \}.
\end{align*}
\item[(ii)]
Fix $\eta>0$ and let $c:=[P,Q]$ be a chord in the disk of length $>\eta$. Then, for $\epsilon>0$ small enough (depending on $\eta$), for any chord $c'$ such that $d_H(c,c')<\epsilon$, one can label the endpoints of $c'$ by $P'$ and $Q'$ in a unique way so that $d(P,P')<d(Q,P')$ and $d(Q,Q') < d(Q,P')$
(informally, this corresponds to pairing the closest of these endpoints).
 Furthermore, there exists a constant $s_\eta > 0$ depending only on $\eta$ such that
\begin{equation}
\label{eq:compare_distance}
d_H(c,c') \leq \max \{ d(P,P'), d(Q,Q') \} \leq s_\eta d_H(c,c').
\end{equation}

\item[(iii)] Let $g \geq 1$ and $(A_i,B_i,C_i)_{1 \leq i \leq g}$ be $g$ triples of points on the disk so that all $3g$ points are distinct. 
Recall that $E$ denotes the set $\bigcup_{i=1}^g \{ A''_i, B''_i, C''_i \}$. Let $c:=[P,Q]$ be a chord of the disk of length $>\eta$ and let $c':=[P',Q']$ as in (ii). Assume that $P$ and $P'$ are in the same connected component of $\bS^1 \backslash E$, and that $Q$ and $Q'$ are also in the same component (possibly different from the one containing $P$ and $P'$). Then:
\begin{align*}
d(R^g(P),R^g(P'))=d(P,P'),\quad d(R^g(Q),R^g(Q'))=d(Q,Q').
\end{align*}
In particular, by (i) and (ii), 
\begin{align*}
d_H(R^g(c),R^g(c'))\leq \max \{ d(P,P'), d(Q,Q') \} \leq s_\eta d_H(c,c').
\end{align*}
\end{itemize}
\end{lemma}

\begin{proof}[Proof of Lemma \ref{lem:chorddistances}]
Remark first that (iii) is straightforward by definition of the rotation operation. Thus, we only have to prove (i) and (ii). 

We start by proving (i).
Let $X$ be a point in $c$. Then, identifying points and their Euclidean coordinates,
we have $X=\alpha P +(1-\alpha) Q$ for some $\alpha \in [0,1]$.
Set $X'=\alpha P' +(1-\alpha) Q'$, which is a point of the chord $c'$.
Then, using the fact that Euclidean distance is associated with a norm $\Vert \cdot \Vert$, we get:
\begin{multline*}
d(X,c') \le \Vert X - X'\Vert  \le \big\Vert  \alpha (P-P') +(1-\alpha) (Q-Q') \big\Vert \\
\le \alpha \Vert P-P'\Vert  +(1-\alpha) \Vert Q-Q'\Vert   \le  \max \{ d(P,P'), d(Q,Q') \}.
\end{multline*}
By symmetry, we get that for all $Y'$ in $c'$, one has 
\[d(Y',c) \le \max \{ d(P,P'), d(Q,Q') \},\]
concluding the proof of (i).

Let us now prove (ii).
For convenience and without loss of generality, we set $P=e^{i \alpha}, Q=e^{-i\alpha}$, where $\alpha \in [0,\pi/2]$, and set $\ell:=2\sin \alpha$ the Euclidean length of $c$. 
We want to show that for $\eps>0$ the condition $d_H(c,c')<\eps$ forces
one of the extremity of $c'$ to be in the upper half-circle and the other in the lower half circle.
First, we remark that $d(1,c)=1-\sqrt{1-\ell^2/4} \geq 1-\sqrt{1-\eta^2/4}$, and that $d(-1,c)\ge d(1,c)$. 
Hence, if $\epsilon<1-\sqrt{1-\eta^2/4}$, the extremities of $c'$ are different from $1$ and $-1$. 
In addition, if one assumes that both are in the upper half-circle,
 then $d(Q,c')>\ell/2\geq \eta/2 \geq 1-\sqrt{1-\eta^2/4}$. 
Thus, for $\epsilon$ small enough, $c'$ has one extremity in each half-circle,
and one can define $P'$ and $Q'$ unambiguously.

Let us now prove the right inequality in \cref{eq:compare_distance}. 
Observe that it immediately follows from the two following statements,
 where we write $P'=e^{\pi i  \beta}$ with $\beta \in  (0,\pi)$.
\begin{itemize}
\item[(a)] if $0 < \beta \le \alpha$ or $\pi-\alpha \leq \beta < \pi$,
  then $d(P',c) \geq K(\eta) d(P',P)$ for some constant $K(\eta)$ (independent of $\beta$).
\item[(b)] if $\alpha < \beta \leq \pi - \alpha$ then $d(P',c)=d(P,P')$.
\end{itemize}
Item (b) is trivial, so we focus on (a).
In this case $P'$ has an orthogonal projection on $c$, say $U$,
so that $d(P',c)=d(P',U)$. 
Call $V=(0,\sqrt{1-\ell^2/4})$ (resp. $W$) the intersection of $[P,Q]$ (resp. $[P,P']$) 
and the horizontal axis.
We note that $d(P,V)=\ell/2 \le 1$, while $d(V,W) \ge 1-\sqrt{1-\ell^2/4} \ge 1-\sqrt{1-\eta^2/4}$ 
since $W$ is outside the unit disk.
We immediately get by Thales' theorem:
\begin{align*}
d(P,U) \le \frac{d(P,U)}{d(P,V)}=\frac{d(P',U)}{d(V,W)} \le \frac{d(P',U)}{1-\sqrt{1-\eta^2/4}},
\end{align*}
On the other hand, we have $d(P',U)^2+d(P,U)^2=d(P,P')^2$. 
Both identities together prove (a) (recalling that $d(P',c)=d(P',U)$).
\end{proof}

\subsection{Proof of \cref{thm:cvginfty}}
\label{ssec:ProofGToInfty}
In this section, we prove that $\bm \Sieve^g_\infty$
converges to the disk $\D$ as $g$ tends to $+\infty$.
We start with a lemma.
\begin{lemma}
\label{lem:Triangles_In_XgInfty}
Almost surely, for each $i \le g$,
the three chords $[A'_i,B'_i]$, $[B'_i,C'_i]$ and $[C'_i,A'_i]$ belong to $\bm \Sieve^g_\infty$.
\end{lemma}
\begin{proof}
We prove in fact a stronger version: fix $i \le g$,
for any $\delta>0$ and $\eps$ sufficiently small
(with a threshold depending on $\delta$), 
there exists a (random) time $M$ such that
$R^g(\bL_M^{(2)})$ contains a chord at Hausdorff distance $<\eps$
of $[A_i', B'_i]$ with probability $>1-\delta$.
A small adaptation of the proof provides
 the same result for $[B'_i,C'_i]$ and $[C'_i,A'_i]$; 
 the lemma follows, letting $\eps$ go to $0$
 and then $\delta$ to $0$.

Fix $\delta>0$. 
We can choose $\eta>0$ such that, with probability $>1-\delta/2$, 
 the chords $[A_i,B_i], [B_i, C_i], [C_i, A_i]$ have length $>\eta$.
Also, for $\eps$ small enough,
we have $\Delta(E)>\eps$ with probability $1-\delta/2$,
where $\Delta(E)$ is the minimal distance between two points of the set
$E=\{A''_j, B''_j, C''_j, \ j\le g\}$.
 We assume that both conditions hold and that $\epsilon<\eta/4$
 (which is possible up to taking a smaller $\epsilon$).
 
We now use the fact that the process $(\bL_c^{(2)})_{c \geq 0}$
is coded by a Brownian excursion $\be$ 
and thus by a continuum random tree $\cT_\infty$.
The point $A''_i=(R^g)^{-1}(A'_i)$ corresponds a.s. to a leaf $y_i$ of $\cT_\infty$ by this coding, since the mass measure of $\cT_\infty$ is supported by its leaves.
For any $\epsilon>0$, denote by $A_\epsilon(y_i)$ the set of ancestors of $y_i$ whose subtree has $h$-mass less than $\epsilon$.
Almost surely $\ell(A_\epsilon(y_i))>0$ and thus, for $M$ large enough, 
there is a point of $\cP_M(\cT_\infty)$ in $A_\epsilon(y_i)$.
This point codes a chord $c$ in $\bL_M^{(2)}$ 
of length $<\epsilon$ around $A''_i$. 
 
Let us check where $R^g(c)$ lies.
The composition $R^{i-1}$ of the first $i-1$ rotations
acts like a piecewise rotation on an arc containing $A''_i$ and both extremities of $c$
(since $\Delta(E)>\eps$,
 the points $A''_j, B''_j, C''_j$ for $j<i$ are at distance more than $\eps$ of $A''_i$),
  thus $R^{i-1}(c)$
is a chord of length $<\eps$ around $R^{i-1}(A''_i)=A_i$.
Applying $R_{i}$, we get that $R^{i}(c)=R_i(R^{i-1}(c))$ 
is a chord at distance $<\eps$ from $[A_i,D_i]$, where $D_i=R_i(A_i)$
(see \cref{rem:continuity}).
Note that $[A_i,D_i]$ has the same length as $[B_i,C_i]$,
which is at least $\eta$. In particular $R^{i}(c)$ has length at least $\eta-2\eps>\eta/2$.

The condition $\Delta(E)>\eps$ ensures that all points $R^i(A''_j)$, $R^i(B''_j)$
and $R^i(C''_j)$, for $j>i$ are at distance at least $\eps$ of $A_i$ and $D_i$.

In particular, the chords $R^i(c)$ and $[A_i, D_i]$ satisfy the hypothesis of
Lemma~\ref{lem:chorddistances} (iii) with respect to $R_{i+1}^g$
  (the extremities are pairwise in the same connected components
of the circle without $\bigcup_{j>i} \{  R^i(A''_j),R^i(B''_j),R^i(C''_j)\}$).
Using that $A'_i=R^g_{i+1}(R_i(A_i))=R^g_{i+1}(D_i)$
and $B'_i=R^g_{i+1}(R_i(B_i))=R^g_{i+1}(A_i)$, we have 
\[ d_H\big(R^g(c),[A'_i,B'_i]\big)
= d_H\Big(R_{i+1}^g\big(R^i(c)\big),R_{i+1}^g([D_i, A_i])\Big) 
\leq s_{\eta/2} d_H\big(R^i(c),[D_i, A_i] \big) < s_{\eta/2} \eps.\]
Recall that this holds with probability $>1-\delta$.
Since $R^g(c)$ belongs to $R^g(\bL_M^{(2)})$,
this proves our statement (note that $\eta$ depends only on $\delta$
and not on $\eps$).
\end{proof}

We can now prove \cref{thm:cvginfty}.

\begin{proof}[Proof of \cref{thm:cvginfty}]
Let $(\{A_i, B_i, C_i\})_{1 \leq i \leq g}$ be  i.i.d. triples of uniform points on the circle 
and let us define $(\{A'_i,B'_i,C'_i\})_{1 \leq i \leq g}$ as above.
Clearly, $(\{A'_i,B'_i,C'_i\})_{1 \leq i \leq g}$ are also i.i.d. uniform. 
Moreover, from \cref{lem:Triangles_In_XgInfty},
almost surely, 
the $3g$ chords $[A'_i,B'_i]$, $[B'_i,C'_i]$ and $[C'_i,A'_i]$  (for $1\le i \le g$)
belong to $\bm \Sieve^g_\infty$.

Let us fix $s \geq 1$. We discretize the circle, 
denoting  $I_k$ the arc $(e^{-2\pi i k/s}, e^{-2\pi i (k+1)/s})$ (for $1 \leq k \leq s$),
and consider the following event $E_{s,g}$:
for any $k_1, k_2,k_3 \in [ 1,s ]$, there exists $1 \leq i \leq g$ such that $(A'_i, B'_i, C'_i)$
lies in  $I_{k_1} \times I_{k_2} \times I_{k_3}$, up to reordering the triple.
Since the $3g$ chords $[A'_i,B'_i]$, $[B'_i,C'_i]$ and $[C'_i,A'_i]$  (for $1\le i \le g$)
belong to $\bm \Sieve^g_\infty$, the event $E_{s,g}$ implies 
$$d_H(\bm \Sieve^g_\infty, \D ) \le 1/s.$$

It is clear that the complementary event $E_{s,g}^c$ has probability $p_{s,g} \leq s^3 \left( 1-\frac{1}{s^3} \right)^{g}$.
In particular, for fixed $s$, one has $\sum_{g=1}^\infty p_{s,g} < \infty$.
By the Borel-Cantelli lemma, almost surely only a finite number of $E_{s,g}$ do not occur;
thus, any subsequential limit $\bm \Sieve^\infty_\infty$ 
of $\bm \Sieve^g_\infty$ (as $g$ tends to infinity)
satisfies $d_H(\bm \Sieve^\infty_\infty, \D ) \le 1/s$.
Letting $s$ tend to infinity, the only possible limit of $\bm \Sieve^g_\infty$
along a subsequence is $\D$ and we conclude by compactness
of the space of compact subspaces of the disk with respect to Hausdorff distance.
\end{proof}

\subsection{Convergence after rotation of the Brownian lamination}

We prove here the convergence of the image by rotations of the laminations constructed from the Brownian excursion:

\begin{proposition}
\label{prop:rotationofcontinuouslaminations}
For almost every $g$-tuple of triples of points, almost surely,
\begin{align*}
R^g\left(\bL_M^{(2)}\right) \underset{M \rightarrow \infty}{\rightarrow} R^g\left(\bL_\infty^{(2)}\right)
\end{align*}
for the Hausdorff distance.
\end{proposition}

\begin{proof}
Let us first remark that, by \cite[Proposition $2.2$ (ii)]{thevenin2019geometric}, we have $\lim_{M \rightarrow \infty} \bL_M^{(2)} = \bL_\infty^{(2)}$, which proves it for $g=0$. Furthermore, the sequence $(\bL_M^{(2)})_{M \geq 0}$ is a.s. nondecreasing for the inclusion. 
Hence, $(R^g(\bL_M^{(2)}))_{M \geq 0}$ is also a.s. nondecreasing for the inclusion,
 and therefore converges a.s. to some limit $\bL$.
 Since, for any $M$, one has $\bL_M^{(2)} \subseteq \bL_\infty^{(2)}$, this implies
 $R^g(\bL_M^{(2)}) \subseteq R^g(\bL_\infty^{(2)})$ and taking the limit $M \to \infty$,
 we have $\bL \subseteq R^g(\bL_\infty^{(2)})$. We will prove the reverse implication.\medskip

Throughout the proof, for a point $P$ in the unit circle, we denote by $x_P$ the unique element of $[0,1)$ 
such that $P = e^{-2\pi i x_P}$.
Let us fix $\epsilon>0$, and recall that $E$ denotes the set $\bigcup_{i=1}^g \{ A''_i, B''_i, C''_i \}$.
Recall that $\bL_\infty^{(2)}$ is constructed from a Poissonian rain 
under a Brownian excursion $\be$. 
It is well-known that any given $x \in(0,1)$ is a.s. not a one-sided local minimum of $\be$.
Hence, using the notation of \cref{ssec:background_lamination},
there exists $r>0$ such that, for all $P \in E$, for all $t' \in (\be_{x_P} - r, \be_{x_P}), d(x_P,t')-g(x_P,t')<\epsilon$.  
Without loss of generality, taking $\epsilon$ small enough, one may also assume that all intervals $[g(x_P,\be_{x_P} - r), d(x_P,\be_{x_P} - r)]$ are disjoint, and that the Euclidean distance between two of these intervals is at least $2 \epsilon$. We also let $\eta := \min_{P \in E} \min \{dist(g(x_P,\be_{x_P} - r),x_P), dist(x_P,d(x_P,\be_{x_P} - r))\}$ where $dist$ denotes the Euclidean distance. 
Thus, it is clear that any chord in $\bL_\infty^{(2)}$ having an endpoint in one of the arcs $(e^{-2\pi i(x_P-\eta)}, e^{-2\pi i (x_P+\eta)})$ has both of its endpoints at distance less than $2\pi\epsilon$ from $P$.\medskip

We now construct a finite set $L_\eps$ of chords in $R^g(\bL_\infty^{(2)})$
 such that $\bS^1 \cup L_\eps$ is at Hausdorff distance $<2\pi \epsilon$ from $R^g(\bL_\infty^{(2)})$. 
 To this end, take $s \geq 1$ such that $2\pi/s < \eta/2$ 
 and denote by $I$ the set of all arcs of the form
  $(e^{-2\pi i a/s}, e^{-2\pi i (a+1)/s})_{0 \le a \le s-1}$ 
  not containing a point in $E$ and not neighboring an arc containing a point in $E$.
 Now, for any $x_1, x_2 \in I$, if there exists a chord of length $>1/s$ in $R^g(\bL_\infty^{(2)})$ connecting $x_1$ to $x_2$, 
we select arbitrarily one such chord and add it to $L_\eps$.
 Finally, we add to the set of remaining chords the $3g$ chords of the form $[A'_i, B'_i], [B'_i, C'_i], [C'_i, A'_i]$ (these chords are also included in $R^g(\bL_\infty^{(2)})$). 

The set of chords $L_\eps$ that we obtain is clearly finite, and 
we will check that $\bS^1 \cup L_\eps$ is at distance $<2\pi\epsilon$ from $R^g(\bL_\infty^{(2)})$. 
Indeed, taking a chord of $R^g(\bL_\infty^{(2)})$, three things may happen. 
\begin{itemize}
\item First, if $c$ is of the form $[A'_i, B'_i], [B'_i, C'_i]$ or $[C'_i, A'_i]$ for some $i \leq g$, then $c \subset L_\eps$. 
\item Secondly, we consider the case where none of the endpoints of $c$ is in an arc containing a point of $E$ or neighbor of such an arc.
Then there is a chord $c'$ in $L_\eps$ with both its extremities in the same arc as $c$.
 We have $d_H(c,L_\eps) \le d_H(c,c')<\epsilon$. 
\item The third case is when one of the endpoints of $c$, say $U$, is in an arc containing a point of $E$ or a neighbor of such an arc.
In this case, $x_U$ is at distance less than $\eta$ from a point of $E$ by definition of $s$. Furthermore, since we assumed the arcs $[g(x_P,\be_{x_P} - r), d(x_P,\be_{x_P} - r)]$
to be disjoint (when $P$ runs over $E$), this point of $E$ is unique. 
Say without loss of generality that it is $A''_1$. Then $x_{(R^g)^{-1}(U)} \in (x_{A''_1}-\eta, x_{A''_1}+\eta)$ and, by construction of $\eta$,
one has $x_{(R^g)^{-1}(V)} \in (x_{A''_1}-\eta, x_{A''_1}+\eta)$ where $V$ is the other endpoint of $c$. This implies that there exist two elements $X_U, X_V \in \{A'_1, B'_1, C'_1 \}$ such that $d(U,X_U)<2\pi\eta$ and $d(V,X_V)<2\pi\eta$. Hence, $c$ is at Hausdorff distance $<2\pi\epsilon$ from either one of the chords $[A'_1, B'_1], [B'_1, C'_1]$ or $[C'_1, A'_1]$, or the circle $\bS^1$.
\end{itemize}

Now, remark that by definition a chord of $L_\eps$ is either of the form $[A'_i, B'_i], [B'_i, C'_i], [C'_i, A'_i]$, or is obtained as a rotation of a chord of $\bL_\infty^{(2)}$, which is an accumulation point of actual chords $c(s,t)$ for some $(s,t) \in \cEG(\mathbbm{e})$. 
Hence, we can modify $L_\eps$ into a finite subset of chords $L'_\eps$ made only of the chords $[A'_i, B'_i], [B'_i, C'_i], [C'_i, A'_i]$ and of rotations of chords of the form $c(s,t)$, such that $\bS^1 \cup L'_\eps$ is at distance $<2\pi\epsilon$ from $R^g(\bL_\infty^{(2)})$. The only thing left to do is prove that any such chord is well approximated in $R^g(\bL_M^{(2)})$ with high probability as $M \rightarrow \infty$.

First we consider a chord $c_0=R^g(c)$ in $L'_\eps$ for some chord $c:=c(s,t)$, which is not of the form
 $[A'_i, B'_i], [B'_i, C'_i], [C'_i, A'_i]$.
 Recall that, by construction, its endpoints are not in the same arcs as the points $A'_i, B'_i, C'_i$ nor in neighbors of these arcs. 
 For any $\delta>0$, set $t':= \inf\{ u<t, d(s,u)-g(s,u) < d(s,t)-g(s,t)+\delta\}$. Almost surely $t'<t$ by continuity of $\mathbbm{e}$, and $V_{c,\delta} := [g(s,t), d(s,t)] \times [t',t]$ has positive $2$-dimensional Lebesgue measure. Thus, as $M$ tends to $\infty$, the probability that $\cN_M(\mathbbm{e}) \cap V_{c,\delta}$ is nonempty goes to $1$. But any point in $V_{c,\delta}$ codes a chord $c'$ at distance $<\delta$ from $c$. Finally, for $\delta$ small enough, this chord $c'$ and $c$ satisfy the assumption of Lemma \ref{lem:chorddistances} (iii), so that $d_H(R^g(c'),R^g(c))<\delta$.
 Since $R^g(c')$ belongs to $R^g(\bL_M^{(2)})$,
 we have found a chord in $R^g(\bL_M^{(2)})$ at distance $<\delta$ of $c_0=R^g(c)$.

We finally need the fact that the chords $[A_i', B'_i], [B'_i, C'_i], [C'_i, A'_i]$
are also well approximated as $M \rightarrow \infty$, which has been done earlier in the proof of \cref{lem:Triangles_In_XgInfty}.
\end{proof}

\section{Proofs of the scaling limit theorems for the final sieve and the sieve process}
\label{sec:proofs_Scaling}

This section is devoted to the proofs of the main convergence results of the paper, Theorems \ref{thm:main} and \ref{thm:mainprocess}, which respectively state the convergence of the random sieve $\Sieve(\bm F_n^g)$ and sieve-valued process $(\bm \Sieve_k(\bm F_n^g))_{0 \leq k \leq n}$, for the Hausdorff distance on the compact subsets of the unit disk. The main idea is to prove that the sieve-valued process $(\bm \Sieve_k(\bm F_n^g))_{0 \leq k \leq n}$ can be approximated by a sieve process obtained from a size-conditioned Galton-Watson tree whose vertices are labelled uniformly at random, and prove the convergence of this very process.

\subsection{Strategy of the proof}

Let us briefly explain how we prove Theorems \ref{thm:main} and \ref{thm:mainprocess}. 
The first step, which is the purpose of Section \ref{ssec:lamGWtree}, is to show a convergence result for rotations of lamination-valued processes constructed from a size-conditioned Galton-Watson tree, generalizing Theorem \ref{thm:cv_lamination_G0}. Recall that, for all $g\geq 0$, the operation $R^g$ performs $g$ rotations on a given sieve, around $g$ i.i.d. triples of uniform points. Recalling the construction of the process $(\bL_u(T_n))_{u \geq 0}$ from a uniform labelling of the edges of the tree $T_n$, we claim that the process $R^g((\bL_u(T_n)))_{u \geq 0}$ converges in distribution after renormalization to the image by $R^g$ of the limiting process $(\bL_c^{(2)})_{c \geq 0}$, where $T_n$ denotes a $Po(1)$-GW tree with $n-1$ edges labelled uniformly from $1$ to $n-1$. We also prove the convergence of the sieve $R^g(\bL(T_n))$, as $n \rightarrow \infty$. 

The connection between factorizations of the $n$-cycle and size-conditioned Galton-Watson trees is made clear by a bijection of Goulden and Yong \cite{GY02}, presented in Section \ref{ssec:gouldenyong}. This bijection associates to any element $F_n^0$ of $\setFZ$ a dual tree $T(F_n^0)$ with $n$ labelled vertices.

If $\rFZ$ is uniform on $\setFZ$, we show in \cref{ssec:chordtotree} that the lamination processes $R^g(\Sieve_k(\rFZ))_{0\le k \le n}$ and $R^g(\bL_k(T(\rFZ))_{0\le k \le n})$ are close to each other with high probability, and so are their images by $R^g$. It appears that $T(\rFZ)$ is distributed as a conditioned $Po(1)$-GW tree, but its labelling is not uniform. To overcome this, we use an operation on $T(\rFZ)$ introduced in \cite{thevenin2019geometric}, which shuffles labels on its vertices without changing much the associated sieve-valued process. 
In \cref{ssec:shuffling}, we prove that $R^g(\bL_k(T(\rFZ))_{0\le k \le n})$
is close to the analog process after shuffling, which is distributed as $(R^g(\bL_k(T_n)))_{0\le k \le n})$, where $T_n$ is a size-conditioned $Po(1)$-GW tree with uniform labelled edges.

We note that the proximity of the relevant processes before rotation had been established in \cite{thevenin2019geometric}; our contribution here is to show that they are still valid
after rotations. Due to the noncontinuity of $R_g$, this is delicate and based on precise estimates concerning the structure of large Galton-Watson trees.

The only thing left is to connect the sieve-valued process $(\Sieve_k(\rF))_{0\le k \le n}$ obtained from a uniform element $\rF$ of $\setF$ to the process $(R^g(\Sieve_k(\rFZ))_{0\le k \le n}$, obtained by rotating $g$ times the process associated to a uniform minimal factorization $\rFZ$ of $\setFZ$. This is where we make use of the algorithm presented in Section \ref{ssec:generation} and of several technical results proved in Section \ref{sec:asympbij}. We can assume that the element $\rF$ of $\setF$ that we consider was constructed by this algorithm, starting from a minimal factorization $\rFZ$ of the $n$-cycle. In Section \ref{ssec:genusgtogenus0}, we prove that $(\Sieve_k(\rF))_{0\le k \le n}$ and $(R^g(\Sieve_k(\rFZ))_{0\le k \le n}$ are close to each other. This is not trivial for several reasons. First, there are $2g$ more chords in the sieve associated to $\rF$ than in the lamination constructed from $\rFZ$. Second, the algorithm does not exactly correspond to the operation $R^g$, and the same chord in $\bL(\rFZ)$ may be sent to different places by $R^g$ and by the algorithm. However, it is possible to prove that such chords are necessarily either small or already well-approximated before they appear in the sieve-valued processes. 

Finally, all these results together prove the convergences of Theorems \ref{thm:main} and \ref{thm:mainprocess}.

\subsection{Notation}

We recall the main lines of the random generation algorithm from \cref{ssec:generation}.
We first take a uniform minimal factorization $\rFZ$ of the $n$-cycle. Then, for given $\aaa, \bbb, \ccc \leq n$, we add transpositions $(\aaa \, \bbb)$ and $(\aaa \, \ccc)$ at given positions $\vvv$ and $\www$ (or $\www$ and $\vvv$) so that the final product of transpositions is an $n$-cycle; we then conjugate all transpositions
so that the product is $(1\, \cdots\, n)$.
Iterating this $g$ times with numbers $\vvv,\www,\aaa,\bbb,\ccc$ taken uniformly at random
gives an asymptotically uniform factorization $\rF$ of genus $g$ of the cycle $(1\, \cdots\, n)$.

We will keep these notations throughout the section: namely,
$\vvv,\www$ denote positions of newly introduced transpositions,
$\aaa,\bbb,\ccc$ the moved elements.
Moreover, a capital $F$ denotes a factorization of genus $g$, 
while a small $f$ denotes the factorization of genus $0$ giving birth to $F$.
Further notation for this section is given in \cref{tab:notation}.
\begin{table}[!ht]
\renewcommand{\arraystretch}{1.6}
\[\begin{tabular}{|c|c|}
\hline
$T_n$ & $Po(1)$-GW tree with $n$ vertices labelled uniformly form $1$ to $n$.\\
\hline
$\bL(T)$ & lamination constructed from the contour function of $T$.\\
\hline
$\H(T)$ & height of $T$.\\
\hline
$\theta_x(T)$ & subtree of $T$ rooted in $x$\\
\hline
$\rF$ & uniform factorization of the $n$-cycle of genus $g$\\
\hline
$\rFZ$ & uniform {\em minimal factorization} of the $n$-cycle\\
\hline
$(\aaa \,\bbb),(\aaa \,\ccc)$ & transpositions that are added to $\rFZ$ at each step to get $\rF$ by the algorithm of \cref{ssec:generation}. \\
\hline
$\vvv,\www$ & locations of these transpositions. \\
\hline
$\Sieve_k(F)$ & union of $\bS^1$ and of the first $k$ chords in the sieve associated to $F$ \\
\hline
$T(f)$ & tree associated to $f$ by the Goulden-Yong bijection\\
\hline
$R_{A,B,C}(L)$ & sieve obtained by rotating a lamination $L$ according to the points $A,B,C$. \\
\hline
$R^g$ & $g$ iterations of operations of the form $R_{A,B,C}$. \\
\hline
\end{tabular}
\]
\caption{Summary of the main notation of \cref{sec:proofs_Scaling}}
\label{tab:notation}
\end{table}

\bigskip

For $T$ a tree with $n$ vertices and $A := e^{-2\pi i a} \in \bS^1$ (for $a \in [0,1)$), we denote by $V(A)$ the vertex of $T$ such that, at time $2na$, the contour exploration is in the edge between $V(A)$ and its parent in $T$ (or at the vertex $V(A)$). If $a \in [1-1/n,1)$, by convenience, we set $V(A)=\emptyset$, the root of $T$. This way, it is not hard to see that if $A$ is uniform on $\bS^1$ then $V(A)$ is uniform on the vertices of $T$.

\subsection{Rotation of lamination processes of Galton-Watson trees}
\label{ssec:lamGWtree}

In this first part, we study the convergence of the sieve process obtained by applying $g$ successive rotations to the lamination process constructed from the contour function of a (uniformly labelled) size-conditioned $Po(1)$-GW tree (this result holds in fact for any critical offspring distribution with finite variance). 
\begin{theorem}
\label{thm:convergence_rotated_GW}
Let $T_n$ be a $Po(1)$-GW tree conditioned on having $n$ vertices, labelled uniformly at random from $1$ to $n$. Then, for any fixed $g \geq 0$, for almost every $g$-tuple of i.i.d. uniform triples of points, on $\mathcal D(\R_+, \SetSieve(\D))$:
\begin{equation}
\label{eq:RotatedGW}
\left(R^g\left( \bL_{(c/\sqrt{2})\sqrt{n}}(T_n)\right) \right)_{c \geq 0} \underset{n \rightarrow \infty}{\overset{(d)}{\rightarrow}} \left( R^g(\bL_c^{(2)}) \right)_{c \geq 0}
\end{equation}
Moreover, jointly with the convergence above,
we have convergence of the endpoint of this process:
\begin{equation}
\label{eq:convergence_RgLTn}
R^g \big(\bL(T_n) \big)  \underset{n \rightarrow \infty}{\overset{(d)}{\rightarrow}} R^g(\bL_\infty^{(2)}).
\end{equation}
\end{theorem}
Remark that, in genus $0$, this is Theorem \ref{thm:cv_lamination_G0}.	
In order to prove Theorem \ref{thm:convergence_rotated_GW},
we start by a lemma giving precise estimates on the distribution of vertices 
with small labels in a uniformly labelled tree.

\begin{lemma}
\label{lem:uniformancestrallines}
Fix $\delta>0$ and $K>0$. Let $T_n$ be a $Po(1)$-GW tree conditioned on having $n$ vertices, labelled uniformly at random from $1$ to $n$, and let $V_1, \ldots, V_k$ be $k$ i.i.d. random vertices of $T_n$. Then, there exists $\eta>0$ such that, for $n$ large enough,
 with probability $>1-\delta$, no vertex in the ancestral line of any of the $V_i$'s 
has both a subtree of size $<\eta n$ and a label $< K \sqrt{n}$.
\end{lemma}

We will also need the following continuous counterpart of this lemma.

\begin{lemma}
\label{lem:uniformancestrallinescontinuous}
Fix $\delta>0$ and $K>0$. Let $\cN_K$ be a Poisson point process of intensity $K d\ell$ on the skeleton of Aldous' CRT $\cT_\infty$. Let $L_1, \ldots, L_k$ be $k$ i.i.d. random leaves of $\cT_\infty$. Then, there exists $\eta>0$ such that, with probability $>1-\delta$, there is no point of $\cN_K$ with a subtree of size $<\eta$ in the ancestral line of any of the $L_i$'s.
\end{lemma}

Assuming these results, whose proof is postponed, we can directly prove Theorem \ref{thm:convergence_rotated_GW}.

\begin{proof}[Proof of Theorem \ref{thm:convergence_rotated_GW}]
We fix $M>0$ and we first prove that $\lim_{n \to \infty} R^g(\bL_{(M/\sqrt{2}) \sqrt{n}}(T_n)) = R^g(\bL_M^{(2)})$.
Recall that we denote $E=\{A''_i,B''_i,C''_i ;\, i\le g\}$ and set $\eta_0:= \inf\{ d(U,V), U,V \in E \}$ where $d$ denotes the Euclidean distance in $\R^2$. Fix $\delta_0>0$. By \cref{lem:uniformancestrallines,lem:uniformancestrallinescontinuous}, 
there exists $\eta_1 \in (0,\eta_0)$ such that, with probability $>1-\delta_0$, no ancestral line of any of the elements of the set $\{V(P), P \in E \}$ contains a vertex of label $< (M/\sqrt{2}) \sqrt{n}$ whose subtree has size $< \eta_1 n$ and the analogue statement holds in the continuous. We assume these events hold and work conditionally on them.

We make use here of the chord-to-chord correspondence of Proposition \ref{prop:chordtochord}. 
Fix $\epsilon,\delta>0$. There exists $K_0>0$
such that \cref{eq:DefK0} holds, taking $\eta=\eta_1$ in the definition of the events $E^M_{Y,Z,R}$.
 We can take $s$ large enough so that, with probability $>1-\delta$, for any $K\le K_0$,
any given $Y,Z \in I_s^{(K)}$, no element of $E$ lies in any arcs of the lists $Y$ or $Z$ or in a neighbor arc.
Recall that we can couple $\bL_{(M/\sqrt{2}) \sqrt{n}}(T_n)$ and $\bL_M^{(2)}$ such that, for $n \ge n_0$
with probability at least $1 -\delta$,
the event $E^M_{Y,Z,R}(\kL_n)$  holds if and only if $E^M_{Y,Z,R}(\kL)$ holds
(for any $K\le K_0$, and any $Y,Z,R$) and one of this event holds. 

We condition on the 
events $E^M_{Y,Z,R}(\kL_n)$ and $E^M_{Y,Z,R}(\kL)$, for some $K\le K_0$, and some $Y,Z,R$.
We will show that, for $n$ large enough, we have 
\[d_H(R^g(\bL_{(M/\sqrt{2}) \sqrt{n}}(T_n)), R^g(\bL_M^{(2)})) \le \max(\eta_1,\eps).\]

We first show that, for $n$ large enough, the sieve $R^g(\bL_{(M/\sqrt{2}) \sqrt{n}}(T_n))$ 
is included in the $\max(\eta_1,\eps)$-neighborhood of $R^g(\bL_M^{(2)})$.
To this end we can forget chords of length $\le \eta_1$ in $R^g(\bL_{(M/\sqrt{2}) \sqrt{n}}(T_n))$.
So let us consider a chord of length $>\eta_1$ in $R^g(\bL_{M\sqrt{n}}(T_n))$
(since $E^M_{Y,Z,R}(\kL_n)$ holds, there are at most $K_0$ such chords). 
Either it arises as rotation of a chord that already had length $>\eta_1$ in $\bL_{(M/\sqrt{2})\sqrt{n}}(T_n)$, 
or from a chord of length $<\eta_1$ that has become larger after a rotation. We study both cases separately.

We first consider chords that were already large before rotations. Let us consider the chords of lengths $>\eta_1$ in $\bL_{(M/\sqrt{2}) \sqrt{n}}(T_n)$, for $n$ such tat $E^M_{Y,Z,R}(\kL_n)$ holds. Take $c$ such a chord. 
In particular, by definition of the event $E^M_{Y,Z,R}(\kL_n)$,
the chord $c$ connects some arc in $Y$ to an arc in $Z$.
We set $\eps'=\min(\eps/s_{\eta_1},1/s)$, where $s_{\eta_1}$ is the constant from \cref{lem:chorddistances}.
By the chord-to-chord convergence, for $n$ large enough there exists a chord $c' \in \bL_M^{(2)}$ such that $d_H(c,c')<\epsilon'$. 
Since no points of $E$ lies in some arcs of $Y$ or $Z$ or in some neighboring arc, 
 $c$ and $c'$ satisfy the conditions of Lemma \ref{lem:chorddistances} (iii). 
Thus, $d_H(R^g(c), R^g(c'))<s_{\eta_1} \epsilon' \le\eps$
and the chord $R^g(c)$ of $R^g(\bL_{(M/\sqrt{2})\sqrt{n}}(T_n))$
is included in the $\eps$-neighborhood of $R^g(\bL_M^{(2)})$, as wanted.
\medskip

Now, assume that a small chord $c$ of length $<\eta_1$ in $\bL_{(M/\sqrt{2}) \sqrt{n}}(T_n)$ is such that $R^g(c)$ has length $> \eta_1$. Then necessarily, since $\eta_1<\eta$, the chords $c$ isolates one of the elements of $E$, say $P$, from all others. Thus $c$ corresponds to an ancestor of $V(P)$, whose subtree has size $<\eta_1 n$ and whose label is $< M \sqrt{n}$. This contradicts our assumption. 
Therefore no chord of length $>\eta_1$ of $R^g(\bL_{(M/\sqrt{2})\sqrt{n}}(T_n))$
is obtained as rotation of a chord of length $<\eta_1$.
\medskip

We conclude that $R^g(\bL_{(M/\sqrt{2}) \sqrt{n}}(T_n))$ is included in a 
$\max(\eta_1,\eps)$-neighborhood of $R^g(\bL_M^{(2)})$ for $n$ large enough.
A reverse argument shows that $R^g(\bL_M^{(2)})$
is included in a 
$\max(\eta_1,\eps)$-neighborhood of $R^g(\bL_{(M/\sqrt{2}) \sqrt{n}}(T_n))$
for $n$ large enough, which implies
\[d_H(R^g(\bL_{(M/\sqrt{2}) \sqrt{n}}(T_n)), R^g(\bL_M^{(2)})) \le \max(\eta_1,\eps).\]
This happens conditionally on both $E^M_{Y,Z,R}(\kL_n)$ and $E^M_{Y,Z,R}(\kL)$ occurring
for some $K\le K_0$, and some $Y,Z,R$, so with probability at least $1-\delta$,
for $n$ large enough.

Hence, in the Hausdorff topology, we have the convergence in distribution
\begin{equation}
\label{eq:lim_RgLM}
 \lim_{n \to +\infty} R^g(\bL_{(M/\sqrt{2}) \sqrt{n}}(T_n)) =R^g(\bL_M^{(2)}).
 \end{equation}
Moreover, in the above argument one can see that the chord $R^g(c')$
approximating $R^g(c)$ appears in the same interval of time of the form $(jM/D, (j+1)M/D)$
(see the definition of the event $E^M_{Y,Z,R}(\kL_n)$). Letting $D$ go to $+\infty$,
this proves the convergence of the processes 
\[ \lim_{n \to +\infty} \big(R^g(\bL_{(c/\sqrt{2}) \sqrt{n}}(T_n)) \big)_{c \le M} = \big(R^g(\bL_c^{(2)})\big)_{c \le M},\]
in the Skorokhod topology.
Since this holds for any $M>0$ we have convergence of the processes on $[0,+\infty)$.
\medskip

It remains to prove \cref{eq:convergence_RgLTn},
 which corresponds informally to the case $c=+\infty$.
Using the Skorokhod representation theorem, we assume
that the random trees $T_n$ and the limit $(\bL_c^{(2)})_{c \ge 0}$ are defined
on the same probability space such that the convergence
\[ \lim_{n \to +\infty} \big(R^g(\bL_{(c/\sqrt{2}) \sqrt{n}}(T_n)) \big)_{c \ge 0} = \big(R^g(\bL_c^{(2)})\big)_{c \ge 0}\]
holds almost surely in the Skorokhod space $\mathcal D(\R_+,\SetSieve(\D))$.

Let $X$ be a subsequential limit of $R^g(\bL(T_n))$ 
(which exists by compactness of the set of laminations of the disk).
For any $M >0$, we have $R^g_{(M/\sqrt{2}) \sqrt{n}}(\bL(T_n)) \subseteq R^g(\bL(T_n))$.
Going to the limit, and using \eqref{eq:lim_RgLM}, we get that, a.s., $R^g(\bL_{M}^{(2)}) \subseteq X$.
Since this holds for all $M$, we have $R^g(\bL_\infty^{(2)}) \subseteq X$.

Conversely, since $X$ is a limit of $R^g$ applied to a sequence of laminations (i.e. sieves without crossings),
we get that $(R^g)^{-1}(X)$ is a.s. a lamination.
 This lamination contains $\bL_\infty^{(2)}$. 
 Since the Brownian triangulation is a.s. maximal for the inclusion \cite[Proposition 2.1]{LGP08},
 therefore we have $(R^g)^{-1}(X) = \bL_\infty^{(2)}$, i.e. $X = R^g(\bL^{(2)}_\infty)$. 
 The same holds for any subsequential limit of $R^g(\bL(T_n))$,
 showing that $R^g(\bL(T_n))$ converges to $R^g(\bL^{(2)}_\infty)$, as wanted.
\end{proof}

Let us now prove the technical Lemmas \ref{lem:uniformancestrallines} and \ref{lem:uniformancestrallinescontinuous}.

\begin{proof}[Proof of Lemma \ref{lem:uniformancestrallines}]
First note that, by a union bound, it is enough to prove the lemma for a single vertex $V_1$.
For a (random) tree $T$ with $n$ vertices and numbers $H_0,\eta>0$, 
we consider the  event $E(T,\eta, H_0)$: 
there exists a vertex $b$ in $T$ such that $|\theta_b(T)| < \eta n$ and $H(\theta_b(T))>H_0\sqrt{n}$, 
where we recall that $\theta_b(T)$ denotes the fringe subtree of $T$ rooted in $b$ and $H((\theta_b(T)))$ denotes its height.

Now, for $n \geq 1$, let $T_n$ be a $Po(1)$-GW tree with $n$ vertices. We claim that, for fixed $H_0$,
\begin{equation}
\label{eq:Tech}
\lim_{\eta \to 0} \Big(\limsup_{n \rightarrow \infty} \P(E(T_n,\eta,H_0)) \Big) = 0.
\end{equation}
To this end, observe that the event $E(T_n,\eta,H_0)$ implies that $\omega_{\tilde C_{T_n}}(\eta) \ge H_0$,
where $\tilde C_{T_n}$ is the normalized contour function of $T_n$ considered in Theorem~\ref{thm:aldleg},
and $\omega_f$ the modulus of continuity of $f$.
Then \eqref{eq:Tech} follows from the tightness of the sequence $\tilde C_{T_n}$,
see \cite[eq. (7.8)]{Bil_Cv}.

We come back to the proof of the lemma.
 Fix $H_0=\delta/2K >0$, take a tree $T$ with $n$ vertices labelled uniformly at random
 and assume that there is no vertex $b$ of $T$ such that $|\theta_b(T)| < \eta n$ and $H(\theta_b(T))>H_0\sqrt{n}$.
 Then, for a vertex $V_1$ of the tree, the number of ancestors of $V_1$ 
 whose subtree has size $<\eta n$ is less than $H_0 \sqrt{n}$. 
 In particular, taking $V_1$ uniformly at random among the vertices of $T$, 
 the probability that an ancestor of $V_1$ has both a subtree of size $<\eta n$ and a label $< K \sqrt{n}$
  is bounded by $K H_0=\delta/2$, independently of $T$.  
  For any $\delta>0$, by the above claim, 
  our assumption is satisfied by $T_n$ with probability $>1-\delta/2$,
   for some $\eta>0$ and $n$ large enough. This concludes the proof.
\end{proof}

\begin{proof}[Proof of Lemma \ref{lem:uniformancestrallinescontinuous}]
One can prove Lemma \ref{lem:uniformancestrallinescontinuous} in an analogous way: for any $\eta>0$, define the random variable $Y_\eta$ as
\begin{align*}
Y_\eta := \underset{\substack{x \in \cT_\infty \\ |\theta_x(\cT_\infty)|<\eta}}{\sup} H(\theta_x(\cT_\infty)).
\end{align*}
   Then $Y_\eta \le \omega_{ \mathbbm{e}}(\eta) \rightarrow 0$ almost surely as $\eta \rightarrow 0$.
In particular, for any $H_0>0$, the probability $\P(Y_\eta>H_0)$ tends to $0$,
giving a continuous analogue of \eqref{eq:Tech}. The rest of the proof is easily adapted,
replacing sizes of the fringe subtrees divided by $n$ by the mass measure $h$ of fringe subtrees in
$\mathcal T_\infty$,
the number of ancestors divided by $\sqrt{n}$ by the length measure $\ell$,
and vertices of labels $<K\sqrt{n}$ by points of the Poisson point process $\cN_K$.
\end{proof}

\subsection{The Goulden-Yong bijection}
\label{ssec:gouldenyong}

Before proving Theorem \ref{thm:mainprocess}, we expose the connection between minimal factorizations and the $Po(1)$-GW tree studied in Section \ref{ssec:lamGWtree}. This connection goes back to Goulden and Yong \cite{GY02}, who exhibit a bijection between the set $\setFZ$ of minimal factorizations of the $n$-cycle and the set of Cayley trees of size $n$ - that is, nonplane trees with $n$ vertices labelled from $1$ to $n$. It is known (see e.g. \cite[Example $10.2$]{Jan12}) that a $Po(1)$-GW tree conditioned to have $n$ vertices is in some sense distributed as a uniform Cayley tree of size $n$, which establishes our connection. In this whole subsection, all operations that we present are deterministic.

\bigskip

Let us present this bijection. We take $F_n^0:=(\transpo_1, \ldots, \transpo_{n-1}) \in \setFZ$ and, as in the construction of the sieve $\Sieve(F_n^0)$, for any $i\leq n-1$, we associate to the transposition $\transpo_i := (x_i \, y_i)$ the chord 
\begin{align*}
c_i := \left[ e^{-2\pi i  x_i/n}, e^{-2\pi i  y_i/n} \right].
\end{align*}

Furthermore, we give to $c_i$ the label $i+1$. This allows us to associate to any $f \in \setFZ$ the lamination $\Sieve(f)$ whose $n-1$ chords are labelled from $2$ to $n$ (see Fig. \ref{fig:gy}, up-left). Goulden and Yong observe that, in this lamination, around each vertex of the form $e^{-2\pi i  k/n}$, the labels of the chords admitting this vertex as an endpoint are sorted in increasing clockwise order. 

An interesting remark is that the chords of $\Sieve(F_n^0)$ form a tree with $n$ vertices and $n-1$ labelled chords. This tree is an embedding in the disk of the Hurwitz map $Map(f)$ defined in Section \ref{sec:facto_maps}, which appears to be a tree since $F_n^0 \in \setFZ$.

Now we define the dual tree of this labelled lamination, $T(F_n^0)$. This tree is obtained by putting a vertex in each face of $\Sieve(F_n^0)$, and connecting two of these vertices if the corresponding faces share a chord as a boundary. We now label this tree the following way: give the label $1$ to the vertex $V_1$ corresponding to the face containing the arc between $e^{-2\pi i /n}$ and $1$. Now, remark that each edge corresponds to a chord of $\Sieve(F_n^0)$, and give each vertex $V$ (except $V_1$) the label of the chord corresponding to the first edge in the unique path in the tree from $V$ to $V_1$. One obtains this way a tree $T(F_n^0)$ with $n$ vertices labelled from $1$ to $n$ (see Fig. \ref{fig:gy}, bottom-right).

\begin{figure}[!ht]
\center
\begin{tabular}{c c c}
\begin{tikzpicture}[scale=0.7, every node/.style={scale=0.7}, rotate=-40]
\draw (0,0) circle (3);
\foreach \i in {1,...,9}
{
\draw[auto=right] ({3.3*cos(-(\i-1)*360/9)},{3.3*sin(-(\i-1)*360/9)}) node{\i};
}
\draw ({3*cos(-2*360/9)},{3*sin(-2*360/9)}) -- ({3*cos(-3*360/9)},{3*sin(-3*360/9)}) node[circle,midway,fill=blue!20, inner sep=2pt]{4};
\draw ({3*cos(2*360/9)},{3*sin(2*360/9)}) -- ({3*cos(1*360/9)},{3*sin(1*360/9)}) node[circle,midway,fill=blue!20, inner sep=2pt]{3};
\draw ({3*cos(-3*360/9)},{3*sin(-3*360/9)}) -- ({3*cos(-4*360/9)},{3*sin(-4*360/9)}) node[circle,midway,fill=blue!20, inner sep=2pt]{2};
\draw ({3*cos(-2*360/9)},{3*sin(-2*360/9)}) -- ({3*cos(0*360/9)},{3*sin(0*360/9)}) node[circle,midway,fill=blue!20, inner sep=2pt]{5};
\draw ({3*cos(-5*360/9)},{3*sin(-5*360/9)}) -- ({3*cos(0*360/9)},{3*sin(0*360/9)}) node[circle,midway,fill=blue!20, inner sep=2pt]{6};
\draw ({3*cos(-7*360/9)},{3*sin(-7*360/9)}) -- ({3*cos(0*360/9)},{3*sin(0*360/9)}) node[circle,midway,fill=blue!20, inner sep=2pt]{7};
\draw ({3*cos(-2*360/9)},{3*sin(-2*360/9)}) -- ({3*cos(-1*360/9)},{3*sin(-1*360/9)}) node[circle,midway,fill=blue!20, inner sep=2pt]{8};
\draw ({3*cos(-6*360/9)},{3*sin(-6*360/9)}) -- ({3*cos(-7*360/9)},{3*sin(-7*360/9)}) node[circle,midway,fill=blue!20, inner sep=2pt]{9};
\end{tikzpicture}
&
\begin{tikzpicture}[scale=0.7, every node/.style={scale=0.7}, rotate=-40]
\draw (0,0) circle (3);
\foreach \i in {1,...,9}
{
\draw[auto=right] ({3.3*cos(-(\i-1)*360/9)},{3.3*sin(-(\i-1)*360/9)}) node{\i};
}
\draw[red,fill=red] (-.5,-.5) circle (.05);
\draw[red,fill=red] (0,1.5) circle (.05);
\draw[red,fill=red] (-.5,2.85) circle (.05);
\draw[red,fill=red] (1.6,-2.5) circle (.05);
\draw[red,fill=red] (1.6,2.5) circle (.05);
\draw[red,fill=red] (2.6,-.8) circle (.05);
\draw[red,fill=red] (-2.2,-2) circle (.05);
\draw[red,fill=red] (-.3,-2.9) circle (.05);
\draw[dashed,red] (-.5,-.5) -- (0,1.5) -- (-.5,2.85);
\draw[dashed,red] (1.6,-2.5) -- (2.6,-.8) --(-.5,-.5) -- (-2.2,-2) (-.5,-.5) -- (-.3,-2.9);
\draw[dashed,red] (0,1.5) -- (2.2,1.5) -- (1.6,2.5);

\draw[fill=green] (2.2,1.5) circle (.15);
\draw ({3*cos(-2*360/9)},{3*sin(-2*360/9)}) -- ({3*cos(-3*360/9)},{3*sin(-3*360/9)}) node[circle,midway,fill=blue!20, inner sep=2pt]{4};
\draw ({3*cos(2*360/9)},{3*sin(2*360/9)}) -- ({3*cos(1*360/9)},{3*sin(1*360/9)}) node[circle,midway,fill=blue!20, inner sep=2pt]{3};
\draw ({3*cos(-3*360/9)},{3*sin(-3*360/9)}) -- ({3*cos(-4*360/9)},{3*sin(-4*360/9)}) node[circle,midway,fill=blue!20, inner sep=2pt]{2};
\draw ({3*cos(-2*360/9)},{3*sin(-2*360/9)}) -- ({3*cos(0*360/9)},{3*sin(0*360/9)}) node[circle,midway,fill=blue!20, inner sep=2pt]{5};
\draw ({3*cos(-5*360/9)},{3*sin(-5*360/9)}) -- ({3*cos(0*360/9)},{3*sin(0*360/9)}) node[circle,midway,fill=blue!20, inner sep=2pt]{6};
\draw ({3*cos(-7*360/9)},{3*sin(-7*360/9)}) -- ({3*cos(0*360/9)},{3*sin(0*360/9)}) node[circle,midway,fill=blue!20, inner sep=2pt]{7};
\draw ({3*cos(-2*360/9)},{3*sin(-2*360/9)}) -- ({3*cos(-1*360/9)},{3*sin(-1*360/9)}) node[circle,midway,fill=blue!20, inner sep=2pt]{8};
\draw ({3*cos(-6*360/9)},{3*sin(-6*360/9)}) -- ({3*cos(-7*360/9)},{3*sin(-7*360/9)}) node[circle,midway,fill=blue!20, inner sep=2pt]{9};
\end{tikzpicture}
&
\begin{tikzpicture}[scale=0.7, every node/.style={scale=0.7}, rotate=-40]
\draw (0,0) circle (3);
\foreach \i in {1,...,9}
{
\draw[auto=right] ({3.3*cos(-(\i-1)*360/9)},{3.3*sin(-(\i-1)*360/9)}) node{\i};
}
\draw[red,fill=red] (-.5,-.5) circle (.05);
\draw[red,fill=red] (0,1.5) circle (.05);
\draw[red,fill=red] (-.5,2.85) circle (.05);
\draw[red,fill=red] (1.6,-2.5) circle (.05);
\draw[red,fill=red] (1.6,2.5) circle (.05);
\draw[red,fill=red] (2.6,-.8) circle (.05);
\draw[red,fill=red] (-2.2,-2) circle (.05);
\draw[red,fill=red] (-.3,-2.9) circle (.05);
\draw[dashed,red] (-.5,-.5) -- (0,1.5) node[circle, black,midway,fill=blue!20, inner sep=2pt]{6} -- (-.5,2.85) node[circle,black,midway,fill=blue!20, inner sep=2pt]{9};
\draw[dashed,red] (1.6,-2.5) -- (2.6,-.8) node[circle, black,midway,fill=blue!20, inner sep=2pt]{8} -- (-.5,-.5) node[circle, black,midway,fill=blue!20, inner sep=2pt]{5} -- (-2.2,-2) node[circle, black,midway,fill=blue!20, inner sep=2pt]{2} (-.5,-.5) -- (-.3,-2.9) node[circle, black,midway,fill=blue!20, inner sep=2pt]{4};
\draw[dashed,red] (0,1.5) -- (2.2,1.5) node[circle, black,midway,fill=blue!20, inner sep=2pt]{7} -- (1.6,2.5) node[circle, black,midway,fill=blue!20, inner sep=2pt]{3};
\draw[fill=green] (2.2,1.5) circle (.15);
\end{tikzpicture}
\\
\begin{tikzpicture}[scale=0.7, every node/.style={scale=0.7}]
\draw[red] (-1,-1) -- (0,0) node[black,circle,midway,fill=blue!20, inner sep=2pt]{2} -- (1,-1) node[black,circle,midway,fill=blue!20, inner sep=2pt]{5} -- (1,-2) node[black,circle,midway,fill=blue!20, inner sep=2pt]{8} (0,0) -- (0,-1) node[black,circle,midway,fill=blue!20, inner sep=2pt]{4};
\draw[red] (0,0) -- (0,1) node[black,circle,midway,fill=blue!20, inner sep=2pt]{6} -- (-1,2) node[black,circle,midway,fill=blue!20, inner sep=2pt]{9};
\draw[red] (0,1) -- (1,2) node[black,circle,midway,fill=blue!20, inner sep=2pt]{7} -- (1,3) node[black,circle,midway,fill=blue!20, inner sep=2pt]{3};
\draw[fill=green] (1,2) circle (.15);
\draw[red,fill=red] (-1,-1) circle (.05);
\draw[red,fill=red] (0,-1) circle (.05);
\draw[red,fill=red] (0,0) circle (.05);
\draw[red,fill=red] (1,-1) circle (.05);
\draw[red,fill=red] (1,-2) circle (.05);
\draw[red,fill=red] (0,1) circle (.05);
\draw[red,fill=red] (-1,2) circle (.05);
\draw[red,fill=red] (1,3) circle (.05);
\end{tikzpicture}
&
\begin{tikzpicture}[scale=0.7, every node/.style={scale=0.7}]
\draw[red] (0,0) -- (0,-1) (-1,-1) -- (0,0) -- (0,1) -- (-1,2);
\draw[red] (1,-2) -- (1,-1) -- (0,0) -- (0,1) -- (1,2) -- (1,3);
\draw (0,0) node[circle,fill=white, draw=black, inner sep=2pt]{6};
\draw (0,1) node[circle,fill=white, draw=black, inner sep=2pt]{7};
\draw (1,2) node[circle,fill=green, draw=black, inner sep=2pt]{1};
\draw (1,3) node[circle,fill=white, draw=black, inner sep=2pt]{3};
\draw (-1,2) node[circle,fill=white, draw=black, inner sep=2pt]{9};
\draw (-1,-1) node[circle,fill=white, draw=black, inner sep=2pt]{2};
\draw (0,-1) node[circle,fill=white, draw=black, inner sep=2pt]{4};
\draw (1,-1) node[circle,fill=white, draw=black, inner sep=2pt]{5};
\draw (1,-2) node[circle,fill=white, draw=black, inner sep=2pt]{8};
\end{tikzpicture}
&
\begin{tikzpicture}[scale=0.7, every node/.style={scale=0.7}]
\draw[red] (1,1) -- (0,0) -- (-1,1) -- (0,2);
\draw[red] (-1,1) -- (-1,2) -- (0,3) (-1,2) -- (-1,3);
\draw[red] (-1,2) -- (-2,3) -- (-2,4);
\draw (-1,2) node[circle,fill=white, draw=black, inner sep=2pt]{6};
\draw (-1,1) node[circle,fill=white, draw=black, inner sep=2pt]{7};
\draw (0,0) node[circle,fill=green, draw=black, inner sep=2pt]{1};
\draw (1,1) node[circle,fill=white, draw=black, inner sep=2pt]{3};
\draw (0,2) node[circle,fill=white, draw=black, inner sep=2pt]{9};
\draw (-1,3) node[circle,fill=white, draw=black, inner sep=2pt]{4};
\draw (0,3) node[circle,fill=white, draw=black, inner sep=2pt]{2};
\draw (-2,3) node[circle,fill=white, draw=black, inner sep=2pt]{5};
\draw (-2,4) node[circle,fill=white, draw=black, inner sep=2pt]{8};
\end{tikzpicture}
\end{tabular}
\caption{The Goulden-Yong mapping, applied to $F_n^0 := (45)(89)(34)(13)(16)(18)(23)(78) \in \mathcal{F}^0_9$. The lamination $\Sieve(F_n^0)$ with labelled chords is represented up-left. Condition $C_\Delta$ is satisfied by the tree $T(F_n^0)$ (bottom-right).}
\label{fig:gy}
\end{figure}
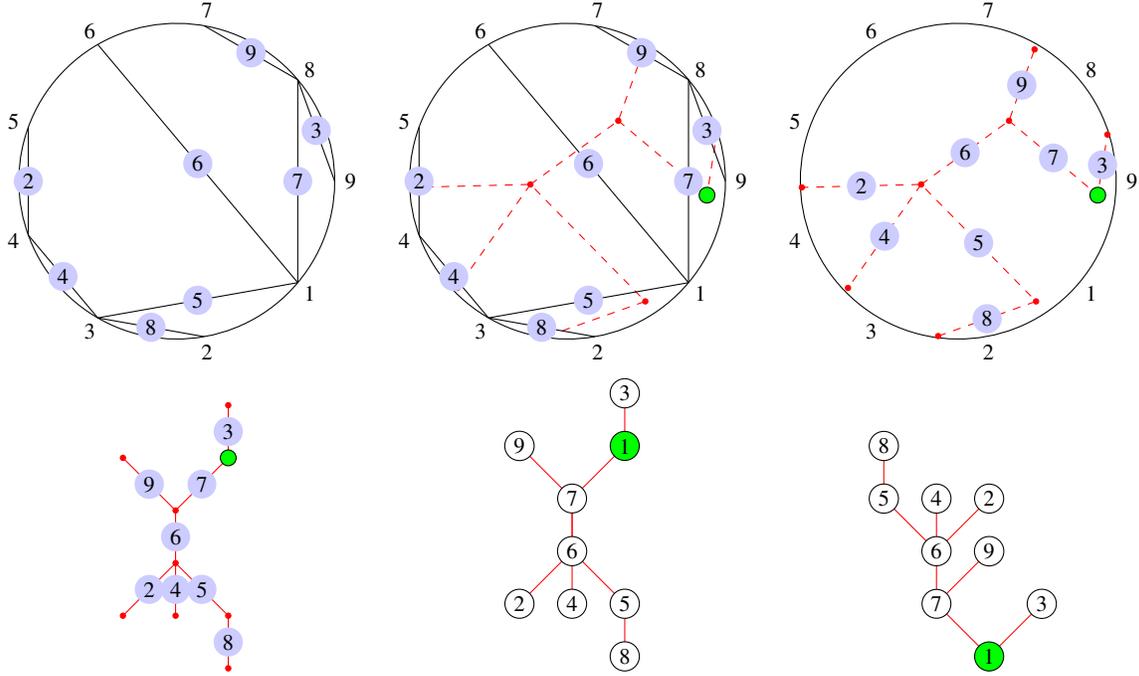

\begin{theorem}[Goulden \& Yong, \cite{GY02}]
For any $n \geq 1$, $F_n^0 \rightarrow T(F_n^0)$ is a bijection from $\setFZ$ to the set of Cayley trees of size $n$.
\end{theorem}

By a small abuse of notation, we will also denote by $T(F_n^0)$ the plane tree obtained by rooting it in $V_1$ and keeping the planar order induced by $\Sieve(F_n^0)$. It appears that, then, the labels satisfy the following condition:

\begin{definition}
We say that a plane tree $T$ with $n$ vertices labelled from $1$ to $n$ satisfies the condition $C_\Delta$ if the following hold:
\begin{itemize}
\item its root is labelled $1$;
\item for any vertex $V \in T$, the labels of $V$ and its potential children are sorted in decreasing clockwise order, starting from the largest of them.
\end{itemize}
\end{definition}

\begin{theorem}[\cite{thevenin2019geometric}]
\label{thm:GouldenYong}
For any $n \geq 1$, the map $F_n^0 \rightarrow T(F_n^0)$ is a bijection between $\setFZ$ and the set of trees with $n$ labelled vertices satisfying condition $C_\Delta$.
\end{theorem}

This implies the following result:

\begin{theorem}[\cite{Jan12,thevenin2019geometric}]
\label{thm:poissontree}
Fix $n \in \Z_+$, and let $\rFZ$ be a uniform element of $\setFZ$. Then $T(\rFZ)$, forgetting about its labels, is distributed as a $Po(1)$-GW tree with $n$ vertices. 
\end{theorem}

We finish this section by a short combinatorial observation that will be useful later.
We note that each vertex $V$ in $T(F_n^0)$ corresponds to a face $F(V)$ of the complement of the lamination $\Sieve(F_n^0)$,
and that the border of each face includes a single arc $(e^{- 2\pi i j/ n}, e^{- 2\pi i (j+1)/n})$.
Call $\Arc(V)$ the arc corresponding to the face $F(V)$. Besides, recall that chords of $\Sieve(F_n^0)$ correspond to edges of $T(f)$.
The following lemma is immediate.
\begin{lemma}
\label{lem:SeparatingChords_AncestralLine} 
Let $V$ and $V'$ be vertices of $T(f)$.
A chord $c$ of $\Sieve(F_n^0)$ separates the arcs $\Arc(V)$ and $\Arc(V')$ if and only if the corresponding edge separates $V$ and $V'$ in $T(f)$.
In particular, $c$ separates $\Arc(V)$ and the {\em root} arc $(e^{-2\pi i/n},1)$
if and only if the corresponding edge is on the ancestral line of $V$.
\end{lemma}

\subsection{From a minimal factorization to a Galton-Watson tree: a shuffling operation}
\label{ssec:shuffling}
We can now use the Goulden-Yong bijection to prove the convergence of the sieve-valued process associated to a uniform element of $\setFZ$. For any $n \geq 1$, we let $\rFZ$ be a uniform element of $\setFZ$. Because of the condition $C_\Delta$ satisfied by its dual tree $T(\rFZ)$, the labelling of this tree is not uniform. We present here a shuffling operation which still allows us to use the result of Theorem \ref{thm:convergence_rotated_GW} and apply it to this tree. This operation was defined in \cite{thevenin2019geometric} and used in the same context of minimal factorizations.

\begin{definition}
\label{def:shuffling}
Let $T$ be a plane tree with $n$ vertices labelled from $1$ to $n$, rooted at the vertex of label $1$, and let $K \leq n$. We define the shuffled tree $T^{(K)}$ as follows: starting from the root of $T$, we perform one of the following two operations on the vertices of $T$. For consistency, we impose that the operation shall be performed on a vertex before being performed on its children.
\begin{itemize}
\item Operation $1$: for a vertex such that the labels of its children are all $> K$, we uniformly shuffle these labels (without shuffling the corresponding subtrees).
\item Operation $2$: for a vertex such that at least one of its children has a label $\leq K$, we uniformly shuffle these labelled vertices and keep the subtrees on top of each of these children.
\end{itemize}
\end{definition}
See Figure~\ref{fig:shuffling} for an example. Note that this operation induces a transformation of the lamination $\bL(T)$ associated to $T$.

\begin{figure}[!h]
\center
\begin{tabular}{c c}
\begin{tikzpicture}[scale=.7, every node/.style={scale=.7}]
\draw (-2,3) -- (0,1.5) -- (-1,3);
\draw (0,3) -- (0,1.5) -- (1,3);
\draw (0,1.5) -- (2,3);
\draw (-2,3.3) circle (.3) node{7};
\draw (-1,3.3) circle (.3) node{4};
\draw (0,3.3) circle (.3) node{28};
\draw (2,3.3) circle (.3) node{12};
\draw (1,3.3) circle (.3) node{16};
\draw (0,1.2) circle (.3) node{9};
\draw (-2,3.6) -- (-2.3,5) -- (-1.7,5) -- cycle;
\draw (-1,3.6) -- (-1.3,4.5) -- (-.7,4.5) -- cycle;
\draw (0,3.6) -- (-.3,5.5) -- (.3,5.5) -- cycle;
\draw (1,3.6) -- (.7,4) -- (1.3,4) -- cycle;
\draw (2,3.6) -- (1.7,6) -- (2.3,6) -- cycle;
\draw (-2,5.5) node{$T_1$};
\draw (-1,5) node{$T_2$};
\draw (0,6) node{$T_3$};
\draw (1,4.5) node{$T_4$};
\draw (2,6.5) node{$T_5$};
\end{tikzpicture}
&
\begin{tikzpicture}[scale=.7, every node/.style={scale=.7}]
\draw (-2,3) -- (0,1.5) -- (-1,3);
\draw (0,3) -- (0,1.5) -- (1,3);
\draw (0,1.5) -- (2,3);
\draw (-2,3.3) circle (.3) node{28};
\draw (-1,3.3) circle (.3) node{4};
\draw (0,3.3) circle (.3) node{16};
\draw (2,3.3) circle (.3) node{7};
\draw (1,3.3) circle (.3) node{12};
\draw (0,1.2) circle (.3) node{9};
\draw (-2,3.6) -- (-2.3,5) -- (-1.7,5) -- cycle;
\draw (-1,3.6) -- (-1.3,4.5) -- (-.7,4.5) -- cycle;
\draw (0,3.6) -- (-.3,5.5) -- (.3,5.5) -- cycle;
\draw (1,3.6) -- (.7,4) -- (1.3,4) -- cycle;
\draw (2,3.6) -- (1.7,6) -- (2.3,6) -- cycle;
\draw (-2,5.5) node{$T_1$};
\draw (-1,5) node{$T_2$};
\draw (0,6) node{$T_3$};
\draw (1,4.5) node{$T_4$};
\draw (2,6.5) node{$T_5$};
\end{tikzpicture}
\\
\multicolumn{2}{c}{(a) Shuffling of a labelled plane tree when $K=3$: Operation $1$ is performed}\\

\begin{tikzpicture}[scale=.7, every node/.style={scale=.7}]
\draw (-2,3) -- (0,1.5) -- (-1,3);
\draw (0,3) -- (0,1.5) -- (1,3);
\draw (0,1.5) -- (2,3);
\draw (-2,3.3) circle (.3) node{7};
\draw (-1,3.3) circle (.3) node{4};
\draw (0,3.3) circle (.3) node{28};
\draw (2,3.3) circle (.3) node{12};
\draw (1,3.3) circle (.3) node{16};
\draw (0,1.2) circle (.3) node{9};
\draw (-2,3.6) -- (-2.3,5) -- (-1.7,5) -- cycle;
\draw (-1,3.6) -- (-1.3,4.5) -- (-.7,4.5) -- cycle;
\draw (0,3.6) -- (-.3,5.5) -- (.3,5.5) -- cycle;
\draw (1,3.6) -- (.7,4) -- (1.3,4) -- cycle;
\draw (2,3.6) -- (1.7,6) -- (2.3,6) -- cycle;
\draw (-2,5.5) node{$T_1$};
\draw (-1,5) node{$T_2$};
\draw (0,6) node{$T_3$};
\draw (1,4.5) node{$T_4$};
\draw (2,6.5) node{$T_5$};
\end{tikzpicture}
&
\begin{tikzpicture}[scale=.7, every node/.style={scale=.7}]
\draw (-2,3) -- (0,1.5) -- (-1,3);
\draw (0,3) -- (0,1.5) -- (1,3);
\draw (0,1.5) -- (2,3);
\draw (-2,3.3) circle (.3) node{12};
\draw (-1,3.3) circle (.3) node{28};
\draw (0,3.3) circle (.3) node{4};
\draw (2,3.3) circle (.3) node{16};
\draw (1,3.3) circle (.3) node{7};
\draw (0,1.2) circle (.3) node{9};
\draw (1,3.6) -- (.7,5) -- (1.3,5) -- cycle;
\draw (0,3.6) -- (-.3,4.5) -- (.3,4.5) -- cycle;
\draw (-1,3.6) -- (-1.3,5.5) -- (-.7,5.5) -- cycle;
\draw (2,3.6) -- (2.3,4) -- (1.7,4) -- cycle;
\draw (-2,3.6) -- (-1.7,6) -- (-2.3,6) -- cycle;
\draw (1,5.3) node{$T_1$};
\draw (0,4.8) node{$T_2$};
\draw (-1,5.8) node{$T_3$};
\draw (2,4.3) node{$T_4$};
\draw (-2,6.3) node{$T_5$};
\end{tikzpicture}
\\
\multicolumn{2}{c}{(b) Shuffling of the same tree when $K=5$: Operation $2$ is performed}
\end{tabular}
\caption{Examples of the shuffling operation. The operation is different in both cases, since in the second case the vertex labelled $9$ has a child with label $4 \leq K$.}
\label{fig:shuffling}
\end{figure}

The main interest of this shuffling is that, if $T(\rFZ)$ is the dual tree obtained from a uniform minimal factorization of the $n$-cycle by the Goulden-Yong bijection, then by Theorem \ref{thm:poissontree} for any $K$, $T^{(K)}(\rFZ)$ has the law of a $Po(1)$-GW tree conditioned to have $n$ vertices, where the root has label $1$ and the other vertices are uniformly labelled from $2$ to $n$.
In what follows, we disregard the fact that the root is forced to have label $1$ and consider $T^{(K)}(\rFZ)$
as a conditioned $Po(1)$-GW tree with uniform labels on vertices;
we let the reader  check that this has no influence on the validity of the results.

Set from now on, for all $n$, $K_n=n^{3/4}$. We use the notation of \cref{ssec:framework}.
In particular, $(A_i,B_i,C_i)_{1 \leq i \leq g}$ are i.i.d. uniform triples of points
on the unit circle and $R^g$ is the composition of the $g$ rotations associated with these triples.

\begin{proposition}
\label{prop:shuffling}
Let $\rFZ$ be a uniform minimal factorization of the $n$-cycle. Then, for any $g \geq 0$, 
\begin{equation}
\label{eq:shuffling}
\sup_{1 \leq k \leq n} d_H\left(R^g\left(\bL_k\left(T(\rFZ)\right)\right), R^g\left(\bL_k\left( T^{(K_n)}(\rFZ)\right)\right)\right) \underset{n \rightarrow \infty}{\overset{\P}{\rightarrow}} 0.
\end{equation}
\end{proposition}

\begin{proof}
For $g=0$, this is essentially a particular case of \cite[Lemma 4.6]{thevenin2019geometric}:
the latter gives a bound on the Skorokhod distance
between the processes $(\bL_{u \sqrt{n}}\left(T(\rFZ)\right))_{u \geq 0}$ and $(\bL_{u\sqrt{n}}\left( T^{(K_n)}(\rFZ)\right)_{u \geq 0}$, and not a bound for the supremum over all times $k$ as stated here, but such a bound can be found in the proof after \cite[Equation (8)]{thevenin2019geometric}.

 Let us now fix $g \geq 1$. We first prove that for any fixed $M>0$:
\begin{equation}
\label{eq:rotationsandshuffling}
\sup_{1 \leq k \leq M \sqrt{n}} d_H\left(R^g\left(\bL_k\left(T(\rFZ)\right)\right), R^g\left(\bL_k\left( T^{(K_n)}(\rFZ)\right)\right)\right) \underset{n \rightarrow \infty}{\overset{\P}{\rightarrow}} 0.
\end{equation} 

Fix $M,\delta,\eps>0$. 
We note that since Lemma \ref{lem:uniformancestrallines}
does not depend on the planar structure of the tree, 
we can apply it to both $T(\rFZ)$ and $T^{(K_n)}(\rFZ)$.
Therefore there exists $\eta>0$ such that with probability $> 1-\delta$ 
there is
\begin{itemize}
\item no chord of length $<\eta$ in $\bL_{M \sqrt{n}}\left(T(\rFZ)\right)$
around an element of $E := \bigcup_{i=1}^g \{A''_i, B''_i, C''_i\}$;
\item and no chord of length $<\eta$ in $\bL_{M \sqrt{n}}\left(T^{(K_n)}(\rFZ)\right)$
around an element of $E$.
\end{itemize}
Let $\eps_0=\min(\eps,\eta)$.
In the following we call {\em long chord} a chord of length $>\eps_0$.
Let $(c_{j,n})_{j \le J_n}$, $(c^S_{j,n})_{j \le J^S_n}$  and $(c_j)_{j \le J}$ be the long chords of
$\bL_{M \sqrt{n}}\left(T(\rFZ)\right)$, $\bL_{M \sqrt{n}}\left( T^{(K_n)}(\rFZ) \right)$ and $\bL_M^{(2)}$, respectively
(a priori, $J_n$, $J^S_n$ and/or $J$ can be infinite).
Chords in $\bL_{M \sqrt{n}}\left(T(\rFZ)\right)$ and $\bL_{M \sqrt{n}}\left( T^{(K_n)}(\rFZ) \right)$
have the same lengths (shuffling does not change the sizes of subtrees of vertices with labels $\le K_n$,
which correspond to the lengths of these chords) so in particular $J_n=J^S_n$ for all $n$.

Since $T^{(K_n)}(\rFZ)$ is a size-conditioned $Po(1)$-GW tree with uniform labelings, we can use the chord-by-chord correspondence of \cref{ssec:chords_chords}.
In particular, with probability $\geq 1-\delta$, $J$ is finite and $J_n=J$ for $n$ large enough. Since we know that \eqref{eq:rotationsandshuffling} holds in genus $0$, we can reindex $(c^S_{j,n})_{j \le J}$ so that, for each $j \le J$, $c_{j,n}$ and $c^S_{j,n}$ both converge to $c_j$ as $n \rightarrow \infty$.

On the other hand, since the points $(A''_i, B''_i, C''_i)_{i \le g}$ are i.i.d. uniform, there exists $\eps_1>0$ such that, with probability $\ge 1-\delta$, no point of $E$ is at distance $<\eps_1$ from an extremity of one of the chords $(c_j)_{j \le J}$.
W.l.o.g., we can take $\eps_1 < 2\eps/s_{\eps_0}$,
where $s_{\eps_0}$ is the constant given in \cref{lem:chorddistances}.

We claim that, for $n$ large enough, each of the following events happens
with probability $\ge 1-\delta$ (and thus both of them happen simultaneously with probability $\ge 1-2\delta$):
\begin{align*}
\sup_{j \le J} d_H(c^S_{j,n},c_j) & <\eps_1/2,\\
\sup_{j \le J} d_H(c^S_{j,n},c_{j,n})  & <\eps_1/2.
\end{align*}
Indeed, the first inequality comes from the chord-by-chord correspondence of Section~\ref{ssec:chords_chords}, and the second one from \eqref{eq:shuffling} in the case $g=0$. Let us assume from now on that these events hold.
This implies that, for any $j\le J$, the extremities of $c_{j,n}$ and $c^S_{j,n}$
are in the same connected components of $\mathbb S^1 \setminus E$.
Using \cref{lem:chorddistances} (iii), 
we have
\[d_H\big(R^g(c_{j,n}),R^g(c^S_{j,n})\big) \le s_{\eps_0} \eps_1/2 < \eps.\]
and \eqref{eq:rotationsandshuffling} follows.

We now want to extend the supremum over $k \le M \sqrt{n}$ in \eqref{eq:rotationsandshuffling} to a supremum over $k \le n$. The key idea is that not much happens after scale $\sqrt{n}$, as the sieve is already close to $\Sieve^g_\infty$.
We recall that $T^{(K_n)}(\rFZ)$ has the distribution of a conditioned $Po(1)$-GW tree with uniform labels on the vertices.
Therefore, by Theorem \ref{thm:convergence_rotated_GW}, one has
jointly, for any $M$,
\begin{align}
\label{eq:rotationsandshuffling2}
R^g \left(\bL_{M \sqrt{n}}\left( T^{(K_n)}(\rFZ)\right)\right) & \underset{n \rightarrow \infty}{\overset{(d)}{\rightarrow}} R^g \left(\bL_{M\sqrt{2}}^{(2)} \right) \text{ and } \\
R^g \left(\bL \left( T^{(K_n)}(\rFZ)\right)\right) & \underset{n \rightarrow \infty}{\overset{(d)}{\rightarrow}} R^g \left(\bL_\infty^{(2)} \right).
\label{eq:rotationsandshuffling3}
\end{align}
By Skorokhod representation theorem, let us assume that these convergences hold almost surely and jointly.
Fix $\eps,\delta>0$.
We can choose $M$ large enough, so that $d_H\left(R^g \big(\bL_{M\sqrt{2}}^{(2)} \big),R^g \big(\bL_\infty^{(2)} \big)\right) \le \eps$
with probability $1-\delta$ (thanks to \cref{prop:rotationofcontinuouslaminations}).
For $n$ large enough, in each of the two convergences above,
the left-hand side is at distance at most $\eps$ from the right-hand side with probability $1-\delta$.
We deduce, that for $n$ large enough, it holds that
\[d_H\left( R^g \left(\bL_{M \sqrt{n}}\left( T^{(K_n)}(\rFZ)\right)\right),
R^g\left(\bL \left( T^{(K_n)}(\rFZ)\right)\right) \right) \le 3\eps, \]
with probability at least $1-3\delta$.
Using the monotonicity of the lamination process $(\bL_c\left( T^{(K_n)}(\rFZ)\right))_{c \geq 0}$, this implies
that
\begin{equation}
\label{eq:rotationsandshuffling4}
\sup_{M \sqrt{n} \le k \le n}  d_H\left(R^g\left(\bL_{k}\left( T^{(K_n)}(\rFZ)\right)\right),
R^g\left(\bL_{M \sqrt{n}}\left( T^{(K_n)}(\rFZ)\right)\right) \right)\le 3\eps
\end{equation}
with probability at least  $1-3\delta$.

We can proceed similarly for $R^g\left(\bL_k\left(T(\rFZ)\right)\right)$. From \cref{eq:rotationsandshuffling,eq:rotationsandshuffling2}, we get that 
\[R^g \left(\bL_{M \sqrt{n}}\left( T(\rFZ)\right)\right) \underset{n \rightarrow \infty}{\overset{(d)}{\rightarrow}} R^g \left(\bL_{M\sqrt{2}}^{(2)} \right). \]
Eq.~\eqref{eq:rotationsandshuffling3} also holds without shuffling since we know that, forgetting the labeling, $T(\rFZ)$ is distributed as a size-conditioned $Po(1)$-Galton--Watson tree 
(\cref{thm:poissontree}). Thus we have that
\begin{equation}
\label{eq:rotationsandshuffling5}
\sup_{M \sqrt{n} \le k \le n}  d_H\left(R^g\left(\bL_{k}\left( T(\rFZ)\right)\right),
R^g\left(\bL_{M \sqrt{n}}\left( T(\rFZ)\right)\right) \right)\le 3\eps,
\end{equation}
with probability at least  $1-3\delta$.

The proposition follows from \cref{eq:rotationsandshuffling,eq:rotationsandshuffling4,eq:rotationsandshuffling5}
\end{proof}

\subsection{From a chord configuration to a tree}
\label{ssec:chordtotree}
We now use the Goulden-Yong bijection exposed in Section \ref{ssec:gouldenyong} to associate to a minimal factorization $f$ of the $n$-cycle a tree $T(f)$ with $n$ labelled vertices.
By Theorem \ref{thm:poissontree}, when the factorization is uniform, then the tree is distributed as a size-conditioned $Po(1)$-GW tree, which helps study its behaviour. In genus $0$, this was fully investigated in \cite{thevenin2019geometric}. Here, we explain how to extend it to greater genera.

\begin{proposition}
\label{prop:chordstotree}
Let $\rFZ$ be a uniform minimal factorization of the $n$-cycle. Then:
\begin{align*}
\sup_{1 \leq k \leq n} d_H\left( R^g(\Sieve_k(\rFZ)), R^g(\bL_k(T(\rFZ))) \right) \underset{n \rightarrow \infty}{\overset{\P}{\rightarrow}} 0.
\end{align*}
\end{proposition}

This allows us to study the lamination obtained from a tree, instead of the chord configuration obtained from the factorization.

\begin{proof}

As in the proof of Proposition \ref{prop:shuffling}, we first prove that, for any $M>0$, 
\begin{equation}
\label{eq:RotatedSievesClosed}
\sup_{1 \leq k \leq M \sqrt{n}} d_H\left( R^g(\Sieve_k(\rFZ)), R^g(\bL_k(T(\rFZ))) \right) \underset{n \rightarrow \infty}{\overset{\P}{\rightarrow}} 0.
\end{equation}

The proof is similar (though slightly simpler since the rotation points are the same on both sides)
 to that of \cref{prop:shuffling}. We only give the main lines.
 
Fix $M,\delta,\eps, \eta>0$. We use again Lemma \ref{lem:uniformancestrallines} for $T(\rFZ)$, together with \cref{lem:SeparatingChords_AncestralLine}, to ensure 
that, with probability $1-\delta$,
there is no small chord (i.e. of length smaller than $\eta$) in $R^g(\Sieve_{M \sqrt{n}}(\rFZ))$ or $R^g(\bL_{M \sqrt{n}}(T(\rFZ)))$ around any element of $E$.
 
Set $\eps_0=\min(\eps,\eta)$ and consider the large chords $(c_j)_{j \le J}$ in $R^g(\bL_{M \sqrt{n}}(T(\rFZ)))$, $(c'_j)_{j \le J'}$ in $R^g(\Sieve_{M \sqrt{n}}(\rFZ))$ (large meaning of length $> \eps_0$).
For $n$ sufficiently large, the extremities of large chords $(c_j)_{j \in J}$ in $R^g(\bL_{M \sqrt{n}}(T(\rFZ)))$ are far from elements of $E$ (this was proved in the proof of Proposition \ref{prop:shuffling}), and thus we can apply \cref{lem:chorddistances} (iii) to $c_j$ and $c'_j$, for any $j \leq J$. Since $d_H(c_j,c'_j)=o(1)$ uniformly in $j$ (e.g. by \cite[Lemma $4.4$]{thevenin2019geometric}), we have $d_H\big(R^g(c_j),R^g(c'_j)\big)=o(1)$, concluding the proof of~\eqref{eq:RotatedSievesClosed}.
 
 We then need to extend the supremum over $k \leq M \sqrt{n}$ in~\eqref{eq:RotatedSievesClosed}
 to a supremum over $k \le n$. To this end we first 
 note  that \cref{eq:RotatedGW,eq:shuffling,eq:RotatedSievesClosed}
 imply that $R^g(\Sieve_{M\sqrt{n}}(\rFZ))$ converges towards $R^g(\bL_{M\sqrt{2}}^{(2)})$.
Then the same 
argument as in the end of the proof of \cref{thm:convergence_rotated_GW}
implies the convergence of
 $R^g(\Sieve(\rFZ))$ to $R^g(\bL_\infty^{(2)})$.
 The rest of the proof is  similar to that of \cref{prop:shuffling}.
 \end{proof}

\subsection{From a uniform factorization of genus $g$ to a uniform minimal factorization}
\label{ssec:genusgtogenus0}

Let $\rF$ be a uniform random factorization in $\setF$
and $\rFZ$ a uniform random minimal factorization.
By \cref{corol:RandomGeneration}, we can construct
$\rF$ and $\rFZ$ on the same probability space so that 
$\rF=\Lambda^g(\rFZ)$ with high probability. We use this construction throughout this section.

We recall the action of $\Lambda^g$. For $1 \leq i \leq g$, we first choose $\aaa_i<\bbb_i<\ccc_i$ a uniform random increasing triple of $[ 1,n ]$, and add to the factorization the transpositions $(\aaa_i \,\bbb_i)$ and $(\aaa_i \,\ccc_i)$ at uniformly chosen times $\vvv_i,\www_i$. 
Then we replace each transposition $(j \,h)$ by the transposition $(\sigma_i(j) \,\sigma_i(h))$,
where $\sigma_i$ is of the form given in Proposition \ref{prop:labelingalgorithm}.

Finally, we set
 $(A_i, B_i, C_i) := (e^{-2\pi i (\aaa_i+\epsilon_{i,1})/n}, e^{-2\pi i (\bbb_i+\epsilon_{i,2})/n}, e^{-2\pi i (\ccc_i+\epsilon_{i,3})/n})$, where the variables $\epsilon_{i,1}, \epsilon_{i,2}, \epsilon_{i,3}$ are i.i.d., independent of $\aaa_i, \bbb_i, \ccc_i$ and uniform on $(-1, 0)$. With this construction, $(A_i, B_i, C_i)_{1 \le i \le g}$ are independent uniform triples of points on the circle, with the condition that $(1,A_i, B_i, C_i)$ appear in clockwise order around the circle.
We use the notation of \cref{ssec:framework} with respect to these triples $(A_i, B_i, C_i)$;
in particular, $R^g$ is the composition of the $g$ consecutive rotations $R_{A_i, B_i, C_i}$.

The following proposition, which is the main result of this section, states that the sieve-valued process $(\Sieve_k(\rF))_{0 \leq k \leq n}$ can be approximated by the lamination-valued process associated to $\rFZ$, rotated $g$ times.

\begin{proposition}
\label{prop:fromuniftounif}
We have \begin{align*}
\sup_{1 \leq k \leq n} d_H\left(\Sieve_k(\rF), R^g\left(\Sieve_k(\rFZ)\right)\right) \underset{n \rightarrow \infty}{\overset{\P}{\rightarrow}} 0.
\end{align*}
\end{proposition}

\begin{proof}
Recall that we have fixed $K_n=n^{3/4}$ as threshold for the shuffling procedure
in \cref{ssec:shuffling}. We first prove our estimates until the threshold, i.e. that
\begin{equation}
\label{eq:ClosedUntilKn}
\sup_{1 \leq k \leq K_n} d_H\left(\Sieve_k(\rF), R^g\left(\Sieve_k(\rFZ)\right)\right) \underset{n \rightarrow \infty}{\overset{\P}{\rightarrow}} 0.
\end{equation}
We first remark that, with high probability, 
at each of the $g$ steps of the algorithm, 
the two transpositions $(\aaa_i \,\bbb_i)$ and $(\aaa_i \,\ccc_i)$ are added to the factorization $\rF$ at times  $\vvv_i$ and $\www_i$
which are larger than $K_n$. 
Therefore they do not appear in $\Sieve_k(\rF)$ for $k \le K_n$,
and $\Sieve_k(\rF)$ is obtained from $\Sieve_k(\rFZ)$ by replacing 
the chord associated to a transposition $(j \,h)$ by that of $(\sigma^g(j) \,\sigma^g(h))$,
where $\sigma^g$ is the composition $\sigma_g \circ \dots \circ \sigma_1$.

For $i \le g$, let us denote
$\tilde \sigma_i$ the permutation defined as follows
\begin{equation}
\label{eq:tilde_sigma}
\tilde \sigma_i(j)=\begin{cases}
j &\text{ if $j\leq \aaa_i$ or $j> \ccc_i$;}\\
j+\ccc_i-\bbb_i&\text{  if $\aaa_i< j\le \bbb_i$;}\\
j-\bbb_i+\aaa_i&\text{  if $\bbb_i< j\leq \ccc_i$,}
\end{cases}.\end{equation}
Note that $\tilde \sigma_i$ is a discrete version of the rotation $R_i$
associated with the triple $(A_i,B_i,C_i)$

Comparing with Proposition \ref{prop:labelingalgorithm} (and taking the same notations), 
we see that, uniformly in $j$,
one has $\tilde\sigma_i(j)=\sigma_i(j)+\O_P(1)$ except for $j$ in a set of size $\O_P(1)$
(namely, except for $j$ between $\aaa_i$ and $\aaa'_i$ or $\aaa''_i$, 
between $\bbb_i$ and $\bbb'_i$ or between $\ccc_i$ and $\ccc'_i$;
note that all such $j$ are at distance $\O_P(1)$ of one of the rotation point $\aaa_i$, $\bbb_i$ or $\ccc_i$, which will be useful later in the proof).
Denoting by  $\tilde \sigma^g$  the composition $\tilde \sigma_g \circ \dots \circ \tilde \sigma_1$,
we have that, uniformly on $j$, $\tilde\sigma^g(j)=\sigma^g(j)+\O_P(1)$ except for $j$ in a set of size $\O_P(1)$.
We denote by $G \subseteq [1,n]$ the set of {\em good} $j$ for which this holds, 
and $\overline G$ its complement in $[1, n]$. We call {\em good transposition} a transposition $(j \,h)$ such that $j,h \in G$, and {\em bad transposition} a transposition that is not good.

If $j$, $h$ are both in the set $G$, then the chord associated with $(\sigma^g(j) \,\sigma^g(h))$
is at distance $\O(1/n)$ from the one associated to $(\tilde \sigma^g(j) \,\tilde \sigma^g(h))$.
The latter corresponds to $R^g(c)$, where $c$ is the chord of $\Sieve(\rFZ)$
associated to the transposition $(j \,h)$.
Hence, \eqref{eq:ClosedUntilKn} follows if we can show that, with high probability,
all transpositions $(j \,h)$ appearing before time $K_n$ in $\rFZ$ are good.

By construction, all points of $\overline G$ are at distance $\O_P(1)$
of one of the  points 
\[\tilde\sigma_1^{-1} \circ \dots \circ \tilde\sigma_{i-1}^{-1} (\aaa_i),\ \ \tilde\sigma_1^{-1} \circ \dots \circ \tilde\sigma_{i-1}^{-1} (\bbb_i),
\text{ or }\ \tilde\sigma_1^{-1} \circ \dots \circ \tilde\sigma_{i-1}^{-1} (\ccc_i),\]
for some $i \le g$.
Since these points are independent and uniform, and there are only $3g$ of them, they are at distance $\O_P(n/K_n)$ of the support of the $K_n$ first transpositions of $\rFZ$ (for the discrete distance on $[n]$ modulo $n$). This proves that, w.h.p., there is no transposition $(j \,h)$ 
 before time $K_n$ in $\rFZ$,
with either $j$ or $h$ in $\overline G$.
This ends the proof of \eqref{eq:ClosedUntilKn}.
\bigskip

We still need to extend the supremum over $k \le K_n$ in \eqref{eq:ClosedUntilKn}
to a supremum over $k \le n$. 
The argument here is different from the proofs of similar statements. Indeed, $\Sieve(\rF)$ is not necessarily the image of a lamination by $g$ rotations, so one cannot use a maximality argument.

Using \cref{thm:convergence_rotated_GW,prop:shuffling,prop:chordstotree}
we have $R^g\big(\Sieve_{M\sqrt n}(\rFZ)\big) \to  R^g(\bL_{M\sqrt 2}^{(2)})$ for any $M>0$
and $R^g\big(\Sieve(\rFZ)\big) \to  R^g(\bL_\infty^{(2)})$.
By monotonicity of $(\Sieve_k(\rFZ))_{k \geq 0}$ and using the convergence $\bL_{M\sqrt 2}^{(2)} \to \bL_\infty^{(2)}$,
this implies
\begin{equation}
\label{eq:Rgk_L2}
\sup_{K_n \le k\le n} d_H \Big(R^g\big(\Sieve_k(\rFZ)\big), R^g(\bL_\infty^{(2)}) \Big) 
\underset{n \rightarrow \infty}{\overset{\P}{\rightarrow}} 0.
\end{equation}

We now consider $\Sieve(\bm F_n^g)$. Note that $\bm F_n^g$ contains two types of transpositions.
\begin{itemize}
\item First, we have the transpositions coming from the initial factorization $\rFZ$,
which we call {\em initial transpositions}.
These are  the form $(\sigma^g(j) \,\sigma^g(h))$ where $(j \,h)$ is
a transposition in the initial factorization $\rFZ$.
\item Secondly, we have the transpositions added during an application of $\Lambda$
in the generation algorithm, which we call {\em added transpositions}.
 Namely, at time $i$, we add the transpositions
$(\aaa_i \,\bbb_i)$ and $(\aaa_i \,\ccc_i)$ and conjugate these transpositions with $\sigma_i$.
 After the $g-i$ remaining steps,
these are transformed into $(\sigma_{i}^g(\aaa_i) \,\sigma_{i}^g(\bbb_i))$
and $(\sigma_{i}^g(\aaa_i) \,\sigma_{i}^g(\ccc_i))$,
where $\sigma_{i}^g$ is the composition $\sigma_g \circ \dots \sigma_{i}$.
\end{itemize}
Among the initial transpositions, we further distinguish the ones for which both $j$ and $h$
are good in the sense defined above.
For $k \le n$, we denote $\Sieve_k^{good}(\bm F_n^g)$ the subset of chords in $\Sieve_k(\bm F_n^g)$
  associated to the good initial transpositions.
 We also denote by $\Sieve_k^{good}(\rFZ)$ the corresponding set of chords in $\Sieve_k(\rFZ)$.
 Using the same argument as to prove \eqref{eq:ClosedUntilKn},
 we have
\begin{equation}
\label{eq:goodclose}
\sup_{0 \leq k \leq n} d_H\left( \Sieve^{good}_k(\bm F_n^g), R^g(\Sieve^{good}_k(\rFZ)) \right)
 \underset{n \rightarrow \infty}{\overset{\P}{\rightarrow}} 0.
\end{equation}
We also know that with high probability, all chords appearing before time $K_n$ are good.
Hence, for $k\ge K_n$, we have with high probability
 \[ R^g\big(\Sieve_{K_n}(\rFZ)\big) \subseteq R^g\big(\Sieve^{good}_k(\rFZ)\big)\subseteq R^g\big(\Sieve(\rFZ)\big).\]
 Using \eqref{eq:Rgk_L2}, this implies
 \[
\sup_{K_n \le k\le n} d_H \Big(R^g\big(\Sieve^{good}_k(\rFZ)\big), R^g(\bL_\infty^{(2)}) \Big) 
\underset{n \rightarrow \infty}{\overset{\P}{\rightarrow}} 0.
\]
and therefore, using \eqref{eq:goodclose},
\[\sup_{K_n \le k\le n} d_H \Big(\Sieve^{good}_k(\bm F_n^g), 
 R^g(\bL_\infty^{(2)})\Big) 
\underset{n \rightarrow \infty}{\overset{\P}{\rightarrow}} 0.
\]
We claim and will prove below that we can replace $\Sieve^{good}_k(\bm F_n^g)$
by $\Sieve(\bm F_n^g)$ in the equation above, i.e. that
\begin{equation}
\label{eq:Xk_LargeK}
\sup_{K_n \le k\le n} d_H \Big(\Sieve_k(\bm F_n^g), 
 R^g(\bL_\infty^{(2)})\Big) 
\underset{n \rightarrow \infty}{\overset{\P}{\rightarrow}} 0.
\end{equation}
The proposition now follows from \cref{eq:ClosedUntilKn,eq:Rgk_L2,eq:Xk_LargeK}
\end{proof}

Here is the proof of the technical point left aside above.

\begin{proof}[Proof of \eqref{eq:Xk_LargeK}]
Since we know that $\Sieve^{good}_k(\bm F_n^g)$ is close to $ R^g(\bL_\infty^{(2)})$,
we only need to prove that chords associated with the initial bad transpositions
and with the added transpositions are well approximated by chords in $R^g(\bL_\infty^{(2)})$.

We start with the added transpositions. With high probability, the distance between 
$(\aaa_i,\bbb_i,\ccc_i)$ and further rotation points becomes large as $n$ tend to infinity,
and thus
we have
\[\sigma_i^g(\aaa_i)=\tilde \sigma_g \circ \dots \circ \tilde \sigma_{i+1} (\sigma_i(\aaa_i))+\O(1).\]
We also note that depending on how $\aaa_i$ compares with the points $a'$ and $a''$
of \cref{prop:labelingalgorithm}, we have 
\[\sigma_i(\aaa_i) = \begin{cases}\phantom{or } \aaa_i +\O(1) =\tilde \sigma_i(\aaa_i) +\O(1);\\
\text{or } \ccc_i +\O(1) =\tilde \sigma_i(\bbb_i) +\O(1);\\
\text{or } \aaa_i+\ccc_i-\bbb_i +\O(1) =\tilde \sigma_i(\ccc_i) +\O(1).\\
\end{cases}\]
Calling $\aaa'_i=\tilde \sigma_g \circ \dots \circ \tilde \sigma_i(\aaa_i)$,
and defining $\bbb'_i$ and $\ccc'_i$ in the same way, we
have that $\sigma_i^g(\aaa_i)$ is at distance $\O(1)$ from either 
$\aaa'_i$, $\bbb'_i$ or $\ccc'_i$.
The same holds for $\sigma_i^g(\bbb_i)$ and $\sigma_i^g(\ccc_i)$.

Letting $A'_i$ be the point on the circle corresponding to $\aaa'_i$,
we have $A'_i=R_{i}^g(A_i)$, matching the notation of \cref{ssec:framework}.
Similar remarks and notation hold for $B$ and $C$.
We conclude that the chord associated with the added transposition
$(\sigma_{i+1}^g(\aaa_i) \,\sigma_{i+1}^g(\bbb_i))$ is close with high probability to one of the chords
$[A'_i, C'_i], [B'_i, C'_i]$ and $[A'_i, B'_i]$.
Since these chords belong a.s. to $R^g( \bL_\infty^{(2)})$ (by \cref{lem:Triangles_In_XgInfty}),
this settles the case of added transpositions.

We now consider initial bad transpositions.
For such a transposition $(j \,h)$ in $\rFZ$, there exists $i$ such that 
$\sigma_{i-1} \circ \dots \circ \sigma_1(j)$ (or $\sigma_{i-1} \circ \dots \circ \sigma_1(h)$,
but say $\sigma_{i-1} \circ \dots \circ \sigma_1(j)$ w.l.o.g.)
is sent by $\sigma_i$ and $\tilde \sigma_i$ to distant values.
But by construction of $\sigma_i$ (see \cref{prop:labelingalgorithm}), this can happen
only if $\sigma_{i-1} \circ \dots \circ \sigma_1(j)$ is the label of a vertex in the subtree
rooted in either $A_i$, $B_i$ or $C_i$ in the map associated to $(R_i^g)^{-1}(\rF)$.
This implies that $\sigma_{i-1} \circ \dots \circ \sigma_1(h)$ is the label of a descendant or of the parent of either
$A_i$, $B_i$ or $C_i$ and thus is at distance $\mathcal O(1)$ of $a_i$, $b_i$ and $c_i$ 
(see \cref{lem_small_desc_trees}).
We can then apply the same argument as above and conclude
that the chord associated with $(\sigma^g(j) \,\sigma^g(h))$ in $\Sieve(\bm F_n^g)$
is also close to either $[A'_i, C'_i], [B'_i, C'_i]$ and $[A'_i, B'_i]$.
\end{proof}

\subsection{Concluding the proof of the scaling limit theorems}

We end this section by providing proofs of Theorems \ref{thm:main} and \ref{thm:mainprocess} using the results of the previous subsections.

\begin{proof}[Proof of Theorems \ref{thm:main} and \ref{thm:mainprocess}]
From Propositions \ref{prop:shuffling}, \ref{prop:chordstotree} and \ref{prop:fromuniftounif},
we have
\[\sup_{1 \leq k \leq n} d_H\Bigg(\Sieve_k(\rF),R^g\left(\bL_k\left( T^{(K_n)}(\rFZ)\right)\right)\Bigg)
\underset{n \rightarrow \infty}{\overset{\P}{\rightarrow}}0.  \]
Since $T^{(K_n)}(\rFZ)$ is a size-conditioned $Po(1)$-GW tree with a uniform labeling on the vertices
and since the rotation points $(A_i,B_i,C_i)$ are i.i.d. uniform triples of vertices independent of $T^{(K_n)}(\rFZ)$,
we can apply \cref{thm:convergence_rotated_GW} to $R^g\left(\bL_k\left( T^{(K_n)}(\rFZ)\right)\right)$.
This proves \cref{thm:main,thm:mainprocess}
\end{proof}

\section{The genus process}
\label{sec:genus_process}

The main goal in this section is to understand the birth of the successive genera in a uniform factorization of genus $g$ of the $n$-cycle, thus proving Theorem \ref{thm:asymptoticgenus}. Our method relies on a fine study of the sieve $\Sieve(\bm F_n^g)$, and in particular of the chords that cross inside the disk. The main tool in this study is a slightly adapted version of Aldous' \textit{reduced tree} \cite{aldous1993} constructed from a finite tree with marked vertices.

\subsection{Definitions and first properties}
\label{ssec:def_genus}

We recall from the introduction that given a list $f := (t_1, \ldots, t_k)$ of transpositions,
$G(f)$ is the minimal genus of a factorization of the long cycle $(1\, \cdots\, n)$
extending $f$.
The genus $G(f)$ actually has a more compact equivalent definition.
For $\pi \in \kS_n$, let $\deg[\pi]$ be the minimal number of transpositions
needed to factorize $\pi$; this is well-known to be $n-\ka(\pi)$,
where $\ka(\pi)$ is the number of cycles of $\pi$.
We leave to the reader the proof of the following fact.
\begin{lemma}
\label{lem:computation}
Let $n,k \geq 1$ and $f=(t_1,\dots,t_k) \in \kT_n^k$. Then 
\begin{align}
\label{eq:formule_Gf}
n-1 + 2 G(f) = k+\deg[(t_1\cdots t_k)^{-1} (1 \, 2 \, \cdots \, n)].
\end{align}
\end{lemma}
This can in turn be interpreted geometrically.
For a partial factorization $f$, we let $M(f)$ be the following combinatorial map
(see Fig. \ref{fig:moff} for an example).
\begin{itemize}
\item Vertices are indexed by integers from $1$ to $n$;
it is convenient to think of them as points on a circle.
\item Edges are of two kinds.
For each $i \le n-1$, there is an edge between $i$ and $i+1  \mod n$
and one between $n$ and $1$ (we think of these edges as arcs of circle).
 Also, each transposition in $f$ is seen as an edge
(which we think of as drawn inside the circle with potential crossings).
\item The cyclic order around any vertex $i$ is given as follows:
first the edge from $i-1 \mod n$ to $i$ then the $t_j$ containing $i$ in {\em decreasing} order of their indices $j$,
and finally the edge from $i$ to $i+1 \mod n$.
\end{itemize}

\begin{figure}
\includegraphics[scale=.5]{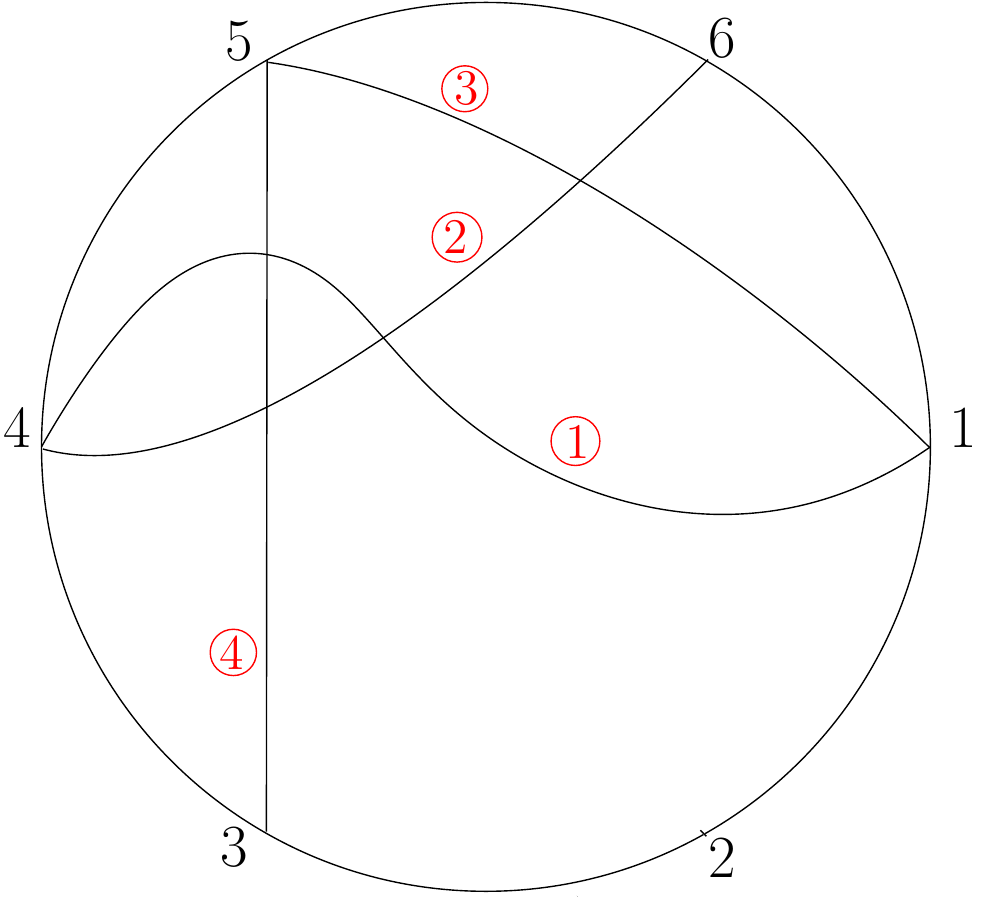}
\caption{The map $M(f)$ associated with $f:=(1 \, 4) (4 \, 6) (1 \, 5) (3 \, 5)$.
 It has genus $G(f)=G(M(f))=2$, which can be seen directly or by \cref{eq:formule_Gf} }
\label{fig:moff}
\end{figure}

\begin{lemma}
With the notation above, $G(f)$ is the genus of $M(f)$.
\end{lemma}
We write $G(M)$ for the genus of a map $M$,
so that the above statement reads $$G(f)=G(M(f)).$$
\begin{proof}
Write $f=(t_1,\dots,t_k)$. The map $M(f)$ has $n$ vertices and $n+k$ edges. Let us describe its faces.
We start at the edge from $i-1$ to $i$ inside the circle 
and follow the contour of the corresponding face
until we find another arc of circle, say from $h-1$ to $h$ (with the convention $1-1=n$).
We claim that $h-1=(t_k \dots t_1) (i)$. The proof of this claim is left to the reader;
it is similar to the proof of Lemma~\ref{lem:Unicellular+Relabeling}. As an example, on Fig. \ref{fig:moff}, for $i=6$ one follows the edges $(4 \, 6)$ and $(1 \, 4)$ before finding the arc $(1 \, 2)$. One can check that, indeed, $(t_k \dots t_1)(6)=1$ (recall that we apply permutations from left to right).
The above equality rewrites as $h=\big(t_k \dots t_1 (1 \dots n)\big) (i)$.
The next arc of circle we will encounter by keeping turning around the same face goes from $h'-1$ to $h'$,
where $h'=\big(t_k \dots t_1 (1 \dots n)\big)^2 (i)$, and so on.
We will come back to the arc from  $i-1$ to $i$ when we have completed the cycle 
of $t_k \dots t_1 (1 \dots n)$ containing $i$.

We conclude that faces of $M(f)$ inside the circle correspond to cycles of $t_k \dots t_1 (1 \dots n)$.
Besides, $M(f)$ has one outer face, hence in total $1 +\kappa \big(t_k \dots t_1 (1 \dots n)\big)$ faces.
Letting $G(M(f))$ be the genus of $M(f)$, Euler formula gives us
\[n-(n+k)+1+\kappa \big(t_k \dots t_1 (1 \dots n)\big) =2 -2G(M(f)).\]
Comparing with \eqref{eq:formule_Gf} gives us the result.
\end{proof}

\begin{remark}
The construction of $M(f)$ is reminiscent -- though different in some aspects --
of a construction of \'Sniady on set partitions \cite[Section 5]{Sniady2006fluctuations}.
\end{remark}

\subsection{Reduction of the factorization}
\label{ssec:Red_facto}

Starting with the map $M(f)$ of a factorization, we perform successively the following operations
 (see \cref{Fig:Reduction}):
\begin{enumerate}
\item removing vertices of degree $2$ (and merging their two incoming edges);
\item splitting vertices of degree $d\ge 4$ into $d-2$ vertices of degree $3$ connected by arcs
of the outer circle;
\item erasing an interior edge, which is not crossing any other edge inside the circle,
and removing its two extremities as in i);
\item removing one of two parallel interior edges (interior edges are parallel if they form a face of degree $4$, together with two arcs of circle) and removing its two extremities as in i).
\end{enumerate}
By performing these operations, we get a  diagram of chords with disjoint extremities,
which we call {\em reduction} of $M(f)$ and denote $\Red(f)$.
All the operations above preserve the genus: indeed,
\begin{itemize}
\item
i) decreases by 1 both the numbers of edges and vertices, while preserving the number of faces ;
\item
ii) increases by $d-2$ both the numbers of edges and vertices, while preserving the number of faces ;
\item 
finally, the edge removals in iii) and iv) decrease by 1 both the numbers of edges and faces, while preserving the number of vertices.
\end{itemize}
Consequently, one has $G(f)=G(\Red(f))$. We call \textit{reduced chord diagrams} all chords diagrams that are equal to their reductions.
\begin{figure}[t]
\[\begin{array}{c|c} 
\begin{tabular}{c}
\includegraphics[height=2cm]{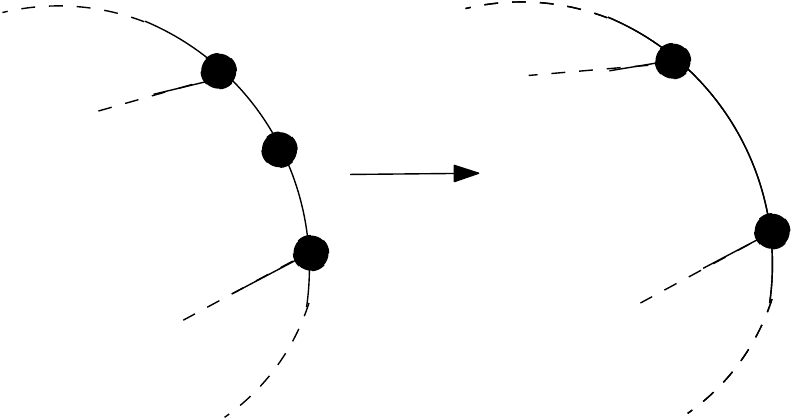}\\ (i) removing degree $2$ vertices
\end{tabular} &
\begin{tabular}{c}
\includegraphics[height=2cm]{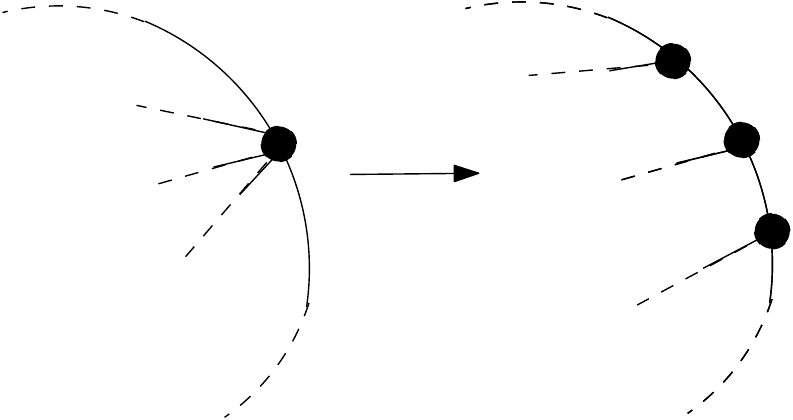}\\ (ii) splitting vertices
\end{tabular} \\
\begin{tabular}{c}
\includegraphics[height=2cm]{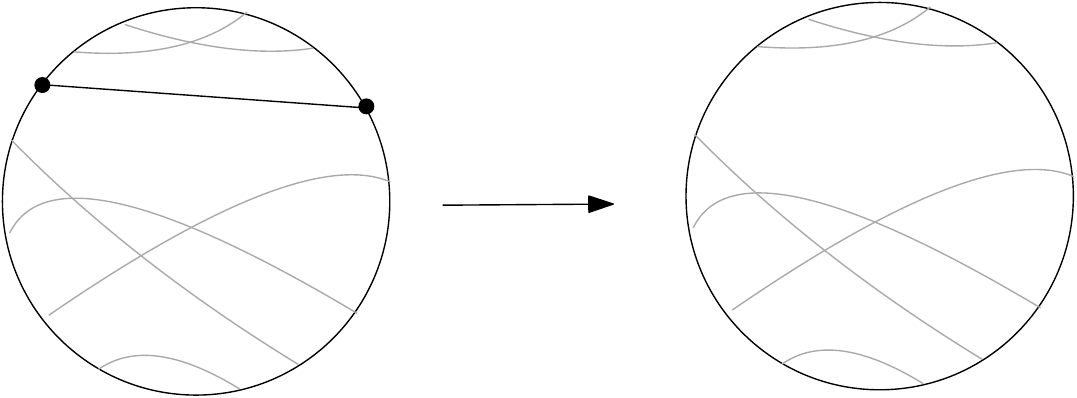}\\ (iii) erasing noncrossing edges
\end{tabular}  &
\begin{tabular}{c}
\includegraphics[height=2cm]{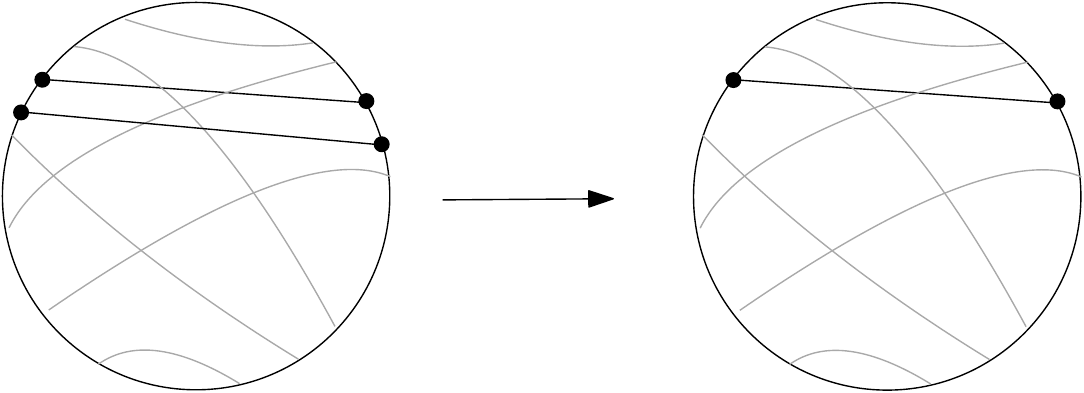}\\ (iv) parallel edges
\end{tabular} 
 \end{array}\]
\caption{The reduction operations}
\label{Fig:Reduction}
\end{figure}

\begin{lemma}
\label{lem:size_reduced_diagram}
A reduced chord diagram of genus $G$ has at most $12 \, G-6$ vertices.
\end{lemma}
\begin{proof}
Consider a reduced chord diagram of genus $G$.
We denote by $v$ its number of vertices,
and by $e$ and $f$ its numbers of {\em internal} edges and faces, respectively. 
There are $e+v$ edges in total and $f+1$ faces. Euler formula gives us 
\[v - (e+v) + (f+1)= 2-2G.\]
Moreover, vertices in a reduced chord diagram are all incident to a single interior edge, yielding
$v=2e$. Also, each internal face is incident to at least $3$ internal edges (if there was only $1$, this edge would be noncrossing; if they were two,
they would be parallel). Thus $3f \le 2e$.
Combining all these equations gives the lemma.
\end{proof}

\subsection{Reduced subtrees}
In his study of limits of random trees, 
Aldous introduced the notion of {\em reduced subtree}
associated to a finite tree $T$ and a subset of its vertices $(A_1,\dots A_k)$.
This notion will be useful for us.
Note that we use here a plane unrooted version of Aldous' reduced subtree.
The following definition is therefore an adaptation of the one given
 in \cite[Section 2.1]{aldous1993}.

\begin{definition}
\label{def:reducedtree}
Let $T$ be a plane tree with $n$ vertices, 
and fix $k \geq 3$. 
Let $A_1, \ldots, A_k$ be $k$ distinct distinguished vertices in $T$.

A vertex $\rho$ in $T$ is a {\em branching point} (w.r.t. $A_1$, \dots, $A_k$)
if there exists $1 \leq u<v<w \leq k$
such that $\rho$ is the unique vertex belonging
to the three paths $\llbracket A_u, A_v \rrbracket$,  
$\llbracket A_v, A_w \rrbracket$
and  $\llbracket A_w, A_u \rrbracket$.

Then the reduced tree $T_{red}(A_1, \ldots, A_k;T)$
is the plane unrooted tree whose vertices are the distinguished vertices
$A_1, \ldots, A_k$ and the branching points.
Two of these vertices, say $x$ and $y$ (which might be distinguished vertices or branching points), are connected by an edge in $T_{red}(A_1, \ldots, A_k;T)$
if there is no branching point on the path between $x$ and $y$ in $T$.
The plane embedding is inherited from that of $T$.
Moreover, we assign a length to the edge $\{x,y\}$ in $T_{red}(A_1, \ldots, A_k;T)$
which is simply the distance between $x$ and $y$ in $T$.
\end{definition}
Note that the output $T_{red}(A_1, \ldots, A_k;T)$ is a geometric tree, i.e. a tree with edge-lengths. An example is given on \cref{fig:reducedtree}
(at the moment, one can disregard edge labels in $T$).
Each edge of $T_{red}(A_1, \ldots, A_k;T)$ corresponds to a path, i.e. a set of edges in $T$.
To avoid confusion, we call  {\em branches} the edges of the reduced tree.
\begin{figure}[!ht]
\[\begin{array}{l c r}
\begin{tikzpicture}[scale=1, every node/.style={scale=0.6}]
\draw (0,0) -- node[left]{$4$} (0,1) -- node[left]{$8$} (-1, 2) 
(0,2) -- node[left]{$6$} (0,1) -- node[left]{$5$} (1,2) -- node[left]{$2$} (1,3) 
(2,3) -- node[left]{$9$} (1,2) -- node[below]{$1$} (3,3) 
(-1, -1) -- node[left]{$10$} (0,0) -- node[left]{$7$} (1,-1) (-1, 2) -- node[left]{$3$} (-1, 3);
\draw[red,fill=white] (-1,-1) circle (.2) node{$A_1$};
\draw[red,fill=white] (-1,2) circle (.2) node{$A_3$} ;
\draw[fill=white] (0,2) circle (.2);
\draw[blue,fill=white] (0,1) circle (.2) node{$B$} ;
\draw[blue,fill=white] (1,2) circle (.2) node{$B'$} ;
\draw[red,fill=white] (1,3) circle (.2) node{$A_4$};
\draw[fill=white] (2,3) circle (.2);
\draw[red,fill=white] (3,3) circle (.2) node{$A_2$};
\draw[fill=white] (0,0) circle (.2);
\draw[fill=white] (1,-1) circle (.2) ;
\draw[fill=white] (-1,3) circle (.2) ;
\end{tikzpicture}

&
\begin{tikzpicture}[scale=1.2, every node/.style={scale=0.6}]
\draw (-1,-1) -- (0,0) node[midway,above] {1};
\draw (0,0) -- (1,-1) node[midway,above] {2};
\draw (1,2) -- (0,1) node[midway,above] {1};
\draw (0,1) -- (-1,2) node[midway,above] {1};
\draw (0,1) -- (0,0) node[midway,left] {1};
\draw[red,fill=white] (-1,-1) circle (.2) node{$A_3$};
\draw[red,fill=white] (1,-1) circle (.2) node{$A_1$};
\draw[red,fill=white] (1,2) circle (.2) node{$A_2$};
\draw[red,fill=white] (-1,2) circle (.2) node{$A_4$};
\draw[blue,fill=white] (0,0) circle (.2) node{$B$};
\draw[blue,fill=white] (0,1) circle (.2) node{$B'$};
\end{tikzpicture}&
\begin{tikzpicture}[scale=1.2, every node/.style={scale=0.6}]
\draw (-1,-1) -- (0,0) node[midway,above] {8};
\draw (0,0) -- (1,-1) node[midway,above] {4};
\draw (1,2) -- (0,1) node[midway,above] {1};
\draw (0,1) -- (-1,2) node[midway,above] {2};
\draw (0,1) -- (0,0) node[midway,left] {5};
\draw[red,fill=white] (-1,-1) circle (.2) node{$A_3$};
\draw[red,fill=white] (1,-1) circle (.2) node{$A_1$};
\draw[red,fill=white] (1,2) circle (.2) node{$A_2$};
\draw[red,fill=white] (-1,2) circle (.2) node{$A_4$};
\draw[blue,fill=white] (0,0) circle (.2) node{$B$};
\draw[blue,fill=white] (0,1) circle (.2) node{$B'$};
\end{tikzpicture}
\\
T & T_{red}(A_1, A_2, A_3, A_4;T)
& T_{min}(A_1, A_2, A_3, A_4;T)
\end{array}\]
\caption{Left: an example of a labelled tree $T$ with $4$ distinguished vertices $A_1, A_2, A_3, A_4$. 
To help the reader we have indicated the associated branching points $B$ and $B'$.
Middle: the reduced tree with branch labels given by the corresponding path lengths.
Right: the reduced tree with branch labels corresponding
to the minimum labels in the corresponding paths in $T$.
\label{fig:reducedtree}
}
\end{figure}

We say that an unrooted tree is proper if all its internal nodes have degree exactly 3.
We will be only interested in the case where $T_{red}(A_1, \ldots, A_k;T)$ 
is a proper tree with $k$ leaves (corresponding to the distinguished vertices);
this happens with high probability for uniform random vertices $(A_1,\dots, A_k)$
in size-conditioned $Po(1)$ Galton-Watson trees, see \cite{aldous1993}.
Then $T_{red}(A_1, \ldots, A_k;T)$ has $2k-3$ branches
and one can encode it as an element of $\mathcal T_k \times \mathbb R^{2k-3}$
where $\mathcal T_k$ is the set of combinatorial unrooted proper trees
 with $k$ distinguished leaves ({\em combinatorial trees} are trees without branch-lengths);
 here, we choose an arbitrary order on the branches of each tree in $\mathcal T_k$
 in order to identify branch lengths to a vector in $\mathbb R^{2k-3}$.
 This gives a natural notion of convergence, taking the discrete topology on the finite
set $\mathcal T_k$ and the usual one on $\mathbb R^{2k-3}$.
Finally, for a geometric tree $T$, we denote by $\lambda T$ the tree where all distances 
are multiplied by $\lambda$. 

The following result is a variant of a result of Aldous
 \cite[Theorem 23 and discussion below]{aldous1993}.
As said above, Aldous considers a rooted nonplane version of $T_{red}(A_1, \ldots, A_k;T)$
but the two versions are easily seen to be equivalent considering that (i) $A_k$ is the root
and the $k-1$ other distinguished vertices as marked vertices in the rooted version;
(ii) the plane embedding is simply taken uniformly at random independently of other
data.

\begin{proposition}
\label{prop:reducedtree_branchlengths}
 For each $n\ge 1$, let $T_n$ be a $Po(1)$ Galton--Watson tree conditioned to have $n$ vertices,
 with a uniform random embedding in the plane.
 Let $(A_1, \ldots, A_k)$ be a uniform $k$-tuple of distinct vertices in $T_n$.
 Then the normalized reduced tree 
 $\frac1{\sqrt{n}} T_{red}(A_1, \ldots, A_k;T_n)$ converges in distribution to 
 $(\hat t,(x_1,\dots,x_{2k-3}))$, where $\hat t$ is a uniform random proper tree
 with $k$ distinguished leaves and $(x_1,\dots,x_{2k-3}) \in \R_+^{2k-3}$ has density
 \begin{equation}
 \label{eq:densite_lengths}
 h(x_1\dots,x_{2k-3})=\Bigg(\prod_{i=1}^{k-2} (2i-1) \Bigg)\, \Bigg(\sum_{i \le 2k-3} x_i\Bigg)\exp\Bigg[-\frac12\Bigg(\sum_{i \le 2k-3}  x_i\Bigg)^2\Bigg]
\end{equation}
and is independent from $\hat t$.
\end{proposition}

We need in fact a slightly different version of the reduced trees.
In this version, the edges of the initial tree $T$ are labelled from $1$ to $n-1$
and we label each branch of the reduced tree
 with the minimal label in the corresponding path in $T$.
We denote the resulting tree by $T_{min}(A_1, \ldots, A_k;T)$.
An example is given on Fig.~\ref{fig:reducedtree}.

\begin{proposition}
\label{prop:reducedtree_minlabels}
Fix $k \geq 0$ and let $T_n$ be a conditioned $Po(1)$-GW tree with $n$ vertices.
We label its edges uniformly at random from $1$ to $n-1$ and consider the unique plane embedding such that the resulting tree satisfies the Hurwitz condition.
We also let $(A_1, \ldots, A_k)$ be a uniform $k$-tuple of distinct vertices in $T_n$.

 Then $\frac{1}{\sqrt{n}}\left(T_{min}(A_1, \ldots, A_k;T_n)\right)$
 converges in distribution to $(\hat t,(y_1,\dots,y_{2k-3}))$, where $\hat t$ is a uniform random proper tree
 with $k$ distinguished leaves and 
 $(y_1,\dots,y_{2k-3})$ is independent from $\hat t$ and satisfies
 \begin{equation}
 \label{eq:Dist_Y}
 \P\big(y_1 \ge \alpha_1,\dots, y_{2k-3} \ge \alpha_{2k-3})
 =\int_{\mathbb R_+^{2k-3}} h(x_1\dots,x_{2k-3}) \exp(-\alpha_i x_i) \prod dx_i, 
 \end{equation}
 where $h$ is given in \cref{eq:densite_lengths} above.
\end{proposition}
We note that the plane embedding is uniform and independent from the tree structure
(but is determined from the edge-labeling).
Hence, we can apply \cref{prop:reducedtree_branchlengths},
and without branch labeling, the reduced tree converges to a uniform random
proper tree $\hat t$. Only the formula for the limiting distribution of the labels
and their independence from $\hat t$ need to be proven. 
We start with a lemma of elementary probability theory,
whose proof is left to the reader.
\begin{lemma}
\label{lem:min_labels}
Fix $\ell \ge 1$ and fix $\alpha_1, \ldots, \alpha_\ell \in \R_+$.
For $x_1, \ldots, x_\ell \in \R^*_+$, for each $n$ large enough and sizes $S_1,\dots, S_\ell$ with $S_i \sim x_i\sqrt{n}$, we take disjoint subsets $B_i \subset [n-1]$ of size $S_i-2$ uniformly at random.
Then, for any compact $\mathcal{K} \subset (\R_+^*)^\ell$, uniformly for $(x_1, \ldots, x_\ell) \in \mathcal{K}$, we have
\[\lim_{n \to \infty} \P\big(\min(B_1) \ge \alpha_1 \sqrt{n},\dots, \min(B_\ell) \ge \alpha_\ell \sqrt{n}\big)=
\prod_{i=1}^\ell \exp(- \alpha_i x_i). \]
\end{lemma}

\begin{proof}[Proof of \cref{prop:reducedtree_minlabels}]
We write $T_{red}(A_1, \ldots, A_k;T_n)=(t;S_1,\dots, S_{2k-3})$.
As said above, we can apply \cref{prop:reducedtree_branchlengths},
proving the convergence of $n^{-1/2}(t;S_1,\dots, S_{2k-3})$ 
to some vector $(\hat t,(x_1,\dots,x_{2k-3}))$.
We assume by Skorokhod theorem that this holds almost surely. 
 In particular, for $i\le 2k-3$, we have $S_i \sim x_i \sqrt{n}$.

Each branch $e_i$ ($i \le 2k-3$) of the reduced tree correspond to a path in $T_n$.
We let $B_i$ be the set of labels on this path, {\em except those of the first and last edge of the path}.
The label sets are uniform distinct random subsets of $[1,n-1]$ of size $S_i-2$,
and are independent from the plane embedding on the reduced tree
(for this independence property, 
it is important that we have not included the labels of the first and last edge of the path
when defining $B_i$).
We can therefore apply \cref{lem:min_labels} and we get that 
\begin{equation}
\label{eq:minB}
\lim_{n \to \infty} \P\big(\min(B_1) \ge \alpha_1 \sqrt{n},\dots, \min(B_{2k-3}) \ge \alpha_{2k-3} \sqrt{n}\big)=
\prod_{i=1}^k \exp(- \alpha_i x_i). \end{equation}
With high probability, the labels of the first and last edges of all those paths
are larger than $n^{2/3}$, and hence larger than $\min(B_i)$ (for all $i \le 2k-3$).
The labels of the reduced tree $T_{min}(A_1, \ldots, A_k;T_n)$ 
are therefore $(\min(B_i))_{i \le 2k-3}$.
Using formula \eqref{eq:minB} and recalling the density of the $x_i$ from Eq. \eqref{eq:densite_lengths} 
gives  \eqref{eq:Dist_Y}.
Moreover, the independence of the limiting vector $(y_1,\dots,y_{2k-3})$ and the combinatorial tree $\hat t$
follows from that of $(x_1,\dots,x_{2k-3})$ and  $\hat t$.
\end{proof}

\subsection{Connecting reduced chords diagrams of factorizations to reduced subtrees}
In what follows, we fix $g \geq 0$ and let $\bm F_n^g$ be a uniform factorization of genus $g$.
As in \cref{sec:proofs_Scaling}, we use \cref{corol:RandomGeneration}
and assume that, w.h.p., $\bm F_n^g=\Lambda^g(\rFZ)$ for a uniform random factorization $\rFZ$.
When applying $\Lambda$ for the $i$-th time,
we let $a_i< b_i< c_i$ be the random integers chosen at step i)
 and $\sigma_i$ the permutation by which we conjugate at step ii).

For fixed $M>0$, we are interested in the genus of the partial factorization $(\bm F_n^g)_{M\sqrt n}$.
With high probability, the times $v_i$ and $w_i$ at which we add transpositions
when applying $\Lambda^g$ are both larger than $M\sqrt n$ and thus we do not see
the added transpositions in this partial factorization.
For a transposition $t=(j \,h)$ and a permutation $\sigma$, we write $t^\sigma=(\sigma(j) \,\sigma(h))$.
This extends to a (partial) factorization $f_k=(t_1,\dots,t_k)$ by setting $f_k^\sigma=(t_1^\sigma,\dots,t_k^\sigma)$.
With this notation, we have
\[(\bm F_n^g)_{M\sqrt n} = (\rFZ)_{M\sqrt n}^{\, \sigma^g},\]
where we recall from \cref{ssec:genusgtogenus0} that 
$\sigma^g$ is the composition $\sigma_g \circ \dots \circ \sigma_1$.

Taking the genus, and using the reduction operations from \cref{ssec:Red_facto},
we have
\begin{equation}
\label{eq:genus_Red}
 G\big( (\bm F_n^g)_{M\sqrt n} \big)=G\big( \Red((\rFZ)_{M\sqrt n}^{\, \sigma^g}) \big). 
 \end{equation}
As in \cref{ssec:genusgtogenus0}, it will be useful to replace the permutations
$\sigma_i$ by their "approximation" $\tilde \sigma_i$, defined in \eqref{eq:tilde_sigma}.
This is done in the following lemma.
\begin{lemma}
\label{lem:Reduced_f_ftilde}
With high probability,
\[\Red((\rFZ)_{M\sqrt n}^{\, \sigma^g}\big) = \Red((\rFZ)_{M\sqrt n}^{\, \tilde \sigma^g}\big) .\]
\end{lemma}
\begin{proof}
Let us draw the points $(a_i,a'_i,a''_i,b_i,b'_i,c_i,c'_i)_{1 \le i \le g}$ on a circle 
(where $a'_i,a''_i,b'_i,c'_i$ are implicitly defined from $\sigma_i$ 
through formula \eqref{eq:sigma}).
These points cut the circle into {\em small pieces} (between $a_i$ and $a'_i$ or $a''_i$,
 or between $b_i$ and $b'_i$ or between $c_i$ and $c'_i$)
 and {\em large pieces} (other parts).
From \cref{prop:labelingalgorithm}, we know that small pieces are of size $\O_P(1)$.
In particular, with high probability, 
none of the first $M\sqrt n$ transpositions in $\rFZ$
contain an integer in one of the small pieces (we recall that $\{a_i,b_i,c_i\}$ is taken uniformly
at random independently from $\rFZ$).

On the other hand, large pieces are permuted in the same way by $\sigma^g$ and $\tilde \sigma^g$.
This implies that the partial factorizations obtained by removing isolated points
from $(\rFZ)_{M\sqrt n}^{\, \sigma^g}$ and $(\rFZ)_{M\sqrt n}^{\, \tilde \sigma^g}$ are the same.
The lemma follows immediately.
\end{proof}

To study $\Red((\rFZ)_{M\sqrt n}^{\, \tilde \sigma^g}\big)$, 
we consider the tree $T(\rFZ)$, defined as the dual of the lamination $\Sieve(\rFZ)$
(this is the tree associated by the Goulden-Yong bijection, \cref{ssec:gouldenyong},
with vertex labels).
Vertices of $T(\rFZ)$ correspond to faces of $\Sieve(\rFZ)$; each such face is adjacent to exactly
one arc $(e^{-2\pi i  (x-1)/n}, e^{-2\pi i  x/n})$, giving a correspondence between 
vertices of $T(\rFZ)$ and integers from $1$ to $n$;
see \cref{fig:DualTree_CrossingEdges}.

We recall that, for each $i \le g$, the triple $(a_i,b_i,c_i)$ is taken uniformly at random
with the condition $a_i<b_i<c_i$.
We let $A_i$ be a uniform random point in the arc
 $(e^{-2\pi i  (a_i-1)/n}, e^{-2\pi i  a_i/n})$,
 and $V(a_i)$ the corresponding vertex of $T(\rFZ)$.
 Similar notations can be defined for $b_i$ and $c_i$.
 Then, for each $i\le g$, the triple $(A_i,B_i,C_i)$
  is closed in total variation to a uniform triple of points on the circle,
 named such that $(1,A_i,B_i,C_i)$ appear clockwise in this order.
 We define $E_0=\{A_i,B_i,C_i:\, 1\le i\le g\}$ 
 and $E=\{V(A_i),V(B_i),V(C_i):\, 1\le i\le g\}$.

\begin{figure}[t]
\[\begin{tikzpicture}[scale=0.7, every node/.style={scale=0.7}, rotate=-40]
\draw (0,0) circle (3);
\foreach \i in {1,...,9}
{
\draw[auto=right] ({3.3*cos(-(\i-1)*360/9)},{3.3*sin(-(\i-1)*360/9)}) node{\i};
}
\draw[red,fill=red] (-.5,-.5) circle (.05);
\draw[red,fill=red] (0,1.5) circle (.05);
\draw[red,fill=red] (-.5,2.85) circle (.05);
\draw[red,fill=red] (1.6,-2.5) circle (.05);
\draw[red,fill=red] (1.6,2.5) circle (.05);
\draw[red,fill=red] (2.6,-.8) circle (.05);
\draw[red,fill=red] (-2.2,-2) circle (.05);
\draw[red,fill=red] (-.3,-2.9) circle (.05);
\draw[dashed,red] (-.5,-.5) -- (0,1.5) -- (-.5,2.85);
\draw[dashed,red] (1.6,-2.5) -- (2.6,-.8) --(-.5,-.5) -- (-2.2,-2) (-.5,-.5) -- (-.3,-2.9);
\draw[dashed,red] (0,1.5) -- (2.2,1.5) -- (1.6,2.5);

\draw[fill=black!40!green] (2.2,1.5) circle (.15);
\draw[very thick,black!40!green] (3,0) arc (0:360/9:3cm);
\draw[fill=black!40!blue] (-.5,-.5) circle (.15);
\draw[very thick,black!40!blue] ({3*cos(4*360/9)},{3*sin(4*360/9)}) arc (360*4/9:360*5/9:3cm);

\draw ({3*cos(-2*360/9)},{3*sin(-2*360/9)}) -- ({3*cos(-3*360/9)},{3*sin(-3*360/9)}) node[circle,midway,fill=blue!20, inner sep=2pt]{4};
\draw ({3*cos(2*360/9)},{3*sin(2*360/9)}) -- ({3*cos(1*360/9)},{3*sin(1*360/9)}) node[circle,midway,fill=blue!20, inner sep=2pt]{3};
\draw ({3*cos(-3*360/9)},{3*sin(-3*360/9)}) -- ({3*cos(-4*360/9)},{3*sin(-4*360/9)}) node[circle,midway,fill=blue!20, inner sep=2pt]{2};
\draw ({3*cos(-2*360/9)},{3*sin(-2*360/9)}) -- ({3*cos(0*360/9)},{3*sin(0*360/9)}) node[circle,midway,fill=blue!20, inner sep=2pt]{5};
\draw ({3*cos(-5*360/9)},{3*sin(-5*360/9)}) -- ({3*cos(0*360/9)},{3*sin(0*360/9)}) node[circle,midway,fill=blue!20, inner sep=2pt]{6};
\draw ({3*cos(-7*360/9)},{3*sin(-7*360/9)}) -- ({3*cos(0*360/9)},{3*sin(0*360/9)}) node[circle,midway,fill=blue!20, inner sep=2pt]{7};
\draw ({3*cos(-2*360/9)},{3*sin(-2*360/9)}) -- ({3*cos(-1*360/9)},{3*sin(-1*360/9)}) node[circle,midway,fill=blue!20, inner sep=2pt]{8};
\draw ({3*cos(-6*360/9)},{3*sin(-6*360/9)}) -- ({3*cos(-7*360/9)},{3*sin(-7*360/9)}) node[circle,midway,fill=blue!20, inner sep=2pt]{9};
\end{tikzpicture}
\]
\caption{The lamination associated to the minimal factorization
(45)(89)(34)(13)(16)(18)(23)(78) (in black) and its dual tree (in red).
We highlight in blue and green two vertices of this dual tree,
and the corresponding arcs of circles.
We further discuss this example to illustrate the proof of \cref{lem:crossing_transpo}.
The blue and green arcs of circle delimit two regions of the circle
 (one from 1 to 5 and the other from 6 to 9).
 There are two  black chords going from one region to the other,
(carrying labels 6 and 7).
 These two chords are dual to the two red edges 
on the path between the blue and green vertices.}
\label{fig:DualTree_CrossingEdges}
\end{figure}
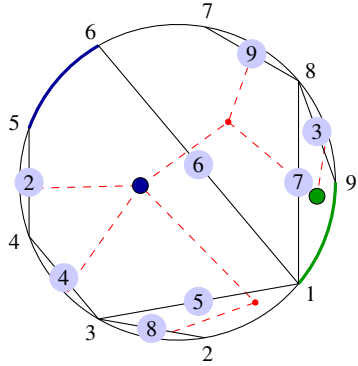

On the other hand, edges in $T(\rFZ)$ correspond to edges of $\Sieve(\rFZ)$
i.e. to transpositions in $\rFZ$.
From now on, we say that two transpositions $\transpo$ and $\transpo'$ in a partial factorization $f$
cross if the corresponding edges cross in $M(f)$.
For example if $a<b<c<d$, then the transpositions $(a\, c)$ and $(b\, d)$ cross
(in any partial factorization $f$ where they both appear).
Also, $(a\, b)$ and $(a\, c)$ cross if and only if $(a\, c)$ appear before $(a\, b)$
in the underlying factorization 
(and similar statements hold, permuting cyclically the indices $a$, $b$ and $c$).

We have the following key lemma.
\begin{lemma}
\label{lem:crossing_transpo}
A transposition may cross others
in the partial factorization $(\rFZ)_{M\sqrt n}^{\, \tilde \sigma^g}$ 
only if its corresponding edge $T(\rFZ)$ belongs to a path between two points of $E$.
\end{lemma}
In the sequel, we will call such transpositions {\em potentially crossing transpositions}.
\begin{proof}
As usual, we split the circle into pieces, which are the connected components
of $\mathbb S^1 \backslash E_0$. 
Suppose that $t$ is a transposition of $\rFZ$,
whose image in $(\rFZ)_{M\sqrt n}^{\, \tilde \sigma^g}$ crosses another transposition,.
Since transpositions in $\rFZ$ are not crossing each other
and since connected components of  $\mathbb S^1 \backslash E_0$
are simply translated by $\mathbb S^1 \backslash E_0$,
the transposition $t$ must
 connect two different connected components of $\mathbb S^1 \backslash E_0$.
 In particular there are two points $P$ and $Q$ in $E_0$ such that $t$ connects the
 two arcs of circles delimited by $P$ and $Q$.
By duality, this implies that
the edge corresponding to $t$ in $T(\rFZ)$ lies on the path between the vertices $V(P)$ and $V(Q)$;
see \cref{fig:DualTree_CrossingEdges}.
\end{proof}

We recall from the previous section that for a tree $T$ and subset $E$
of the vertices, we denote by $T_{min}(E;T)$ the reduced tree $T_{red}(E;T)$
where each branch of  $T_{min}(E;T)$ carries a label indicating the minimum label
on the corresponding path in $T$.

Consider a potentially crossing transposition $t$ in $\rFZ$.
From \cref{lem:crossing_transpo}, the corresponding edge in $T(\rFZ)$
-- call it $e$ -- lies on a path between two points of $E$.
By construction of the reduced tree,
 $e$ is mapped to a branch in the reduced tree $T_{min}(E;T(\rFZ))$;
we denote by $\pi$ this projection from some edges of $T$ 
to branches of $T_{min}(E;T(\rFZ))$.
An important observation is the following:
if we know $\pi(e)$, then we know which pairs of points $P$,$Q$
in $E_0$ are separated by the transposition $t$,
i.e. we know which connected components of $\mathbb S^1 \backslash E_0$ are linked by $t$.

With this in mind, we can prove the following proposition.
\begin{proposition}
\label{prop:ReducedFacto_ReducedTree}
The partial factorization 
$\Red((\rFZ)_{M\sqrt n}^{\, \tilde \sigma^g}\big)$
only depends on the set of branches of labels at most $M\sqrt{n}$ in $T_{min}(E;T(\rFZ))$.
\end{proposition}
\begin{proof}
We consider a branch $e_0$ in the reduced tree $T_{min}(E;T(\rFZ))$ with label $\le M\sqrt n$.
This means that there exists an edge $e$ with $\pi(e)=e_0$ and label $\le M\sqrt n$.
Taking the dual transposition of $e$,
we see that there is at least one transposition in $(\rFZ)_{M\sqrt n}^{\, \tilde \sigma^g}$
crossing $e_0$.

Moreover, if two edges $e$ and $e'$ are mapped to the same branch $\pi(e)=\pi(e')$,
they connect the same pair of connected components of $\mathbb S^1 \backslash E_0$.
In particular, since the sieve of $(\rFZ)_{M\sqrt n}$ is noncrossing,
after applying the composition $\tilde \sigma^g$ of rotations,
the corresponding transpositions are parallel transpositions
in $(\rFZ)_{M\sqrt n}^{\, \tilde \sigma^g}$ (parallel in the sense of \cref{ssec:Red_facto}).
Hence, one removes one of them in the reduction operation.
We conclude that the partial factorization $\Red((\rFZ)_{M\sqrt n}^{\, \tilde \sigma^g}\big)$
contains exactly one transposition crossing $e_0$.

On the other hand, if $e_1$ is a branch of $T_{min}(E;T(\rFZ))$ with label larger than $M\sqrt n$,
then all edges $e$ of $T$ with $\pi(e)=e_1$ have label larger than $M\sqrt n$.
The corresponding transpositions do not belong to the partial factorization $(\rFZ)_{M\sqrt n}$.

Therefore, the reduced tree $T_{min}(E;T(\rFZ))$ -- and more precisely,
its set of branches with label smaller than $M\sqrt n$ -- determines which
transpositions of $\rFZ$ appear in the partial factorization
$\Red((\rFZ)_{M\sqrt n}^{\, \tilde \sigma^g}\big)$, proving the lemma.
\end{proof}

\subsection{Limit of the genus process: end of the proof}

We can now finish the proof of Theorem \ref{thm:asymptoticgenus}.

\begin{proof}[Proof of Theorem \ref{thm:asymptoticgenus}]
The tree $T(\rFZ)$ is a $Po(1)$ GW tree conditioned on having $n$ vertices,
with a uniform labeling of its edges, and embedded in the plane so that it satisfies
Hurwitz condition -- this is a reformulation of \cref{thm:GouldenYong}.
Besides, $E$ is the union of $g$ uniform triples of vertices in the tree;
up to a set of small probability, this corresponds to a uniform set of $3g$ distinct vertices in the tree.
Therefore, we can apply \cref{prop:reducedtree_minlabels} and we know that
$\frac{1}{\sqrt n} T_{min}(E;T(\rFZ))$ converges in distribution.
Combining with \cref{eq:genus_Red,lem:Reduced_f_ftilde,prop:ReducedFacto_ReducedTree},
this proves the convergence of the process $\big( G \big((\bm F_n^g)_{c\sqrt n} \big)\big)_{c\ge 0}$.

In the limit $c \to 0$, 
the limiting tree of $\frac{1}{\sqrt n} T_{min}(E;T(\rFZ))$
has no branch with label at most $c$ with high probability.
Hence the limit as $n \to \infty$ of the chord diagram 
$\Red((\rFZ)_{c \sqrt n}^{\, \tilde \sigma^g}\big)$ is an empty diagram
w.h.p. as $c \to 0$.
Its genus is $0$, proving $\lim_{c \to 0} G^g_c=0$.

For the limit $c \to +\infty$, the limiting tree of $\frac{1}{\sqrt n} T_{min}(E;T(\rFZ))$
has all its branches with label at most $c$.
Then in $(\rFZ)_{c\sqrt n}$, there is at least one transposition dual
to each branch of the reduced tree $T_{min}(E;T(\rFZ))$
(this tree has  $2k-3=6g-3$ branches).
Taking one such transposition for each branch, 
one can check that after rotation by $\, \tilde \sigma^g$,
there are never parallel nor noncrossing transpositions.
This implies that the  reduced chord diagram 
$\Red((\rFZ)_{c \sqrt n}^{\, \tilde \sigma^g}\big)$
has $12g-6$ vertices with probability tending to $1$ as $c$ tends to $+\infty$.
From \cref{lem:size_reduced_diagram}, we deduce that its genus is at least $g$
 (w.h.p. as $c \to +\infty$).
But by construction, it cannot be larger than $g$;
indeed $G^g_c$ is the limit of genera of prefixes of a genus $g$ factorization.
This proves $\lim_{c \to +\infty} G^g_c=g$.

Finally, with probability $1$, the labels on all branches of the limiting tree of $\frac{1}{\sqrt{n}} T_{min}(E;T(\rFZ))$ are all different, which implies that the jumps of $(G^g_c)_{c \ge 0}$ have size $1$.
\end{proof}

To conclude, we give an explicit formula in the case $g=1$.
\begin{proposition}
\label{prop:Genus1}
We have 
\[ \P (G^1_c=0) = 
\int_{\mathbb R_+^{3}} h(x_1,x_2,x_3) \big( 3\exp(-c (x_1+x_2)) - 2\exp(-c (x_1+x_2+x_3) \big) \prod dx_i,\]
where $h$ is given in \cref{eq:densite_lengths} above.
\end{proposition}
\begin{proof}
In genus $g=1$, the reduced tree of the three points $A$, $B$ and $C$
in $T(\rFZ)$ is always a star with three extremities.
The partial factorization $(\rFZ)_{c\sqrt n}^{\, \tilde \sigma}$ has genus $1$
as soon as there is a label smaller than $c\sqrt{n}$ on two of the three branches of this star. 
Indeed, in this case, the transpositions $t_j$ and $t_{j'}$ corresponding to these edges
with label smaller than $c\sqrt{n}$ connect different connected components of 
$\mathbb S \backslash \{A,B,C\}$ and cross after rotation.

On the other hand, if only one branch has one (or several) label smaller than $c\sqrt{n}$,
then all potentially noncrossing transpositions in $(\rFZ)_{c\sqrt n}^{\, \tilde \sigma}$
correspond to the same branch of the reduced tree $T_{min}(A,B,C;T(\rFZ))$
and are therefore parallel. The genus is $0$ in this case.

Denoting by $\ell_1$, $\ell_2$ and $\ell_3$ 
the labels in $T_{min}(A,B,C;T(\rFZ))$, the argument above shows that
\[\P \Big( G\big( (\rFZ)_{c\sqrt n}^{\, \tilde \sigma} \big) =0 \Big)
= \P \big[ |\{ i\le 3: \ell_i \ge c\sqrt n \}| \ge 2 \big].\]
Taking the limit $n\to \infty$ and using \cref{prop:reducedtree_minlabels},
we get that
\[ \P (G^g_c=0) = \P \big[ |\{ i \le 3 : y_i \ge c \}| \ge 2 \big], \]
where the distribution of the $y_i$ is given by \eqref{eq:Dist_Y}.
The proposition follows from an elementary inclusion-exclusion argument.
\end{proof}

\bibliographystyle{bibli_perso}
\bibliography{bibli}

\end{document}